\documentclass{amsart}

\usepackage{Preset}
\usepackage[makeroom]{cancel}
\usepackage{MnSymbol}
\usepackage{nccmath}

\usepackage{mathrsfs}
\usepackage[latin1]{inputenc}
\usepackage[T1]{fontenc}
\usepackage{tikz}

\usetikzlibrary{decorations.pathmorphing}
\usetikzlibrary{shapes,arrows,graphs}
\usepackage[colorlinks=true,linkcolor=blue!50!black,anchorcolor=red,citecolor=blue,filecolor=black,menucolor=black,runcolor=black,urlcolor=black]{hyperref}

\DeclareRobustCommand{\butterfly}{\textnormal{\reflectbox{B}}\!\!\textnormal{B}}

\newcommand*\circled[1]{\tikz[baseline=(char.base)]{\node[shape=circle,draw,inner sep=0.5pt] (char) {#1};}}

\title[Rigidity]{Rigidity of J-rotational rational maps and critical quasicircle maps}

\author{Willie Rush Lim}
\address{Dept. of Mathematics, Brown University, Providence, Rhode Island 02912}
\email{willie\_rush\_lim@brown.edu}

\date{}

\begin{document}

\begin{abstract}
We present a number of rigidity results concerning holomorphic dynamical systems admitting rotation quasicircles. Firstly, we show the absence of line fields on the Julia set of any rational map that is geometrically finite away from a number of rotation quasicircles with bounded type rotation number. As an application, we prove combinatorial rigidity associated to the problem of degeneration of Herman rings of the simplest configuration. Secondly, we extend a result of de Faria and de Melo on the $C^{1+\alpha}$ rigidity of critical circle maps with bounded type rotation number to a larger class of dynamical objects, namely \emph{critical quasicircle maps}. Unlike critical circle maps, critical quasicircle maps may have imbalanced inner and outer criticalities. As a consequence, we prove dynamical universality and exponential convergence of renormalization towards a horseshoe attractor.
\end{abstract}

\maketitle

\setcounter{tocdepth}{1}
\tableofcontents

\section{Introduction}
\label{sec:intro}

\subsection{J-rotational rational maps}

The Fatou set $F(f)$ of a rational map $f: \RS \to \RS$ of degree at least two is defined as the set of points in the Riemann sphere $\RS$ around which the set of iterates of $f$ is equicontinuous, whereas the Julia set $J(f)$ is the complement $\RS \backslash F(f)$. A \emph{rotation domain} $U$ of a rational map $f$ is a periodic component of the Fatou set on which an iterate $f^k$ of $f$ is conjugate to an irrational rotation. It can be either simply connected, in which case $U$ is called a \emph{Siegel disk}, or doubly connected, in which case $U$ is called a \emph{Herman ring}. 

Siegel disks are always centered around a unique periodic point. Examples of Siegel disks arise from the study of analytic linearization near a neutral fixed point, which has essentially received a complete treatment by the works of Brjuno, Herman, Yoccoz, and Perez-Marco. On the other hand, the construction of the first examples of Herman rings was based on the study of linearizability of analytic circle diffeomorphisms by Arnol'd and Herman. A more general construction was later established by Shishikura \cite{S87} via quasiconformal surgery, which allows us to construct Herman rings out of two Siegel disks, and to convert Herman rings into Siegel disks. 

The main focus of this paper is the occurrence of rotational dynamics supported on a single curve. Recall that a Jordan curve in $\RS$ is called a quasicircle if it is the image of a Euclidean circle under a quasiconformal map on $\RS$.

\begin{definition}
    We say that a Jordan curve/quasicircle $X \subset \RS$ is a \emph{rotation curve/quasicircle} of a rational map $f$ if $X$ is invariant under some iterate $f^k$ and $f^k|_X$ is conjugate to an irrational rotation. Additionally, we say that $X$ is a \emph{Herman curve/quasicircle} $X$ of $f$ if it is not contained in the closure of a rotation domain.
\end{definition}

Naturally, rotation domains are foliated by analytic rotation curves. Recall that every irrational number $\theta \in (0,1)$ admits a unique continued fraction expansion
\[
\theta = [0;a_1,a_2,a_3,\ldots] := \cfrac{1}{a_1 + \frac{1}{a_2 + \frac{1}{a_3 + \ldots}}}.
\]
The number $\theta$ is \emph{of bounded type} if there is a well-defined maximum 
\[
\beta(\theta) := \max_{n\geq 1} a_n < \infty.
\]
When the rotation number is of bounded type, Zhang \cite{Z11} proved that the boundary of any Siegel disk of a rational map is a rotation quasicircle containing a critical point. An analogous result holds for Herman rings via a straightforward application of Shishikura's surgery. More recently, examples of Herman quasicircles of arbitrary combinatorics are constructed in \cite{Lim23}. 

In the first part of this paper, we will study a class of rational maps whose dynamics is dominated by rotation quasicircles.

\begin{definition}
    We call a rational map $f: \RS \to \RS$ \emph{J-rotational} if $f$ has a rotation curve and every critical point in the Julia set is either preperiodic or lies in the grand orbit of a rotation quasicircle. Given a J-rotational rational map $f$, we also say that $f$ is \emph{of bounded type} if all of its rotation curves have bounded type rotation number.
\end{definition}

The \emph{postcritical set} $P(f)$ of a rational map $f$ is the closure of the forward orbit of critical points:
\[
    P(f) := \overline{  \bigcup_{f'(c)=0} \bigcup_{n=1}^\infty f^n(c) }.
\]
A rational map is called \emph{geometrically finite} if every critical point on the Julia set has finite orbit, or equivalently, the intersection $P(f) \cap J(f)$ is finite. The definition of bounded type J-rotationality is equivalent to the condition that $f$ is geometrically finite away from a non-empty union of bounded type rotation quasicircles.

A rational map $f: \RS \to \RS$ admits an \emph{invariant line field} on its Julia set $J(f)$ if there is a measurable Beltrami differential $\mu(z) \frac{d \bar{z}}{dz}$ such that $f^*\mu=\mu$ almost everywhere, $|\mu|=1$ on a positive measure subset of $J(f)$, and $\mu=0$ elsewhere. The absence of line fields implies the lack of non-trivial deformation space supported on the Julia set. In \cite[\S3]{McM94} and \cite[\S9]{McS98}, McMullen and Sullivan conjectured that \emph{flexible Latt\'es maps} are the only rational maps that admit invariant line fields. This conjecture implies the \emph{density of hyperbolicity} conjecture, one of the central problems in one dimensional dynamics. 

When $f$ is hyperbolic, i.e. when $P(f) \cap J(f)$ is empty, Sullivan \cite{S83} proved that the Hausdorff dimension of the Julia set of $f$ is less than two and consequently $f$ does not admit invariant line fields on its Julia set. The same result was proven by McMullen \cite{McM00} to hold for a larger class of rational maps, namely geometrically finite ones. In this paper, we follow the ideas behind McMullen's study of the critical behavior of quadratic Siegel disks in \cite{McM98} and prove the following.

\begin{thmx}
\label{thm:NILF-main}
    Every rational map $f$ that is J-rotational of bounded type does not carry any invariant line field on its Julia set $J(f)$. Additionally, if $f$ admits no Herman curves, then $J(f)$ has zero Lebesgue measure.
\end{thmx}

The second statement applies to rational maps whose dynamics are dominated by bounded type Siegel disks and Herman rings. When only a single rotation domain is present, Theorem \ref{thm:NILF-main} is not new. Using puzzle techniques, Petersen \cite{Pe96} originally proved that the Julia set of a quadratic map that has a single bounded type Siegel disk has zero Lebesgue measure. McMullen \cite{McM98} strengthened this result by proving porosity of the Julia set (see Definition \ref{def:porosity}). Zhang \cite{Z08} and Wang \cite{W12} also showed that the Julia set of rational maps that are postcritically finite outside of a single bounded type rotation domain has zero Lebesgue measure. All of these approaches hinge on the property that locally, the Julia set lies only on one side of any boundary component of a rotation domain, which is also utilized in the proof of Theorem \ref{thm:NILF-main}. However, this property does not hold in the setting of Herman curves.

Aside from the aforementioned previous works, the absence of line fields has also been proven for other classes of rational maps. These include
\begin{itemize}
    \item non-recurrent and Collet-Eckmann rational maps \cite{PU01}, 
    \item weakly hyperbolic rational maps \cite{Ha01}, 
    \item some real rational maps \cite{LvS00,She03}, 
    \item rational maps satisfying summability conditions \cite{Ma05,GS09},
    \item rational maps with Cantor Julia sets \cite{YZ10}, 
    \item some classes of infinitely renormalizable unicritical polynomials \cite{McM94, Ya95, Lyu97, K06, KL08, KL09a, Ch10, CS15, Ad16, DL23a, DL23b}, 
    \item only finitely renormalizable unicritical polynomials \cite{Hub93,K98,KvS09,AKLS} and Newton maps \cite{DS22,RYZ}, etc.
\end{itemize}

\subsection{Herman curves of the simplest configuration}

Denote by $\rat_{d}$ the space of degree $d$ rational maps, equipped with the topology of uniform convergence on compact subsets. We will discuss the implication of Theorem \ref{thm:NILF-main} on the following particular class of J-rotational rational maps.

\begin{definition}
\label{main-definition}
    For a pair of integers $d_0, d_\infty \geq 2$ and an irrational number $\theta \in (0,1)$, we define $\HQspace_{d_0,d_\infty,\theta}$ to be the space of rational maps $f \in \rat_{d_0+d_\infty-1}$ such that 
\begin{enumerate}[label=\textnormal{(\Roman*)}]
    \item\label{defn1} the only non-repelling periodic points of $f$ are superattracting fixed points $0$ and $\infty$ at which $f$ has local degrees $d_0$ and $d_\infty$ respectively;
    \item\label{defn2} the map $f$ has an invariant Herman quasicircle $\Hq$ with rotation number $\theta$;
    \item\label{defn3} $\Hq$ separates $0$ and $\infty$;
    \item\label{defn4} every critical point of $f$ aside from $0$ and $\infty$ is contained in $\Hq$.
\end{enumerate} 
\end{definition}

Unless otherwise stated, we assume that $\theta$ is of bounded type. The space $\HQspace_{d_0,d_\infty,\theta}$ consists of rational maps admitting Herman quasicircles of the simplest configuration. Every map in $\HQspace_{d_0,d_\infty,\theta}$ is a conformal welding of two polynomials of degrees $d_0$ and $d_\infty$ admitting invariant Siegel disks satisfying a condition similar to \ref{defn4}. Refer to Figure \ref{fig:cqc-comparison} for some examples.

The combinatorics of a map $f$ in $\HQspace_{d_0,d_\infty,\theta}$ is encoded by the relative position of critical points along the Herman quasicircle of $f$ and their inner and outer local degrees. All admissible combinatorial data can be identified with points in the space $\mathcal{C}_{d_0,d_\infty}$, a compact connected real orbifold of dimension $d_0+d_\infty-3$. (See Definition \ref{combi}.) 

% Anonymous version
Consider the space $\Herspace_{d_0,d_\infty,\theta}$ of degree $d_0+d_\infty-1$ rational maps $f$ satisfying the same conditions \ref{defn1}-\ref{defn4} except that $\Hq$ is replaced with the closure of a genuine Herman ring. In \cite{Lim23}, \emph{a priori bounds} for Herman rings in $\Herspace_{d_0,d_\infty,\theta}$ was proven via an elaborate analysis of near-degenerate surfaces. Such \emph{a priori bounds} shed light on the limits of degenerating Herman rings and give examples of rational maps in $\HQspace_{d_0,d_\infty,\theta}$ of arbitrary combinatorics.

% Real Version
% Consider the space $\Herspace_{d_0,d_\infty,\theta}$ of degree $d_0+d_\infty-1$ rational maps $f$ satisfying the same conditions \ref{defn1}-\ref{defn4} except that $\Hq$ is replaced with the closure of a genuine Herman ring. In \cite{Lim23}, we obtained \emph{a priori bounds} for Herman rings in $\Herspace_{d_0,d_\infty,\theta}$ via an elaborate analysis of near-degenerate surfaces. Such \emph{a priori bounds} allows us to understand the limits of degenerating Herman rings and consequently construct examples of rational maps in $\HQspace_{d_0,d_\infty,\theta}$ of arbitrary combinatorics.

\begin{theorem}[{\cite[Theorem C]{Lim23}}]
\label{APB02}
    The accumulation space 
    \[
    \Herspace_{d_0,d_\infty,\theta}^\partial := \overline{\Herspace_{d_0,d_\infty,\theta}} \backslash \Herspace_{d_0,d_\infty,\theta}
    \]
    is contained in $\HQspace_{d_0,d_\infty,\theta}$ and the map sending $f \in \Herspace^\partial_{d_0,d_\infty,\theta}$ to its associated combinatorics $\comb(f) \in \mathcal{C}_{d_0,d_\infty}$ is surjective. 
\end{theorem}

We denote by $f\sim g$ when two rational maps $f$ and $g$ are conjugate by some linear map $z \mapsto \lambda z$, and by $[f]$ the linear conjugacy class of $f$. If $f \in \HQspace_{d_0,d_\infty,\theta}$, let $\rot(f) = \theta$ denote the rotation number of $f$ along its Herman quasicircle. Both $\comb(f)$ and $\rot(f)$ are invariant under linear conjugation. 

Denote by $\Theta_N$ the set of bounded type irrational numbers $\theta \in (0,1)$ with bound $\beta(\theta)\leq N$; this is a singleton consisting of the golden mean if $N=1$, and a Cantor set if $N \geq 2$. One application of Theorem \ref{thm:NILF-main} is the following theorem.

\begin{thmx}[Combinatorial rigidity]
\label{thm:combinatorial-rigidity}
    If two maps in $\HQspace_{d_0,d_\infty,\theta}$ have the same combinatorics, then they are conformally conjugate. Moreover, for all $N \geq 1$, the map \[
    \bigcup_{\theta \in \Theta_N} \HQspace_{d_0,d_\infty,\theta}/_\sim \to \mathcal{C}_{d_0,d_\infty} \times \Theta_N, \quad [f] \mapsto (\comb(f),\rot(f))
    \]
    is a homeomorphism.
\end{thmx}

In particular, each $\HQspace_{d_0,d_\infty,\theta}/_\sim$ is a compact connected topological orbifold of dimension $d_0+d_\infty-3$. By combining Theorem \ref{APB02} and Theorem \ref{thm:combinatorial-rigidity}, we strengthen the relation between $\HQspace_{d_0, d_\infty,\theta}$ and $\Herspace_{d_0, d_\infty,\theta}$.

\begin{corx}
\label{main-corollary}
    Every rational map in $\HQspace_{d_0, d_\infty,\theta}$ arises as a limit of degenerating Herman rings. More precisely, $
    \HQspace_{d_0, d_\infty,\theta} = \Herspace_{d_0,d_\infty,\theta}^\partial$.
\end{corx}

We also discuss further consequences of rigidity on rational maps admitting trivial Herman curves (those which are Euclidean circles) in Section \S\ref{ss:blaschke}, as well as antipode-preserving cubic rational maps in Section \S\ref{ss:antipode}.

Given a degree $d\geq 2$ rational map $f$ containing an invariant bounded type Herman quasicircle $\Hq$, one can perform Douady-Ghys surgery \cite{G84, D87} to both sides of $\Hq$ and obtain a pair of rational maps $g_+$ and $g_-$ having invariant Siegel disks $Z_+$ and $Z_-$ of complementary rotation numbers. Applying Shishikura's surgery \cite{S87} to $g_+$ and $g_-$, we obtain a family of degree $d$ rational maps $F_t$ admitting an invariant Herman ring $\He_t$ of modulus $t>0$ and of the same rotation number and combinatorics as $f|_\Hq$. The dynamics of $F_t$ on each component of $\RS \backslash \overline{\He_t}$ is quasiconformally conjugate to the dynamics of $f$ on a component of $\RS \backslash \Hq$. In light of Corollary \ref{main-corollary}, we believe in the following conjecture.

\begin{conx}
    Every Herman quasicircle with bounded type rotation number arises as a limit of degenerating Herman rings. More precisely, given $f$ and $F_t$ above, $[F_t] \to [f]$ as $t \to 0$ in the moduli space $\rat_d/\textnormal{PSL}_2(\C)$.
\end{conx} 

Note that the bounded type assumption is essential in the realization and rigidity of maps in $\HQspace_{d_0,d_\infty,\theta}$. Recently, Yang \cite{Y22} proved the existence of a cubic rational map whose Julia set has positive Lebesgue measure and contains a smooth Herman curve of high type Brjuno rotation number. Such a rational map is also constructed as a limit of degenerating Herman rings, but the problem of realization and rigidity for general irrational rotation number $\theta$ and degrees $d_0,d_\infty$ remains open.

In light of Theorem \ref{thm:NILF-main}, it is also reasonable to ask the following question.

\begin{quex}
\label{qn:leb-meas}
    Given a map $f \in \HQspace_{d_0,d_\infty,\theta}$ where $\theta$ is of bounded type, does $J(f)$ have zero Lebesgue measure? Is the Hausdorff dimension of $J(f)$ less than $2$? 
\end{quex}

We believe that Question \ref{qn:leb-meas} is a much more difficult problem. Points along the Herman quasicircle of $f$ are deep points of its Julia set (cf. Theorem \ref{deep-point-theorem}). This hints at the similarity in complexity to Feigenbaum Julia sets (see \cite{McM96}). One possible direction towards this problem is an adaptation of the methods developed by Avila and Lyubich \cite{AL08} in studying the Lebesgue measure of Feigenbaum Julia sets. In particular, it may be possible to formulate a criterion for zero or positive area in terms of escape probabilities and, similar to \cite{DS20,AL22,DL23a}, apply either rigorous computer estimates or various renormalization schemes to obtain a conclusive answer.

\subsection{Critical quasicircle maps}

The last part of this paper goes beyond the realm of rational dynamics and studies rotation quasicircles of general holomorphic maps. 

\begin{definition}
    A (\emph{holomorphic}) \emph{quasicircle map} is a homeomorphism $f: \Hq \to \Hq$ of a quasicircle admitting a holomorphic extension on a neighborhood of $\Hq$. Additionally, we say that $f: \Hq \to \Hq$ is a (\emph{uni}-)\emph{critical} quasicircle map if $\Hq$ contains exactly one critical point of $f$.
\end{definition}

The behaviour at the unique critical point on $\Hq$ can be encoded by two positive integers, namely the inner criticality $d_0$ and the outer criticality $d_\infty$. The total local degree of $f$ at the critical point is $d_0+d_\infty-1$ and it must be at least $2$. When the criticalities are specified, we call $f: \Hq \to \Hq$ a \emph{$(d_0,d_\infty)$-critical} quasicircle map. Unicritical rational maps in $\HQspace_{d_0,d_\infty,\theta}$ serve as model examples of $(d_0,d_\infty)$-critical quasicircle maps. See Figure \ref{fig:cqc-comparison} for some examples.

We say that a critical quasicircle map $f: \Hq \to \Hq$ is \emph{critical circle map} if $\Hq$ is trivially a Euclidean circle and so the inner and outer criticalities coincide. Critical circle maps provide one of the two classical examples of one-dimensional dynamical systems exhibiting remarkable universality phenomena, the other one being Feigenbaum universality observed in unimodal maps. 

\begin{figure}
    \centering
    \includegraphics[width=\linewidth]{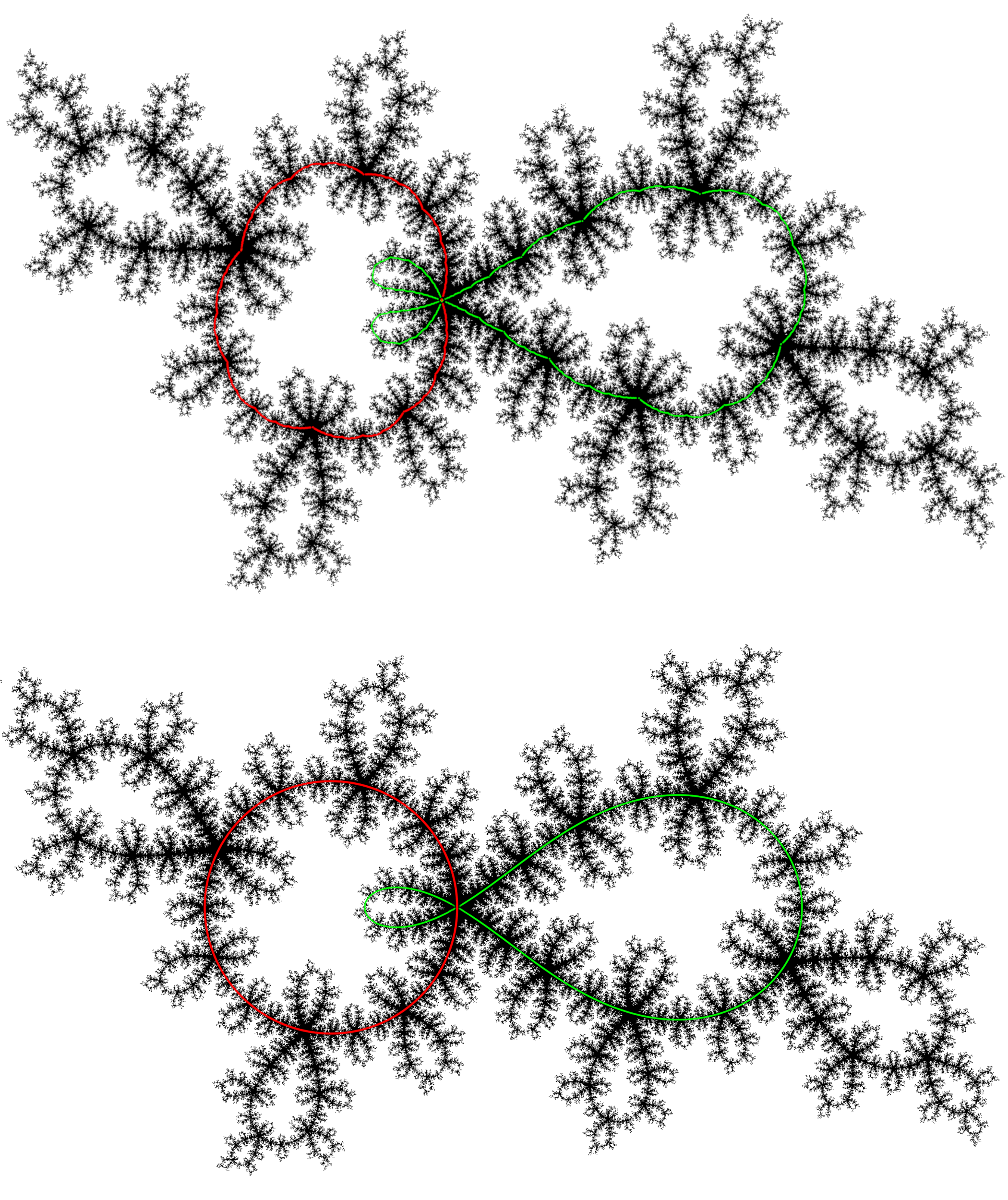}
    \caption{The Julia sets of
    \leavevmode\\
    \vspace{-0.1in}
    \\
    \begin{minipage}{\linewidth} 
    \begin{align*}
        f_{3,2}(z) = bz^3\cfrac{4-z}{1-4z+6z^2} \quad \text{and} \quad f_{2,2}(z) = cz^2 \cfrac{z-3}{1-3z}.
    \end{align*}
    The critical values $b\approx -1.144208-0.964454i$ and $c \approx -0.755700-0.654917i$ are picked such that $f_{3,2}:\Hq \to \Hq$ is an imbalanced critical quasicircle map for some quasicircle $\Hq\subset \C$, $f_{2,2}: \T \to \T$ is a critical circle map, and both have the golden mean rotation number $\theta$. The rational map $f_{3,2}$ is in $\HQspace_{3,2,\theta}$, whereas $f_{2,2}$ is in $\HQspace_{2,2,\theta}$. Both $\Hq$ and $\T$ are colored red, and their preimages are colored green. Refer to Proposition \ref{prop:prototype-example} for details.
  \end{minipage}
    }
    \label{fig:cqc-comparison}
\end{figure}

Like critical circle maps, critical quasicircle maps also exhibit quasisymmetric rigidity. Petersen \cite{Pe04} proved a generalization of Herman-\'Swiatek theorem, which states that a critical quasicircle map is quasisymmetrically conjugate to irrational rotation if and only if its rotation number is of bounded type. Here, we will further extend known rigidity results for critical circle maps to critical quasicircle maps.

Given a constant $\alpha>0$, we say that a map $\phi$ is \emph{uniformly $C^{1+\alpha}$-conformal} on a set $S \subset \C$ if there are constants $C,\varepsilon>0$ such that for every point $z$ in $S$, the complex derivative $\phi'(z)$ at $z$ exists and for $|t|<\varepsilon$,
    \begin{equation}
        \left| \frac{\phi(z+t)-\phi(z)}{t} - \phi'(z) \right| \leq C |t|^\alpha.
    \end{equation}

\begin{thmx}[$C^{1+\alpha}$ rigidity]
\label{thm:c1plusalpharigidity}
    Any two $(d_0,d_\infty)$-critical quasicircle maps $f_1: \Hq_1 \to \Hq_1$ and $f_2 : \Hq_2 \to \Hq_2$ with the same bounded type rotation number are quasiconformally conjugate on the neighborhood of $\Hq_1$ and $\Hq_2$. Moreover, there is some $\alpha>0$ such that the conjugacy is uniformly $C^{1+\alpha}$-conformal on $\Hq_1$.
\end{thmx}

The main strategy in the proof of Theorem \ref{thm:c1plusalpharigidity} consists of complex bounds and an adaptation of McMullen's recipe \cite{McM96}, namely uniform twisting and deep points, which was originally applied in the context of Feigenbaum Julia sets. This strategy has been successfully applied in the study of rigidity of critical circle maps \cite{dFdM2} as well as multicritical circle maps \cite{GY21}. Let us list a number of important applications.

The presence of critical points generally destroys the smoothness of rotation curves. One consequence of Theorem \ref{thm:c1plusalpharigidity} is that the corresponding quasicircle cannot be smooth except when it admits symmetric combinatorics, or equivalently, when the map is quasiconformally conjugate to a critical circle map.

\begin{corx}[Smoothness]
\label{cor:smoothness}
    Given a $(d_0,d_\infty)$-critical quasicircle map $f: \Hq \to \Hq$ with bounded type rotation number, the following are equivalent.
    \begin{itemize}
        \item $\Hq$ is $C^1$ smooth;
        \item the Hausdorff dimension of $\Hq$ is one;
        \item $d_0=d_\infty$.
    \end{itemize}
\end{corx}

See Figure \ref{fig:smoothness} for an example of a $C^1$ smooth Herman curve that is not a round circle. Under combinatorial asymmetry, the dimension is also universal.

\begin{corx}[Universality of dimension]
\label{cor:dimension}
    If two $(d_0,d_\infty)$-critical quasicircle maps $f_1 : \Hq_1 \to \Hq_1$ and $f_2: \Hq_2 \to \Hq_2$ have the same bounded type rotation number, then $\Hq_1$ and $\Hq_2$ have the same Hausdorff dimension, lower box dimension, and upper box dimension.
\end{corx}

Consider a critical quasicircle map $f: \Hq \to \Hq$ with bounded type rotation number $\theta$. Denote by $c$ the critical point of $f$ and by $\{p_n/q_n\}_{n\in\N}$ the best rational approximations of $\theta$. We define the \emph{$n^{\text{th}}$ scaling ratio} of $f$ by
\begin{equation}
\label{eqn:scaling-ratio}
s_n(f) := \frac{f^{q_{n+1}}(c)-c}{f^{q_n}(c)-c}.
\end{equation}

\begin{corx}[Universality of scaling ratios]
\label{cor:universality}
    If two $(d_0,d_\infty)$-critical quasicircle maps $f_1 : \Hq_1 \to \Hq_1$ and $f_2: \Hq_2 \to \Hq_2$ have the same bounded type rotation number, then asymptotically they have the same scaling ratios:
    \[
    \frac{s_n(f_2)}{s_n(f_1)}  \xrightarrow{}  1 \text{ exponentially fast as } n \to \infty.
    \]
\end{corx}

Recall that $\theta = [0;a_1,a_2,\ldots]$ is a \emph{quadratic irrational} if there exist integers $m\geq 1$ and $p\geq 1$ such that $a_{n+p} = a_n$ for all $n \geq m$. If $m=1$, then $\theta$ is called \emph{periodic}.
 
\begin{thmx}[Self-similarity]
\label{thm:self-similarity}
    Consider a $(d_0,d_\infty)$-critical quasicircle map $f: \Hq \to \Hq$ with a quadratic irrational rotation number $\theta$. Then, $\Hq$ is asymptotically self-similar about the critical point. The self-similarity factor is universal depending only on $d_0$, $d_\infty$, and $\theta$.
\end{thmx}

For every $n \geq 1$, denote by $I_{n}$ the shortest interval in $\Hq$ connecting $c$ and $f^{q_n}(c)$. The \emph{$n$\textsuperscript{th} pre-renormalization} of $f$ is the commuting pair 
\[
(f^{q_n}|_{I_{n-1}}, f^{q_{n-1}}|_{I_n}),
\]
and the \emph{$n$\textsuperscript{th} renormalization} $\renorm^n f$ of $f$ is the normalized commuting pair obtained by rescaling the $n$\textsuperscript{th} pre-renormalization to unit size by either an affine map if $n$ is even, or an anti-affine map if $n$ is odd. We also define renormalizations of arbitrary commuting pairs in a way that is compatible with renormalizations of critical quasicircle maps.

\begin{thmx}[Renormalization horseshoe]
\label{horseshoe}
    Let $N$ be a positive integer. There is a renormalization-invariant compact set $\mathcal{A}_N$ inside the space $\mathcal{CP}_N$ of normalized commuting pairs of fixed criticality and of rotation number in $\Theta_N$ with the following properties.
        \begin{enumerate}[label=\textnormal{(\arabic*)}]
            \item The renormalization operator $\renorm: \mathcal{A}_N \to \mathcal{A}_N$ is topologically conjugate to the shift operator on the bi-infinite shift space of $N$ symbols.
            \item For any $\zeta$ in $\mathcal{CP}_N$, the distance between $\renorm^n \zeta$ and $\mathcal{A}_N$ tends to $0$ exponentially fast as $n \to \infty$.
        \end{enumerate}
\end{thmx}

A precise version of the two theorems above can be found in Theorems \ref{thm:exponential-convergence}, \ref{thm:horseshoe}, and \ref{thm:A-invariant-quasicircle}. 

The renormalization theory of critical circle maps has been extensively studied in the last 30 years. De Faria \cite{dF99} introduced the notion of \emph{holomorphic commuting pairs} and proved the universality of scaling ratios and the existence of renormalization horseshoe for critical circle maps. $C^{1+\alpha}$ rigidity was established by de Faria and de Melo \cite{dFdM2} for bounded type rotation number, and later by Khmelev and Yampolsky \cite{KY06} for arbitrary irrational rotation number by studying parabolic bifurcations. Moreover, Yampolsky extended the horseshoe for all irrational rotation numbers in \cite{Y01}, and established global hyperbolicity of renormalization \cite{Y02,Y03}.

In the bounded type regime, if we assume that either $d_0$ or $d_\infty$ is one, the quasicircle $\Hq$ will be the boundary of a rotation domain and the results above have also been studied by various authors. By Douady-Ghys surgery \cite{D87,G84}, $\Hq$ can be assumed to be the boundary of a Siegel disk. The problem of smoothness and Hausdorff dimension of quasiconformal Siegel disks was solved by Graczyk and Jones \cite{GJ02}. Stirnemann \cite{St94} first gave a computer-assisted proof of the existence of a renormalization fixed point with a golden-mean Siegel disk. McMullen \cite{McM98} applied a measurable deep point argument to prove self-similarity of Siegel disks of quadratic polynomials as well as the existence of renormalization horseshoe. In \cite[\S4]{AL22}, Avila and Lyubich proved quasiconformal rigidity of bounded type quadratic Siegel disks. Note that the measurable deep point argument gives an automatic upgrade to $C^{1+\alpha}$ rigidity. McMullen and Avila-Lyubich's methods extend to bounded type unicritical Siegel disks of arbitrary degree. Therefore, our results are only new when $\Hq$ is a \emph{Herman quasicircle}, that is, when $\Hq$ is not contained in the closure of a rotation domain.
 
Gaidashev and Yampolsky \cite{GY22} gave a computer-assisted proof of the golden mean hyperbolicity of renormalization of Siegel disks using the formalism of \emph{almost commuting pairs}. In \cite{DLS}, Dudko, Lyubich, and Selinger constructed a compact analytic operator, called \emph{Pacman renormalization operator}, with a hyperbolic fixed point whose stable manifold has codimension one and consists of maps with a Siegel disk of a fixed rotation number of periodic type. On the other hand, Inou and Shishikura \cite{IS08} gave a computer-assisted proof of the existence and hyperbolicity of quadratic Siegel renormalization fixed points of sufficiently high type rotation number via the near-parabolic regime. This was later extended by Ch\'eritat \cite{Che22}, who presented a computer-free proof of hyperbolicity in high type in the general unicritical case.

% Anonymous Version
In an upcoming sequel of this paper \cite{Lim24}, we will develop an adaptation of Pacman renormalization for critical quasicircle maps. More precisely, we will construct a compact analytic renormalization operator with a hyperbolic fixed point whose stable manifold has codimension one and consists of $(d_0,d_\infty)$-critical quasicircle maps with periodic type rotation number. The existence of the fixed point as well as exponential convergence along the stable manifold will be immediate consequences of Theorem \ref{horseshoe}. Moreover, Theorem \ref{thm:combinatorial-rigidity} will serve as a vital ingredient towards showing the existence of its unstable manifold.

% Real Version
% In an upcoming sequel \cite{Lim24} of this paper, we will develop an adaptation of Pacman renormalization for critical quasicircle maps. More precisely, we will construct a compact analytic renormalization operator with a hyperbolic fixed point whose stable manifold has codimension one and consists of $(d_0,d_\infty)$-critical quasicircle maps with periodic type rotation number. The existence of the fixed point as well as exponential convergence along the stable manifold will be immediate consequences of Theorem \ref{horseshoe}. Moreover, Theorem \ref{thm:combinatorial-rigidity} will serve as a vital ingredient towards showing the existence of its unstable manifold.

\subsection{Outline}

In the first part of this paper, we study bounded type rotation quasicircles of rational maps. In Section \S\ref{sec:rot-curves}, we study the general theory of rotation curves. We work under the bounded type assumption, in which, by Proposition \ref{trichotomy}, we have a trichotomy of rotation curves $\Hq$ based on its inner/outer criticality and on whether or not $\Hq$ is a Herman curve. Assuming further that $\Hq$ is a quasicircle containing a critical point, we prove in Section \S\ref{ss:approx-rot} that on a neighborhood of $\Hq$, the map $f$ always acts as an \emph{approximate rotation} until it gets close to a critical point.

In Section \S\ref{sec:rat-map}, we study rational maps $f$ that are J-rotational of bounded type and prove Theorem \ref{thm:NILF-main}. The key ingredient is Theorem \ref{critical}, which states that for every \emph{nice} point $z$ in $J(f)$ and every scale $r>0$, we can always find a disk of size comparable to $r$ and of distance at most $r$ from $z$ that is mapped univalently by some iterate of $f$ to a disk containing a critical point of $f$. Theorem \ref{critical} is indeed an extension of \cite[Theorem 3.2]{McM94}. We generalize McMullen's proof in the context of many rotation quasicircles and many critical points by applying the approximate rotation mechanism in Section \S\ref{ss:approx-rot}. 

Section \S\ref{sec:applications} is dedicated to a number of consequences of Theorem \ref{thm:NILF-main}. We adopt the classical pullback argument to prove combinatorial rigidity of rational maps in $\HQspace_{d_0,d_\infty,\theta}$, and then apply combinatorial rigidity to deduce the continuity of combinatorics and rotation number. We then explore two immediate applications of Theorem \ref{thm:combinatorial-rigidity}. Firstly, in Section \S\ref{ss:blaschke}, we deduce that Douady-Ghys surgery provides a homeomorphism between the space of trivial Herman curves in $\bigcup_{\theta \in \Theta_N} \HQspace_{d,d,\theta}/_\sim$ and the moduli space of polynomial Siegel disks satisfying a condition similar to \ref{defn4}. Secondly, in Section \S\ref{ss:antipode}, we answer an open question regarding the landing of \emph{hairs} inside the moduli space of antipode-preserving cubic rational maps, which were originally studied in \cite{BBM18,BBM23}. See Figure \ref{fig:antipodal} for an illustration.

In Section \S\ref{sec:deep-points}, we prove under certain assumptions that if a rational map $f$ admits a bounded type rotation quasicircle $\Hq$, then every point on $\Hq$ is deep point of the \emph{local filled Julia set} of $f|_\Hq$ (roughly, the closure of iterated preimages of $\Hq$ near $\Hq$). This, when applied to maps in $\HQspace_{d_0,d_\infty,\theta}$, will act as one of the two main ingredients in upgrading quasiconformal rigidity to $C^{1+\alpha}$ rigidity.

The second part of this paper is concerned with the study of critical quasicircle maps $f: \Hq \to \Hq$. As previously mentioned, it suffices to assume that $\Hq$ is a Herman quasicircle and so the inner and outer criticalities of the critical point are at least two. In Section \S\ref{ss:uni-herman-qc}, we argue that $f$ is a conformal welding of a pair of \emph{quasicritical circle maps}, which are a quasiregular analog of critical circle maps. Based on the work of Avila and Lyubich \cite[\S3]{AL22}, we know that most results on critical circle maps, such as complex bounds and quasiconformal rigidity, hold for quasicritical circle maps of bounded type. In Section \S\ref{ss:complex-bounds}, we introduce the concept of \emph{butterflies}, an analog of holomorphic commuting pairs, and transfer complex bounds for quasicritical circle maps to complex bounds in our setting (Theorems \ref{thm:complex-bounds} and \ref{thm:complex-bounds-butterflies}). By means of the pullback argument, we then show that complex bounds imply quasiconformal rigidity. 

To show that our quasiconformal conjugacy is $C^{1+\alpha}$-conformal, we apply McMullen's Dynamic Inflexibility Theorem \cite[Theorem 9.15]{McM96}, which was originally applied to infinitely renormalizable quadratic-like maps. In our context, this relies on our deep point result in Section \S\ref{sec:deep-points}, together with the \emph{uniform twisting} property, which has been obtained in the course of the proof of the first part of Theorem \ref{thm:NILF-main}. While similar in spirit, our proof avoids de Faria and de Melo's elaborate use of \emph{good geometric control} of holomorphic commuting pairs in \cite[\S4--7]{dF99}.

In the last section, we prove a number of applications of $C^{1+\alpha}$ rigidity. Our universality results are immediate consequences of Theorem \ref{thm:c1plusalpharigidity}. The proof of Corollary \ref{cor:smoothness} uses an additional tool, which is Peter Jones' beta numbers (see Proposition \ref{prop:HD>1}). The construction of renormalization horseshoe is a standard tower rigidity argument. Lastly, Theorem \ref{thm:self-similarity} follows from self-similarity of the invariant quasicircle of each of the renormalization periodic points in the horseshoe.

\subsection{Acknowledgements}

I would like to thank my advisor Dzmitry Dudko for his support and advice, and Michael Yampolsky for helpful initial suggestions without which this work might not have existed. I would also like to thank Araceli Bonifant and John Milnor for illuminating discussion on their work in \cite{BBM23} with Xavier Buff. This project has been partially supported by the NSF grant DMS 2055532 and by Simons Foundation International, LTD. 

\subsection{Notation}
\label{ss:notation}

We will often use the following asymptotic notations.
\begin{itemize}
    \item $g = O(h)$ when $h >0$ and $|g| \leq \alpha h$ for some implicit constant $\alpha>0$;
    \item $g \succ h$ when $g,h >0$ and $g \geq \alpha h$ for some implicit constant $\alpha > 0$;
    \item $g \asymp h$ when $g \succ h$ and $h \succ g$.
\end{itemize}

For any pair of disjoint sets $I$ and $J$ in $\RS$, we say that $I$ and $J$ are \emph{well separated} if there exists an annulus $A$ of modulus $\modu(A) \asymp 1$ separating $I$ and $J$. For any set $I$ contained in a domain $D$, we say that $I$ is \emph{well contained} in $D$ if $I$ and $\partial D$ are well separated.

Equip the Riemann sphere $\RS$ with the spherical metric. On compact subsets of $\C$, the spherical metric is equivalent to the Euclidean metric and so we often make no distinction between either of them. Denote the spherical distance between two points $x$ and $y$ by $|x-y|$ and between two subsets $A$ and $B$ of $\RS$ by $\dist(A, B)$. Denote by $\D(x,\varepsilon)$ the open disk of radius $\varepsilon>0$ centered at $x \in \RS$. Given a pointed topological disk $(U, x)$, we define the following.
\begin{itemize}
    \item $\rin(U,x) := \dist(x,\partial U)$, the inner radius of $U$ about $x$;
    \item $\rout(U,x) := \inf \{ \varepsilon > 0 \: : \: U \subset \D(x,\varepsilon)\}$, the outer radius of $U$ about $x$.
\end{itemize}
For any $C\geq1$, we say that $(U,x)$ has \emph{$C$-bounded shape} if $\rout(U,x) \leq C \rin(U,x)$. We say that $(U,x)$ has \emph{bounded shape} if it has $C$-bounded shape for some implicit constant $C>1$.

Given a hyperbolic Riemann surface $\Omega$, denote by $d_\Omega$ the corresponding hyperbolic metric of $\Omega$. The hyperbolic distance between subsets of $\Omega$ will be denoted by $\dist_\Omega(\cdot, \cdot)$. Denote by $\D_\Omega(x,\varepsilon)$ the hyperbolic open ball of radius $\varepsilon>0$ centered at $x \in \Omega$. Given a pointed topological disk $(U, x)$ in $\Omega$, we denote by $r_{\text{in}, \Omega}(U,x) := \dist_\Omega(x,\partial U)$ the hyperbolic inner radius of $U$ about $x$.

\section{Rotation curves}
\label{sec:rot-curves}

\begin{definition}
    Let $f: U \to V$ be a holomorphic map. An open domain $W$ is a \emph{rotation domain} of $f$ if 
    \begin{enumerate}[label=\textnormal{(\roman*)}]
        \item $W \subset U \cap V$ and $f(W)=W$,
        \item $f|_W$ is analytically conjugate to an irrational rotation, and
        \item every other domain satisfying (i) and (ii) is either contained in $W$ or disjoint from $W$.
    \end{enumerate}
    A Jordan curve/quasicircle $\Hq$ is a \emph{rotation curve/quasicircle} of $f$ if 
    \begin{enumerate}[label=\textnormal{(\roman*)}]
    \setcounter{enumi}{3}
        \item $\Hq \subset U \cap V$ and $f(\Hq)=\Hq$, and
        \item $f|_\Hq$ is conjugate to an irrational rotation.
    \end{enumerate} 
    We say that $\Hq$ is a \emph{Herman curve/quasicircle} of $f$ if it is a rotation curve/quasicircle that is not contained in the closure of a rotation domain.
\end{definition}

Let us fix a holomorphic map $f: U \to V$, where $U,V \subset \C$ are domains in $\RS$ and assume that it admits a rotation curve $\Hq \subset U$. Let $\phi: \Hq \to \T$ be a conjugacy between $f|_\Hq$ and the irrational rotation $R_\theta|_\T$. Note that $\phi$ is unique up to post-composition with any rotation.

\subsection{Irrational rotations}
\label{ss:irrational}

Let us identify $\T$ with the quotient $\R / \Z$, in which $R_\theta$ can be written as $R_\theta(x)=x+\theta$. Equip $\Hq$ with the pushforward under $\phi^{-1}$ of the Euclidean metric on $\T$, called the \emph{combinatorial metric}; this is the unique invariant normalized metric on $\Hq$. For any interval $I \subset \Hq$, denote its combinatorial length by $|I|$. For any pair of distinct points $x,y \in \Hq$, we denote by $[x,y]$ the shortest closed interval in $\Hq$ having endpoints $x$ and $y$. 

Let $\{ p_n/q_n \}_{n \in \N}$ be the sequence of best rational approximations of $\theta$. This sequence is determined by the recurrence relation 
\[
p_{n} = a_{n} p_{n-1} + p_{n-2} \quad \text{and} \quad q_{n} = a_{n} q_{n-1} + q_{n-2}
\]
where $p_0= q_{-1}=0$, $q_0 = p_{-1} = 1$, and $[0;a_1,a_2,\ldots]$ is the continued fraction expansion of $\theta$. The $q_n$'s are precisely the first return times for $R_\theta$ (and thus for $f|_\Hq$ too) which alternate in the following fashion. For any $x \in \T$,
\[
R^{q_1}_\theta(x) < R^{q_3}_\theta(x) < R^{q_5}_\theta(x) < \ldots < x < \ldots < R^{q_6}_\theta(x) < R^{q_4}_\theta(x) < R^{q_2}_\theta(x).
\]

\begin{definition}
    An interval $I \subset \Hq$ is a \emph{level $n$ combinatorial interval} if it is of the form $[x, f^{q_n}(x)]$ for some $x \in \Hq$.
\end{definition} 

Every level $n$ combinatorial interval has the same combinatorial length equal to
\begin{equation}
    \label{eqn:comb-length}
    l_n: = |p_n - q_n\theta|.
\end{equation}

\begin{proposition}
\label{renormalization-tiling}
    For any $c \in \Hq$ and $n \in \N$, the collection of combinatorial intervals
        \[
            \mathcal{P}_{n}(c) := \left\{ [f^i(c), f^{q_n+i} (c)] \right\}_{i=0}^{q_{n+1}-1}  \cup \{ [f^{q_{n+1}+j}(c), f^j(c)] \}_{j=0}^{q_n-1}
        \]
    forms a tiling of $\Hq$, that is, they have pairwise disjoint interiors and their union is $\Hq$. Moreover, $\mathcal{P}_{n+1}(c)$ is a refinement of $\mathcal{P}_{n}(c)$.
\end{proposition}

Recall that $\theta$ is of \emph{bounded type} if there is a uniform bound on the terms $a_n$ in its continued fraction expansion $[0;a_1,a_2,\ldots]$. If so, we denote the optimal bound by 
\[
\beta(\theta) := \max_{i \geq 1} a_i.
\]
For any positive integer $N$, we define the set $\Theta_N$ to be the set of bounded type irrationals $\theta \in (0,1)$ satisfying $\beta(\theta)\leq N$. The bounded type assumption controls the rate of decrease of the lengths in (\ref{eqn:comb-length}).

\begin{proposition}
\label{bounded-type}
    If $\theta$ is in $\Theta_N$, there exists a pair of constants $\tilde{C}, C>1$ depending only on $N$ such that for every positive integer $n$, 
    \[
    \tilde{C} l_{n+1} \leq l_n \leq C l_{n+1}.
    \]
\end{proposition}

To control the orbit of critical values of $f$ along $\Hq$, we will use the following lemma several times in this paper.

\begin{lemma}
\label{counting}
    Suppose $\theta$ is in $\Theta_N$ and $S$ is a finite subset of $\Hq$. There is some constant $\varepsilon>0$ depending only on $N$ and $|S|$ such that for all $n\in \N$, every combinatorial interval $I\subset \Hq$ of level $n$ contains a subinterval $J \subset I$ of length $|J| \geq \varepsilon l_n$ that is disjoint from $\bigcup_{i=0}^{q_{n+2}-1} f^i(S)$.
\end{lemma}
    
\begin{proof}
    By Proposition \ref{renormalization-tiling}, for every $c \in S$, the finite orbit $\mathcal{O}_c = \{ f^i(c) \}_{i=0,\ldots q_{n+2}-1}$ partitions $\Hq$ into intervals of lengths between $l_{n+2}$ and $l_{n}$. By Proposition \ref{bounded-type}, the number of points in $\bigcup_{c \in S} \mathcal{O}_c$ that lie within $I$ is at most some constant $K$ depending only on $N$ and $|S|$. Therefore, there is a subinterval of $I$ of length at least $l_n/K$ that satisfies the desired property.
\end{proof}

In later sections, we will assume that the conjugacy $\phi$ is quasisymmetric, which allows us to transfer what is known in the combinatorial metric back to $\Hq$ as a subset of $\RS$, equipped with the spherical metric.

\begin{lemma}
\label{lem:qc-control}
    Suppose $\theta$ is in $\Theta_N$ and the conjugacy $\phi: \Hq \to \T$ is $K$-quasisymmetric. For every point $c$ on $\Hq$,
    \begin{enumerate}[label=\textnormal{(\arabic*)}]
        \item the tilings $\mathcal{P}_n(c)$ have bounded geometry, that is, the diameters of any two adjacent tiles of the same level, or any two consecutive nested tiles, are comparable with a constant depending only on $N$ and $K$;
        \item there are positive constants $C, \varepsilon_1, \varepsilon_2$ depending only on $N$ and $K$ such that $\varepsilon_1<\varepsilon_2<1<C$ and for every $n \geq 2$,
        \[
            C^{-1} \varepsilon_1^n \leq \frac{|f^{q_n}(c)-c|}{\diam(\Hq)} \leq C \varepsilon_2^n.
        \]
    \end{enumerate}
\end{lemma}

\subsection{A trichotomy}
\label{ss:combinatorics}

Let us label the two connected components of $\RS \backslash \Hq$ by $Y^0$ and $Y^\infty$. Unless otherwise stated, throughout this paper, we will assume without loss of generality that $\Hq$ separates $0$ and $\infty$, $Y^0$ contains $0$, and $Y^\infty$ contains $\infty$.

There is a pair of functions
\[
d_0: \Hq \to \N \quad \text{and} \quad d_\infty:\Hq \to \N
\]
where for any point $c$ on $\Hq$, any $\bullet \in \{0,\infty\}$, and any point $z \in Y^\bullet$ sufficiently close to $f(c)$, the number of preimages of $z$ near $c$ in $Y^\bullet$ is equal to $d_\bullet(c)$. The local degree of $f$ at $c$ is $d_0(c)+d_\infty(c)-1$, and clearly $c$ is a critical point if and only if either $d_0(c) \geq 2$ or $d_\infty(c) \geq 2$.

\begin{definition}
    We say that a point $c$ on $\Hq$ is an \emph{inner} critical point if $d_0(c)\geq2$, and an \emph{outer} critical point if $d_\infty(c)\geq2$.
\end{definition} 

By the compactness of $\Hq$, the numbers
\[
D_0 := \sum_{c \in \Hq} d_0(c)-1 \qquad \text{ and } \qquad D_\infty := \sum_{c \in \Hq} d_\infty(c)-1
\]
are finite. The sum $D_0+D_\infty$ is precisely the number of critical points of $f$ on $\Hq$ counting multiplicity. Without loss of generality, we make the convention that
\[
D_\infty\geq D_0. 
\]

There are three distinct cases.
\begin{enumerate}[label=\protect\circled{\textnormal{\Alph*}}]
    \item\label{case-A} $D_\infty=0$ and $\Hq$ contains no critical points of $f$.
    \item\label{case-B} $D_\infty>D_0=0$ and $\Hq$ has only outer critical points.\\
    In this case, denote by $\left(c^\infty_1,\ldots,c^\infty_{D_\infty}\right)$ the tuple of outer critical points of $f$ along $\Hq$ counting outer multiplicity, i.e. every $c \in \Hq$ appears in the tuple precisely $d_\infty(c)-1$ times.
    \setcounter{enumi}{7}
    \item\label{case-H} $D_0 \geq 1$ and $\Hq$ has both inner and outer critical points.\\
    In this case, denote by $\left(c^\infty_1,\ldots,c^\infty_{D_\infty}\right)$ and $\left(c^0_1,\ldots,c^0_{D_0}\right)$ the tuples of outer and inner critical points counting outer and inner multiplicity respectively.
\end{enumerate}

Let us formally define the combinatorics of rotation curves as follows. For any $n \in \N$, the $n^{\text{th}}$ symmetric product $\SP^n(\T)$ of the unit circle $\T$ is the quotient of the $n$-dimensional torus $\T^n$ under the symmetric group $S_n$ acting by permutation. Elements of $\SP^n(\T)$ are precisely unordered $n$-tuples of elements of $\T$. 

\begin{definition}
\label{combspace}
     Define $\mathcal{S}_{d}$ to be the quotient space of $\SP^{d-1}(\T)$ modulo the action of $\T$ by rigid rotation. Define $\mathcal{C}_{m,n}$ to be the quotient space of $\SP^{m-1}(\T) \times \SP^{n-1}(\T)$ modulo the action of $\T$ by rigid rotation. %Both $\mathcal{S}_d$ and $\mathcal{C}_{m,n}$ are endowed with the quotient topology.
\end{definition}

Both $\mathcal{S}_{d}$ and $\mathcal{C}_{m,n}$ are compact orbifolds of real dimension $d-2$ and $m+n-3$ respectively.

\begin{definition}
\label{combi}
    In case \ref{case-B}, the \emph{combinatorics} of $f|_\Hq$ is the element $\comb(f)$ of $\mathcal{S}_{D_\infty}$ induced by the tuple $\left(\phi(c^\infty_1),\ldots,\phi(c^\infty_{D_\infty})\right)$. 
    In case \ref{case-H}, the \emph{combinatorics} of $f|_\Hq$ is the element $\comb(f)$ of $\mathcal{C}_{D_0,D_\infty}$ induced by the pairs of tuples $\left(\phi(c^0_1),\ldots, \phi(c^0_{D_0})\right)$ and $\left(\phi(c^\infty_1), \ldots, \phi(c^\infty_{D_\infty})\right)$.
\end{definition}

Note that $\comb(f)$ is well-defined because the conjugacy $\phi|_\Hq$ is unique up to post-composition with rigid rotation. One may also compare our definition with the notion of \emph{signature} of a multicritical circle map. See \cite[\S6]{EdF}.

Consider the set $\mathscr{H} \subset \R \backslash \Q$ of \emph{Herman numbers}. The set $\mathscr{H}$ has full Lebesgue measure and is characterized by a rather complicated arithmetic condition that was devised by Herman and Yoccoz \cite{H79, Yo02} as the optimal condition for an analytic circle diffeomorphism to be analytically linearizable. In this paper, we only need the property that $\mathscr{H}$ contains the set of bounded type irrationals.

\begin{proposition}[Trichotomy of rotation curves]
\label{trichotomy}
    Suppose $\theta \in \mathscr{H}$. Exactly one of the following holds.
    \begin{enumerate}[label=\protect\circled{\textnormal{\Alph*}}]
        \item $\Hq$ is an analytic curve contained in a rotation domain of $f$.
        \item $\Hq$ is a boundary component of a rotation domain of $f$ and contains an outer critical point but no inner critical points.
        \setcounter{enumi}{7}
        \item $\Hq$ is a Herman curve containing an inner critical point and an outer critical point.
    \end{enumerate}
\end{proposition}

\begin{proof}
    Recall the numbers $D_0$ and $D_\infty$ in the discussion above. We are assuming that $D_\infty \geq D_0$. It is clear that if $D_0 \geq 1$, then $\Hq$ must be a Herman curve.
    
    Suppose $D_0=0$, so that $\Hq$ contains no inner critical points. There is an annulus $W \subset Y^0$ such that $\Hq$ is one of the boundary components of $W$ and $f$ is univalent on $W$. Since $f|_{\Hq}$ is an orientation-preserving self-homeomorphism of $\Hq$, the image $Z:= f(W)$ is again an annulus contained in $Y^0$ with $\Hq$ being one of its boundary components. Pick a conformal isomorphism $\psi: Y^0 \to \D$ and define the univalent map $F:= \psi \circ f \circ \psi^{-1}$ from $\psi(W)$ to $\psi(Z)$. By Schwarz reflection, $F$ extends continuously to a univalent map $F:W' \to Z'$ where $W'$ and $Z'$ are the smallest $\T$-symmetric annuli containing $\psi(W)$ and $\psi(Z)$ respectively. 
    
    The map $F$ restricts to an analytic circle diffeomorphism with rotation number $\theta$. Since $\theta$ is a Herman number, $F|_\T$ must be analytically linearizable, so $F$ admits a $\T$-symmetric Herman ring. By pulling back this Herman ring via $\psi$, we obtain an invariant annulus $A^0 \subset W$ such that $\Hq$ is a boundary component of $A^0$ and $f|_{A^0}$ is analytically conjugate to $R_\theta$. Denote by $A$ the rotation domain of $f$ containing $A^0$.

    If $\Hq$ contains an outer critical point, then $f|_\Hq$ cannot be analytically conjugate to $R_\theta$ and $\Hq$ has to be a boundary component of $A$. Otherwise, $D_\infty=0$ and by the same argument, there is an annulus $A^\infty \subset Y^\infty$ such that $\Hq$ is a boundary component of $A^\infty$ and $f|_{A^\infty}$ is analytically conjugate to $R_\theta$. Hence, $A^0 \cup \Hq \cup A^\infty$ lies in $A$ and $\Hq$ is an invariant analytic curve.
\end{proof}

The trichotomy breaks when $\theta \not\in \mathscr{H}$. For instance, there exist cubic rational maps admitting Herman curves of arbitrary non-Herman irrational rotation number containing no critical points. See \cite[Proposition 6.6]{BF14} and \cite{Y22}. Nonetheless, when $f$ is a rational map, the following weaker property still holds.

\begin{proposition}
\label{herman-in-pf}
    Every rotation curve of a rational map is contained in either its Fatou set or its postcritical set.
\end{proposition}

\begin{proof}
    Let $\Hq$ be a rotation curve of a rational map $f$ contained in its Julia set $J(f)$. Suppose for a contradiction that there is a point $x \in \Hq$ lying outside of the postcritical set $P(f)$. Since $P(f)$ is compact, there is an open disk neighborhood $D$ of $x$ that lies within $\Omega := \RS \backslash P(f)$. By taking iterated preimages of $D$, we deduce that $\Hq$ must a compact subset of $\Omega$ because $f^{-1}(\Omega) \subset \Omega$. 
    
    Denote by $A$ the connected component of $\Omega$ containing $\Hq$, and by $B$ the connected component of $f^{-1}(\Omega)$ containing $\Hq$. Since $f^{-1}(\Omega) \subset \Omega$, then $B$ is a subset of $A$.
    
    Suppose $B=A$. Then, $\{f^n|_A\}_{n\in\N}$ forms a normal family and $\Hq$ is a subset of the Fatou set. This contradicts the assumption that $\Hq$ is a subset of $J(f)$.

    Suppose $B$ is a proper subset of $A$. In this case, $f|_B$ expands the hyperbolic metric of $A$. In particular, with respect to the hyperbolic metric of $A$, $f|_{\Hq}$ is a uniformly expanding homeomorphism. This contradicts the fact that compact metric spaces do not admit expanding self homeomorphisms. 
\end{proof}

\subsection{Approximate rotation}
\label{ss:approx-rot}

Let us now assume that $\Hq$ is a quasicircle containing a critical point of $f$. Further regularity of the conjugacy $\phi$ can be achieved under the bounded type assumption.

\begin{theorem}[\cite{Pe04}]
\label{petersen}
    The rotation number $\theta$ of $\Hq$ is of bounded type if and only if there is a quasiconformal map $\phi: \RS \to \RS$ which restricts to a conjugacy between $f|_\Hq$ and $R_\theta|_\T$.
\end{theorem}

In what follows, let us assume that the rotation number $\theta$ is of bounded type and fix a linearizing quasiconformal map $\phi: \RS \to \RS$ guaranteed by this theorem. Consider the function 
\begin{equation}
\label{eqn:escape-function}
    L(z) := \log\left(\dist(\phi(z),\T)\right).
\end{equation}
Given a point $z$ near $\Hq$, we will measure the rate of escape of iterates of $z$ using the function $L$.

Let us fix $\kappa>0$ such that $f$ extends to a holomorphic map on the open annulus 
\[
    A := \{ - \infty \leq L(z) < -\kappa \}
\]
 on which the only critical points of $f$ inside of $A$ lie on $\Hq$. Split $A$ into two annuli
\[
    A^0 := A \cap Y^0 \qquad \text{and} \qquad A^\infty := A \cap Y^\infty.
\]
We call $A$, $A^0$, and $A^\infty$ a \emph{collar}, \emph{inner collar}, and \emph{outer collar} of $f|_\Hq$.

Let us define the quasi-rotation 
\[
    F(z) := \phi^{-1} (e^{2\pi i\theta} \phi(z)).
\]
Clearly, $F$ coincides with $f$ on $\Hq$. In case \ref{case-B}, we can modify $\phi$ such that it also provides a conjugacy on the inner collar $A^0$, and thus $F \equiv f$ on $\overline{A^0}$.

Let us equip $\nohq: = \RS \backslash \Hq$ with the hyperbolic metric. The following definition is inspired by \cite[\S3]{McM98}.

\begin{definition}
\label{approx-rotation}
    Suppose an iterate $f^i: U \to V$ is well defined for some $i \geq 0$ and some pair of topological disks $U,V \subset A$. Given a constant $C>0$, we say that $f^i :U \to V$ is an $C$-\emph{approximate rotation} if it is a univalent function of bounded distortion such that 
    \[
    d_{\nohq}(f^i(x), F^i(x))\leq C \qquad \text{ for all } x \in U \backslash \Hq.
    \]
    We will also call $f^i :U \to V$ an \emph{approximate rotation} if it is a $C$-approximate rotation for some implicit constant $C>0$ independent of $i$, $U$, and $V$.
\end{definition}

Given a Jordan domain $U$ and a pair of disjoint arcs $I$ and $J$ on the boundary of $U$, we denote by $\mathcal{L}_U(I,J)$ the extremal length of the family of proper curves in $U$ connecting $I$ and $J$. The domain $U$ is conformally equivalent to a Euclidean rectangle where $I$ and $J$ correspond to the vertical sides of unit length and $\mathcal{L}_U(I,J)$ is equal to the length of the horizontal side.

\begin{lemma}[One-sided approximate rotation]
\label{approx-rot-01}
There exists a constant $C>0$ such that the following holds. For any point $z \in \Hq$ and sufficiently small scale $r>0$, there exists a $C$-approximate rotation 
\[
    f^i : (U,y) \to (V,c)
\]
such that 
\begin{enumerate}[label=\textnormal{(\arabic*)}]
    \item\label{approx-1} the point $y$ lies in $\partial U \cap \Hq$ and its image $c$ is an outer critical point of $f$,
    \item\label{approx-2} both $U$ and $V$ are contained in the outer collar $A^\infty$, and 
    \item\label{approx-3} the boundary $\partial U$ of $U$ contains the interval $[y,z] \subset \Hq$ satisfying
    \[
    \mathcal{L}_{U}([y,z], \partial U \backslash \Hq) \succ 1 \qquad \text{and} \qquad \dist(z,\partial U \backslash \Hq) \asymp r.
    \]
\end{enumerate}
\end{lemma}

In Case \ref{case-H}, this lemma also works for inner critical points. 

For any $0<a<\pi$ and any small interval $I \subset \T$, we define $P_a(I) \subset \D$ to be the Jordan disk enclosed by the interval $I$ together with the unique circular arc in $\D$ that has the same endpoints as $I$ and meets the circular arc $\T\backslash I$ at an angle of $a$. To prove Lemma \ref{approx-rot-01}, we will use the following tool. 

\begin{lemma}[{\cite[ Lemma 3.3]{McM98}}]
    \label{lem:poincare-nbh}
    Consider a $K$-quasiconformal map $g: P_\alpha(I) \to \D$ which extends continuously to the identity on some interval $I \subset \T$. For any $\beta\in (0,\alpha)$, there is some constant $C=C(\alpha,\beta,K)>0$ such that
    \[
        d_\D(g(x),x) \leq C \quad \text{ for all } x \in P_\beta(I).
    \]
\end{lemma}

\begin{proof}[Proof of Lemma \ref{approx-rot-01}]
    Since $\phi$ is uniformly continuous with respect to the hyperbolic metric of $\nohq$, it is sufficient to prove the lemma in the $w$-coordinate, where $w=\phi(z)$. (Compare with \cite[Theorem 3.4]{McM98}.) In the $w$-coordinate, $f$ is an irrational rotation along $\Hq$, which is the unit circle, and all iterates of $f$ are quasiregular with uniform dilatation. Let us pick a small scale $r>0$ and a point $w$ on $\Hq$.
    
    Denote by $\{p_k/q_k\}_{k\in \N}$ the rational approximations of $\theta$ and consider the combinatorial lengths $l_k := |p_k - q_k \theta|$. Let us pick $n \in \N$ such that $r \asymp l_n$. Apply Lemma \ref{counting} by taking $S$ to be the set of critical values of $f$ on $\Hq$ in order to obtain a pair of intervals $J_{q_{n+2}} \subset I_{q_{n+2}}$ in $\Hq$ such that
    \begin{itemize}
        \item[(i)] $J_{q_{n+2}}$ contains the level $n$ combinatorial interval centered at $f^{q_{n+2}}(w)$;
        \item[(ii)] the endpoints of $J_{q_{n+2}}$ split $I_{q_{n+2}}$ into three connected components each having combinatorial length $\asymp l_n$;
        \item[(iii)] $I_{q_{n+2}} \backslash J_{q_{n+2}}$ does not contain any critical value of $f^{q_{n+2}}$.
    \end{itemize}
    For $j=0,1,\ldots q_{n+2}$, let 
    \[
    J_j := (f|_{\Hq})^{-q_{n+2}+j}(J_{q_{n+2}}) \quad \text{and} \quad I_j := (f|_{\Hq})^{-q_{n+2}+j}(I_{q_{n+2}}).
    \]
    By Proposition \ref{renormalization-tiling}, there is some minimal $i<q_n+q_{n+1}$ such that $J_i$ contains an outer critical point $c \in \Hq$. Let $y := (f|_{\Hq})^{-i} (c)$.

    Let $\psi_\infty: \D \to \nohq^\infty$ be a biholomorphism sending $0$ to a point outside of $A^\infty$. By Carath\'eodory's theorem, $\psi_\infty$ extends to a homeomorphism on the boundary $\T \to \Hq$. For any $a\in(0,\pi)$, we define the domain 
    \[
    P^\infty_a(I_i) := \psi_\infty\left(P_a(\psi_\infty^{-1}(I_i))\right) \subset \nohq^\infty.
    \]
    We will set
    \[
        V:= P^\infty_{\frac{\pi}{2}}(I_i).
    \]
    
    \begin{claim} 
        For sufficiently small $r$, there is a well-defined inverse branch of $f^{-i}$ mapping $V$ to a domain $U$ in $A^\infty$ touching $\Hq$ along the interval $I_0$. Moreover,
        \begin{equation}
        \label{eqn:approx-rot-claim}
            d_{\nohq}(f^i(x), F^i(x))=O(1) \quad \text{ for all } x \in U.
        \end{equation}
    \end{claim}
    
    \begin{proof}
    Let us fix a small constant $\varepsilon \in \left( 0, \frac{\pi}{2} \right)$. There exists some $M = M(\varepsilon) \in \N$ such that for any interval $I \subset \Hq$ with combinatorial length at most $l_M$,
    \begin{itemize}
        \item[(a)] the domain $P^\infty_{\varepsilon}(I)$ is contained in $A^\infty$, and
        \item[(b)] if $I$ contains no outer critical values of $f$, there is a well-defined inverse branch of $f^{-1}$ mapping $P^\infty_{\varepsilon}(I)$ to a domain touching $\Hq$ on $(f|_{\Hq})^{-1}(I)$.
    \end{itemize}
    We will take $r$ to be sufficiently smaller than $l_M$.
    
    Denote by $V_{-1}$ a univalent lift of $V$ under $f$ that is touching $\Hq$. Since $I_i$ does not contain any outer critical value of $f^i$, $V_{-1}$ touches $\Hq$ precisely along $I_{i-1}$ and so $F \circ f^{-1}$ is the identity map on $I_i$. By Lemma \ref{lem:poincare-nbh}, there is some constant $C>0$ depending on the dilatation of $\phi$ such that 
    \[
    d_{\nohq}\left((f|_{V_{-1}})^{-1}(x), F^{-1}(x)\right) \leq C \quad \text{ for all } x \in V.
    \]
    Therefore, we can take $\varepsilon$ to be small enough and $M$ to be large enough so that $V_{-1}$ is contained in $P_\varepsilon^\infty(I_{-1})$. Property (b) ensures that $V_{-1}$ can again be lifted to a domain $V_{-2}$ touching $\Hq$ along $I_{i-2}$. By the same argument, we have 
    \[
    d_{\nohq}\left((f|_{V_{-2}})^{-2}(x), F^{-2}(x) \right) \leq C \quad \text{ for all }x \in V
    \]
    and again $V_{-2} \subset P_\varepsilon^\infty(I_{-1})$. Inductively, we can define the domains $V_{-2}$, $V_{-3}$, $\ldots$, $V_{-i}$ by pulling back. By (a), the domain $U:= V_{-i}$ is contained in $A^\infty$.
    \end{proof}
    
    We can ensure that $f^i:U \to V$ has bounded distortion by shrinking $I_i$ (and thus $V$ and $U$) by a little bit. By (\ref{eqn:approx-rot-claim}), we conclude that $f^i: U \to V$ is an approximate rotation. It remains to prove \ref{approx-3}. Properties (i) and (ii) imply that 
    \begin{equation}
    \label{eqn:extremal-length-w}
        \mathcal{L}_{V}(J_n, \partial V \backslash \Hq) \succ 1
    \end{equation}
    and that
    \begin{equation}
    \label{eqn:boundary-of-V}
        |x-f^i(w)| \asymp |I_i| 
        \quad \text{ for all } x \in \partial V \backslash \Hq.
    \end{equation}
    in the $w$-coordinate. Since $f^{-i}$ is uniformly quasiconformal on $V$, (\ref{eqn:extremal-length-w}) implies the first estimate in \ref{approx-3}. Since $f^{-i}$ is uniformly quasisymmetric on $V$, (\ref{eqn:boundary-of-V}) implies $\dist(z,\partial U \backslash \Hq) \asymp |I_0|$. Since $|I_0| \asymp l_n \asymp r$, the second estimate in \ref{approx-3} follows.
\end{proof}

Instead of an approximate rotation on one side of $\Hq$, we can also consider a two-sided approximate rotation, yielding the following lemma. See Figure \ref{fig:approx-rot} for an illustration. 

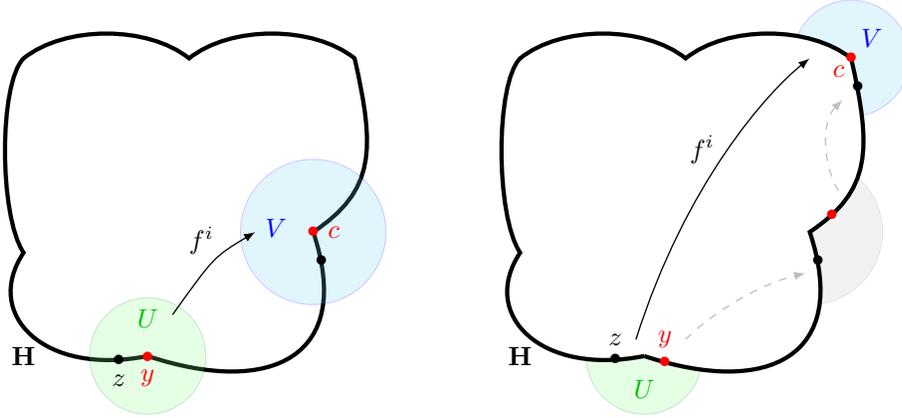
\begin{figure}
 \centering
    \begin{tikzpicture}[scale=1.1]
        \coordinate (t1) at (-5,2.6) {};
        \coordinate (t2) at (-3,2.6) {};
        \coordinate (t3) at (-1,2.6) {};
        \coordinate (t4) at (1,2.6) {};
        \coordinate (t5) at (3,2.6) {};
        \coordinate (t6) at (5,2.6) {};
        \coordinate (yl) at (-3.5,-1) {};
        \coordinate (yr) at (2.5,-1) {};
        \coordinate (cl) at (-1.5,0.5) {};
        \coordinate (cr) at (4.5,0.5) {};
        \coordinate (bl) at (-5,0.25) {};
        \coordinate (br) at (1,0.25) {};

        % lines and regions
        \filldraw[green!50!black,fill=green!50!white,opacity=0.2] (yr) circle (20pt);
        \filldraw[gray,fill=gray!50!white,opacity=0.2] (cr) circle (25pt);
        \filldraw[blue,fill=cyan!50!white,opacity=0.2] (t6) circle (20pt);
        \draw[ultra thick] (yl) .. controls (-2,-1.5) and (-1,-1) .. (cl) .. controls (-0.75,1) and (-0.75,1.5) .. (t3) .. controls (-1.5,3) and (-2.5,3) .. (t2) .. controls (-3.5,3) and (-4.5,3) .. (t1) .. controls (-5.3,2.3) and (-5.3,0.7) .. (bl) .. controls (-5.5,-0.5) and (-4.75,-1.25) .. (yl);
        \draw[ultra thick, fill=white] (yr) .. controls (4,-1.5) and (5,-1) .. (cr) .. controls (5.25,1) and (5.25,1.5) .. (t6) .. controls (4.5,3) and (3.5,3) .. (t5) .. controls (2.5,3) and (1.5,3) .. (t4) .. controls (0.7,2.3) and (0.7,0.7) .. (br) .. controls (0.5,-0.5) and (1.25,-1.25) .. (yr);
        \filldraw[green!50!black,fill=green!50!white,opacity=0.2] (yl) circle (20pt);
        \filldraw[blue,fill=cyan!50!white,opacity=0.2] (cl) circle (25pt);
        \draw[line width=0.5pt,-latex] (-3.2,-0.5) .. controls (-2.7,0.2) .. (-2.2,0.5);
        \node [black, font=\bfseries] at (-2.85,0.4) {$f^i$};
        \draw[gray!50!white,line width=0.5pt,-latex,dashed] (3,-0.8) .. controls (3.4,-0.4) and (4,-0.1) .. (4.45,0);
        \draw[gray!50!white,line width=0.5pt,-latex,dashed] (4.85,1) .. controls (4.7,1.2) and (4.6,1.8) .. (4.9,2.1);
        \draw[line width=0.5pt,-latex] (2.4,-0.8) .. controls (3,1.1) and (4,2.2) .. (4.5,2.6);
        \node [black, font=\bfseries] at (3.2,1.5) {$f^i$};

        % bullets
        \node [red, font=\bfseries] at (-3.5,-1.02) {$\bullet$};
        \node [red, font=\bfseries] at (cl) {$\bullet$};
        \node [black, font=\bfseries] at (-3.85,-1.06) {$\bullet$};
        \node [black, font=\bfseries] at (-1.4,0.15) {$\bullet$};
        \node [red, font=\bfseries] at (2.75,-1.08) {$\bullet$};
        \node [red, font=\bfseries] at (4.77,0.7) {$\bullet$};
        \node [red, font=\bfseries] at (t6) {$\bullet$};
        \node [black, font=\bfseries] at (2.15,-1.04) {$\bullet$};
        \node [black, font=\bfseries] at (4.6,0.15) {$\bullet$};
        \node [black, font=\bfseries] at (5.08,2.25) {$\bullet$};        

        % labels
        \node [red, font=\bfseries] at (-3.5,-1.27) {$y$};
        \node [red, font=\bfseries] at (-1.25,0.5) {$c$};
        \node [black, font=\bfseries] at (-3.85,-1.3) {$z$};
        \node [red, font=\bfseries] at (2.75,-0.8) {$y$};
        \node [red, font=\bfseries] at (4.85,2.45) {$c$};
        \node [black, font=\bfseries] at (2.15,-0.8) {$z$};
        \node [green!75!black, font=\bfseries] at (-3.5,-0.55) {$U$};
        \node [blue, font=\bfseries] at (-1.95,0.55) {$V$};
        \node [green!75!black, font=\bfseries] at (2.5,-1.4) {$U$};
        \node [blue, font=\bfseries] at (5.25,2.85) {$V$};
        \node [black, font=\bfseries] at (-5,-1) {$\Hq$};
        \node [black, font=\bfseries] at (1,-1) {$\Hq$};
\end{tikzpicture}

    \caption{On the left, we have a two-sided approximate rotation as stated in Lemma \ref{approx-rot-02}. The iteration stops once the forward orbit of $z$ is close to either an inner or outer critical point. On the right, we have a one-sided approximate rotation described in Lemma \ref{approx-rot-01}. In this case, the iteration stops when it reaches an outer critical point.}
    \label{fig:approx-rot}
\end{figure}

\begin{lemma}[Two-sided approximate rotation]
\label{approx-rot-02}
    There exists a constant $C>0$ such that the following holds. Given any point $z \in \Hq$ and sufficiently small scale $r>0$, there is an $C$-approximate rotation 
    $$
    f^i : (U,y) \to (V, c)
    $$
    such that $y$ lies on $\Hq$, $c \in \Hq$ is a critical point of $f$, and $(U,y)$ is a pointed disk that well contains the interval $[y,z] \subset \Hq$ and has bounded shape and diameter $\asymp r$.
\end{lemma}

This lemma is a generalization of \cite[Theorem 3.4]{McM98} which was originally formulated for bounded type quadratic Siegel disks.

\begin{proof}
    We can adapt the same proof as the previous lemma with a few modifications. Take $i$ to be the smallest number such that $J_i$ contains either an outer or an inner critical point $c \in \Hq$. Define the half-neighborhood $P^0_a(I) \subset Y^0$ of an interval $I \subset \Hq$ in an analogous way, and define the Jordan domain $V$ by gluing $P^\infty_{\frac{\pi}{2}}(I_i)$ and $P^0_{\frac{\pi}{2}}(I_i)$. The rest of the proof resumes as before. Replace $U$ with a smaller disk (e.g. a hyperbolic ball $\D_U(y,R)$ for some definite radius $R>1$) so that $(U,y)$ has bounded shape.
\end{proof} 

For each $\bullet \in \{0,\infty\}$, the round annulus $\phi^{-1}(A^\bullet)$ admits a canonical radial foliation connecting its two boundary components. For each $z \in \Hq$, we denote by $\gamma_z^\bullet$ the unique proper curve in $A^\bullet$ such that $\phi(\gamma_z^\bullet)$ is the radial leaf with endpoint $\phi(z)$.

\begin{lemma}[Local preimages of $\Hq$]
\label{lem:preim}
    For every outer critical point $c \in \Hq$ of $f$, there are $2d_\infty(c)-2$ pairwise disjoint open quasiarcs $\Gamma_1^\infty,\ldots, \Gamma^\infty_{2d_\infty(c)-2}$ in $A^\infty$ which are all mapped into $\Hq$ by $f$ and attached to $c$ at one of its endpoints. Every point $z$ on $\Gamma^\infty_1 \cup \ldots \cup \Gamma^\infty_{2d_\infty(c)-2}$ satisfies
    \[
        \dist_\nohq(z, \gamma_c^\infty) = O(1).
    \]
\end{lemma}

In Case \ref{case-H}, this lemma also works for inner critical points.

\begin{proof}
    Let $d(c) := d_0(c) + d_\infty(c)-1$ be the local degree of $f$ at $c$. There exists an open disk neighborhood $Q \subset A$ of $c$ on which $f$ is a degree $d(c)$ covering map branched only at $c$. When $Q$ is sufficiently small, the map $f$ on $Q$ is of the form $h(z)^{d(c)} + f(c)$ for some univalent map $h: Q \to \RS$ with bounded distortion sending $c$ to $0$. 
    
    The disk $Q$ can be picked such that $\phi(f(Q))$ is a round disk orthogonal to $\T$. Then, the preimage of the interval $\Hq \cap f(Q)$ will consist of $\Hq \cap Q$ as well as pairwise disjoint open quasiarcs 
    \[
    \Gamma_1^0, \: \ldots, \Gamma_{2d_0(c)-2}^0, \quad \Gamma_1^\infty, \: \ldots, \Gamma^\infty_{2d_\infty(c)-2}
    \]
    where each $\Gamma_i^\bullet$ is contained in $A^\bullet \cap Q$ and connects the critical point $c$ to a point in $A^\bullet \cap \partial Q$.
    
    Let us pick a point $z$ on $\Gamma_i^\infty$ for some $i$. Let $w$ be a point on $\Hq$ closest to $z$, i.e. $|z-w| = \dist(z,\Hq)$. Note that the hyperbolic metric of $\nohq$ at $z$ is comparable to $\dist(z,\Hq)^{-1}$. As such, in order to prove the lemma, it is sufficient to show that 
    \begin{equation}
    \label{cone-cdn}
    |z-w| \succ |z-c|.
    \end{equation}
    
    Before we do so, we will introduce another disk neighborhood $\tilde{Q}$ of $c$ in a similar way as $Q$, except that $\tilde{Q}$ is larger than $Q$ and $\modu\left(\tilde{Q}\backslash \overline{Q}\right) \asymp 1$. 
    If $\tilde{Q}$ does not contain $w$, then by Teichm\"uller estimates \cite[\S3]{A06},
    \[
    |z-c| \leq \diam\left( Q \right) \prec \dist\left(\partial \tilde{Q}, \partial Q\right) \leq |z-w|
    \]
    and we are done. 
    
    Suppose instead that $w$ lies inside of $\tilde{Q}$. Since $\Hq$ is a quasicircle, the ratio of the distance between $f(z)$ and $f(w)$ to the diameter of the interval $[f(z),f(w)] \subset \Hq$ must be bounded above by some definite constant. In particular, since the critical value $f(c)$ lies on $[f(z),f(w)]$, then
    \[
    |f(z)-f(w)| \succ |f(z)-f(c)|.
    \]
    Consider the univalent map $h: \tilde{Q} \to \RS$ described previously. The estimate above can be rewritten as
    \[
    |h(z)^{d(c)}- h(w)^{d(c)}| \succ |h(z)|^{d(c)},
    \]
    which implies that
    \[
    |h(z) - h(w)| \succ |h(z)|.
    \]
    Since $h$ has bounded distortion on $\tilde{Q}$, this estimate implies (\ref{cone-cdn}).
\end{proof}

\begin{corollary}
\label{approx-rot-to-preim}
    Given any point $z$ in the outer collar $A^\infty$, if $z$ is sufficiently close to $\Hq$, there exists an approximate rotation $f^i : U \to V$ such that $U$ contains $z$, $V$ contains a hyperbolic ball $D \subset \nohq$ of radius $\asymp 1$ centered at some point on $f^{-1}(\Hq) \backslash \Hq$, and 
    \[
    \dist_\nohq(f^i(z), D) = O(1).
    \]
\end{corollary}

Again, in Case \ref{case-H}, this corollary also holds for points inside the inner collar.

\begin{proof}
    Let $w \in \Hq$ be the unique point such that $z$ lies on the radial segment $\gamma^\infty_w$. By Lemma \ref{approx-rot-01}, there is an approximate rotation $f^i :U \to V$ such that
    \begin{enumerate}[label=\textnormal{(\roman*)}]
        \item $U$ and $V$ are disks in $A^\infty$ and $z \in U$,
        \item there is some interval $[w,y] \subset \partial U \cap \Hq$ such that $c := f^i(y)$ is an outer critical point and $\mathcal{L}_U([w,y],\partial U \backslash \Hq) \succ 1$,
        \item $|w-y| = O(|w-z|)$.
    \end{enumerate} 
    Refer to Figure \ref{fig:approx-rot-to-preim}. Clearly, (iii) implies that 
    \[
    \dist_{\nohq}(z,\gamma^\infty_{y}) = O(1).
    \] 
    Since $f^i$ is an approximate isometry on $U$, then 
    \[
    \dist_{\nohq}(f^i(z),\gamma^\infty_{c}) = O(1).
    \]
    Because of Lemma \ref{lem:preim} and the fact that $c$ is an outer critical point, we have
    \begin{equation}
    \label{eqn:hyp-ball-estimate}
    \dist_\nohq(f^i(z), f^{-1}(\Hq)) = O(1).
    \end{equation}
    From (ii), we have $\mathcal{L}_V([f^i(w),c], \partial V\backslash \Hq) \succ 1$. Together with (\ref{eqn:hyp-ball-estimate}), we conclude that $V$ contains a hyperbolic ball $D$ with the desired properties.
\end{proof}

\begin{figure}
 \centering
    \begin{tikzpicture}[scale=1.1]
        \draw[gray,fill=gray!10!white] (-5,0) .. controls (-5,2.7) and (-1,2.7) .. (-1,0) -- (-4,0);
        \draw[gray,fill=gray!10!white] (1,0) .. controls (1,2.7) and (5,2.7) .. (5,0) -- (1,0);
        \filldraw[green, fill=green!50!white, opacity=0.3] (2.7,0.78) circle (12pt);
        \draw[ultra thick] (-5.5,0) -- (-0.5,0);
        \draw[ultra thick] (0.5,0) -- (5.5,0);
        \draw[red!80!white] (-3.8,0) -- (-3.8,2.4);
        \draw[gray!50!white] (3.4,0) -- (3,1.5);
        \draw[gray!50!white] (3.4,0) -- (3.8,1.5);
        \draw[gray!50!white] (3.4,0) .. controls (3,0.35) .. (2.6,1);
        \draw[gray!50!white] (3.4,0) .. controls (3.8,0.35) .. (4.2,1);
        
        \node [black, font=\bfseries] at (-5.2,-0.3) {$\Hq$};
        \node [black, font=\bfseries] at (5.2,-0.3) {$\Hq$};
        
        \node [red!80!white, font=\bfseries] at (-4.1,2.2) {$\gamma^\infty_w$};
        \node [red, font=\bfseries] at (-3.8,1.2) {$\bullet$};
        \node [red, font=\bfseries] at (-3.6,1.2) {$z$};
        \node [red, font=\bfseries] at (-3.8,-0.02) {$\bullet$};
        \node [red, font=\bfseries] at (-3.8,-0.25) {$w$};
        \node [black, font=\bfseries] at (-2.6,-0.02) {$\bullet$};
        \node [black, font=\bfseries] at (-2.6,-0.27) {$y$};
        
        \node [red, font=\bfseries] at (2.3,-0.02) {$\bullet$};
        \node [red, font=\bfseries] at (2.3,-0.28) {\small $f^i(w)$};
        \node [red, font=\bfseries] at (2.3,1.2) {$\bullet$};
        \node [black, font=\bfseries] at (3.4,-0.02) {$\bullet$};
        \node [black, font=\bfseries] at (3.4,-0.25) {$c$};
        \node [black, font=\bfseries] at (2.7,0.8) {$\bullet$};
        \node [red, font=\bfseries] at (2.3,1.55) {\small $f^i(z)$};
        
        \node [gray, font=\bfseries] at (-1.6,0.50) {$U$};
        \node [gray, font=\bfseries] at (1.60,0.50) {$V$};
        \node [green!80!black, font=\bfseries] at (2.7,0.52) {\small $D$};
        \draw[line width=0.5pt,-latex] (-1,1) .. controls (-0.5,1.4) and (0.5,1.4) .. (1,1);
        \node [black, font=\bfseries] at (0,1) {$f^i$};
\end{tikzpicture}

    \caption{The construction in the proof of Corollary \ref{approx-rot-to-preim}.}
    \label{fig:approx-rot-to-preim}
\end{figure}
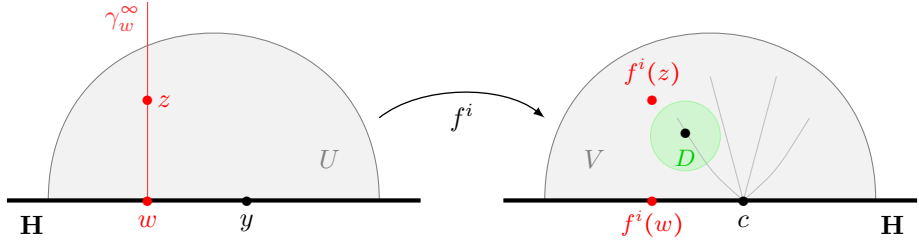

\section{Rigidity of J-rotational rational maps}
\label{sec:rat-map}

Throughout this section, we will consider a rational map $f: \RS \to \RS$ that is J-rotational of bounded type. Our goal here is to prove Theorem \ref{thm:NILF-main}. 

\subsection{J-rotationality}

Denote by $\mathcal{C}(f)$ be the set of critical points of $f$. The intersection of the Julia set $J(f)$ and the postcritical set $P(f)$ of $f$ contains some number, say $m \geq 1$, of pairwise disjoint rotation quasicircles $X_1$, $X_2$, $\ldots$, $X_m$. We will denote
\[
    \PJ = X_1 \cup X_2 \cup \ldots \cup X_m.
\]
Proposition \ref{trichotomy} guarantees that each $X_j$ contains a critical point of $f$, and is either a Herman curve or a boundary component of a rotation domain.

Let us denote by $\mathbf{S}$ the set of the critical and postcritical points of $f$ in its Julia set that are outside of any rotation curve. Hence, we can write
\[
    \left( \mathcal{C}(f) \cup  P(f) \right) \cap J(f) = \PJ \cup \mathbf{S}.
\]
The assumption that $f$ is J-rotational implies that $\mathbf{S}$ is a finite, possibly empty, set. 

\begin{lemma}
\label{lem:parabolic-or-repelling}
    Every point in $\mathbf{S}$ is eventually mapped to either a periodic point in $\mathbf{S}$ or a rotation curve in $\PJ$. Every periodic point in $\mathbf{S}$ is either parabolic or repelling in nature.
\end{lemma}

\begin{proof}
    The first assertion follows from J-rotationality after passing to a higher iterate. Consider a periodic point $s$ in $\mathbf{S}$. Since $s$ is contained in the Julia set, $s$ is either repelling or parabolic or Cremer (irrationally indifferent). Suppose for a contradiction that $s$ is Cremer. A theorem of Ma\~ne \cite{Ma93} asserts that $s$ is contained in the omega limit set $\omega(c)$ of a critical point $c$ that is recurrent, that is, $c \in \omega(c)$. The critical point $c$ has to be in $J(f)$, hence it is also in $\mathbf{S}$. By recurrence, $c$ has an infinite forward orbit, but this contradicts the finiteness of $\mathbf{S}$.
\end{proof}

Note that for all $n \geq 1$, $J(f^n) = J(f)$ and any invariant line field of $f$ on $J(f)$ is an invariant line field of $f^n$ on $J(f^n)$. As such, by replacing $f$ with a higher iterate, we will assume from now on that every periodic point in the postcritical set $P(f)$ is a fixed point and every component of $\PJ$ is fixed by $f$.

\begin{remark}
    We do not assume that the Fatou set $F(f)$ is non-empty. Indeed, it is possible that $F(f)$ is empty and $P(f)$ consists of a single trivial Herman curve. Here is one way to construct such an example. Pick a real rational map $g$ such that all of its critical points in $\RS$ lie along $\hat{\R} :=\R \cup \{\infty\}$ and that $g$ restricts to an orientation-preserving self homeomorphism of $\hat{\R}$. (For example, $g(z)=z^d$ for any odd $d\geq 3$.) Pick any M\"obius transformation $\psi$ sending $\hat{\R}$ onto the unit circle $\T$. By standard monotonicity considerations, for any irrational $\theta \in (0,1)$, there is a unique $\alpha \in [0,1)$ such that the rational map $f(z) := e^{2\pi i \alpha} \psi \circ g \circ \psi^{-1}(z)$ restricts to an analytic multicritical circle map with rotation number $\theta$. Since all critical points lie on $\T$, $f$ has no attracting basins. The rotation curve $\T$ must be a Herman curve because every critical point in $\T$ is both an inner and outer critical point. This also rules out the existence of rotation domains. Therefore, the Fatou set of $f$ is empty.
\end{remark}

\subsection{Thick thin decomposition}
\label{ss:decomposition}

Let  
\[
\Omega := \RS \backslash \left(\mathcal{C}(f) \cup P(f)\right).
\]
We shall equip the open set $\Omega$ with the hyperbolic metric $\rho_\Omega(z) |dz|$. Note that $f^n: f^{-n}(\Omega) \to \Omega$ is an unbranched covering map for all $n \geq 1$.

For any point $z$ in $\Omega$, denote by $\| f'(z) \|$ the norm of the derivative of $f$ at $z$ with respect to the hyperbolic metric of $\Omega$.

\begin{lemma}
\label{lem:expansion}
    Consider a point $z$ in $\Omega$.
    \begin{enumerate}[label=\textnormal{(\arabic*)}]
        \item If $f(z)$ is also contained in $\Omega$, then $\| f'(z) \| \geq 1$.
        \item If $z$ is in $J(f)$ and its forward orbit stays entirely within $\Omega$, then
    \[
        \| (f^n)'(z) \| \to \infty \quad \text{ as } n \to \infty.
    \]
    \end{enumerate}
\end{lemma}

\begin{proof}
    (1) follows from a standard application of Schwarz lemma to the inclusion map $f^{-1}(\Omega) \hookrightarrow \Omega$. For every point $z$ in $J(f)$, the distance $r_n$ between $z$ and $f^{-n}\left(P(f) \cup \mathcal{C}(f)\right)$ decreases to zero as $n \to \infty$. At $z$, the inclusion map $\iota_n : f^{-n}(\Omega) \to \Omega$ is a contraction with respect to the hyperbolic metrics of the domain and range by a factor of $C(r_n)$ for some function $C(r)$ which tends to zero as $r \to 0$. By composing $\iota_n^{-1}$ with $f^n$, we obtain $\| (f^n)'(z) \| \geq C(r_n)^{-1} \to \infty$.
\end{proof}

For every rotation quasicircle $X_j$, let us pick a collar $A_j$ of $f|_{X_j}$ as described in \S\ref{ss:approx-rot}. In particular, we would want $A_j$ to be skinny enough such that it is compactly contained in $\Omega \cup X_j$. Denote the union of all collars by 
\[
\collar := A_1 \cup A_2 \cup \ldots \cup A_m.
\]

For every point $s$ in $\mathbf{S}$, let us fix a disk neighborhood $D_s$ of $s$ small enough such that they are disjoint from each other, as well as from $\collar$. Consider an $\varepsilon$-neighborhood $N_\varepsilon$ of $J(f)$. We define the \emph{thick part of} $\Omega$ by 
\begin{equation}
    \label{eqn:thicc}
    \Othick := N_\varepsilon \backslash \left( \mathbf{X} \cup \bigcup_{s \in \mathbf{S}} D_s \right).
\end{equation}
We select $\varepsilon>0$ to be small enough such that $\Othick$ does not contain any critical and postcritical points of $f$ that lie in the Fatou set $F(f)$. By making each $A_j$ skinnier, we also assume that 
\[
\collar \subset \Othick.
\]

\begin{lemma}
\label{1/d-metric}
    For every point $z$ in $\Othick$,
    \[
    \rho_\Omega(z) \asymp \frac{1}{\dist(z,\PJ)}.
    \]
    If $z$ lies in some collar $A_j$, the denominator can be replaced with $\dist(z,X_j)$.
\end{lemma}

\begin{proof}
    Since $\Omega_{\text{thick}} \backslash \collar$ is compactly contained in $\Omega$, then $\rho_\Omega(z) \asymp1 \asymp \frac{1}{\dist(z,\PJ)}$ for all $z \in \Omega_{\text{thick}}\backslash \collar$. Hence, it suffices to consider points $z$ lying in $A_j\backslash X_j$ for some $j$.
    
    Let us equip the connected component $E$ of $\RS \backslash X_j$ containing $z$ with the hyperbolic metric. A standard application of Koebe quarter theorem to any Riemann mapping $(\D,0) \to (E,z)$ yields the estimate $\rho_E(z) \succ \frac{1}{\dist(z,X_j)}$. Then, we apply Schwarz Lemma to the inclusion map on the connected component of $\Omega$ containing $z$ into $E$ in order to obtain  $\rho_\Omega(z) \geq \rho_E(z)$, which gives us the estimate $\rho_\Omega(z) \succ \frac{1}{\dist(z,X_j)}$.

    As we apply Schwarz lemma to the inclusion map $\D\left(z, \dist(z, P(f))\right) \hookrightarrow \Omega$, we also obtain the estimate $\rho_\Omega(z) \prec \frac{1}{\dist(z, P(f))}$. We conclude the proof with the observation that for every $z \in A_j$, $\dist(z,X_j) \asymp \dist(z, \PJ) \asymp \dist(z, P(f))$.
\end{proof}

In what follows, we will pay closer attention to certain compact subsets of the Julia set which behave nicely under iteration with respect to the hyperbolic metric of $\Omega$.

\begin{definition}
\label{def:little-julia-set}
    Define 
    \[
    \Jthick := \{ z \in J(f) \: : \: f^n(z) \in \overline{\Othick} \text{ for all } n \geq 0\}.
    \]
    For every $j \in \{1,\ldots,m\}$, define the \emph{$j$\textsuperscript{th} local Julia set} of $f$ to be
    \[
    J^{\textnormal{loc}}_j := \{ z \in J(f) \: : \: f^n(z) \in \overline{A_j} \text{ for all } n \geq 0\}.
    \]
\end{definition}

Clearly, $\Jthick$ contains each of the local Julia sets.

\subsection{Visiting a critical point from $\Jthick$}
\label{ss:nearby-critical-visit}

The key ingredient towards proving Theorem \ref{thm:NILF-main} is the following theorem. 

\begin{theorem}[Nearby critical visits]
\label{critical}
    For every point $z \in \Jthick$ and scale $r>0$, there exist an integer $i\geq 0$ and a pair of pointed disks $(U,y)$ and $(V,c)$ such that
    \begin{enumerate}[label=\textnormal{(\arabic*)}]
        \item $f^i: (U,y) \to (V,c)$ is a univalent map with bounded distortion,
        \item $(U,y)$ has bounded shape with diameter $\asymp r$,
        \item $|y-z| = O(r)$, and
        \item $c$ is a critical point of $f$ located in $\PJ$.
    \end{enumerate}
\end{theorem}

Theorem \ref{critical} is an extension of \cite[Theorem 3.2]{McM98}, which was originally stated in the context of bounded type quadratic Siegel disks. The proof relies on the approximate rotation mechanism introduced in \S\ref{ss:approx-rot}, which is made compatible with the hyperbolic metric of $\Omega$ thanks to Lemma \ref{1/d-metric}.

\begin{lemma}
\label{visiting-from-jf}
    For every point $z$ in $\Jthick\backslash \PJ$, there is a univalent map 
    $$
    f^i: (B,x) \to (V,c)
    $$
    such that $i\geq 0$, $c$ is a critical point of $f$ in $\PJ$, and $B$ is a hyperbolic ball in $\Omega$ of radius $r \asymp 1$ centered at a point $x$ in $\Omega$ satisfying $d_{\Omega}(x,z) = O(1)$. 
\end{lemma}

\begin{proof}
    For every critical point $c$ on $\PJ$, consider two nested disk neighborhoods $Q_{c} \subset \tilde{Q}_c$ of $c$ with the following properties.
    \begin{enumerate}[label=\textnormal{(\roman*)}]
        \item $Q_{c}$ is compactly contained in $\tilde{Q}_c$ and $\text{mod}(\tilde{Q}_c \backslash \overline{Q_c}) \asymp 1$.
        \item The map $f$ is a covering map from $Q_{c}$ and $\tilde{Q}_c$ onto their respective images, branched only at $c$. 
        \item The map $f|_{\tilde{Q}_c}$ can be written as $f(z)=h(z)^l + f(c)$ where $l$ is the local degree of $c$ and $h$ is a univalent map of bounded distortion.
    \end{enumerate} 
    By taking $\tilde{Q}_c$ to be sufficiently small, we can further assume that the disks $\tilde{Q}_c$ are pairwise disjoint and are contained in $\collar$. Let $\mathbf{Q}$ be the union of disks $Q_{c}$ across every critical point $c \in \PJ$ of $f$. For brevity, let us also denote the strict preimage of $\PJ$ by $\PJ^{-1}:= f^{-1}(\PJ) \backslash \PJ$. 
    
    Pick a point $z$ in $\Jthick\backslash \PJ$. We will split into four cases.
    
    \vspace{0.1in}
    
    \noindent \textbf{Case 1:} $z \in \PJ^{-1} \cap \mathbf{Q}$.
    
    Suppose $z \in Q_{c_0}$ for some critical point $c_0 \in \PJ$. Let $z'$ be the unique point in the intersection $Q_{c_0} \cap \PJ$ such that $f(z)=f(z')$. By Lemmas \ref{approx-rot-02} and \ref{lem:preim}, there is an approximate rotation 
    \[
    f^i: (U',x') \to (V,c)
    \]
    such that $i \geq 1$, $c$ is a critical point of $f$, and $(U',x')$ is a pointed disk in $\tilde{Q}_{c_0} $ that avoids $\PJ^{-1}$, well contains the interval $[x',z']$, and
    \begin{equation}
    \label{eqn:00}
    \rin(U',x') \asymp |x'-c_0|.
    \end{equation}
    
    Let $(U,x)$ be the pointed disk containing $z$ such that $f(U) = f(U')$ and $f(x)=f(x')$. See Figure \ref{fig:inside-p2}. Since both $U$ and $U'$ are contained in $\tilde{Q}_{c_0}$, there is a univalent map $g: (U,x,z) \to (U',x',z')$ of bounded distortion such that $f \circ g = f$ on $U$. Therefore, $U$ avoids $\PJ$, well contains the interval $[x,z]$, and by (\ref{eqn:00}),
    \[
    \rin(U,x) \asymp |x-c_0| \geq \dist(x,\PJ).
    \]
    By Lemma \ref{1/d-metric}, this implies that $U$ contains a hyperbolic ball $B \subset \Omega$ of definite radius centered at $x$. Therefore, $f^i: (B,x) \to (V,c)$ is the desired univalent map.

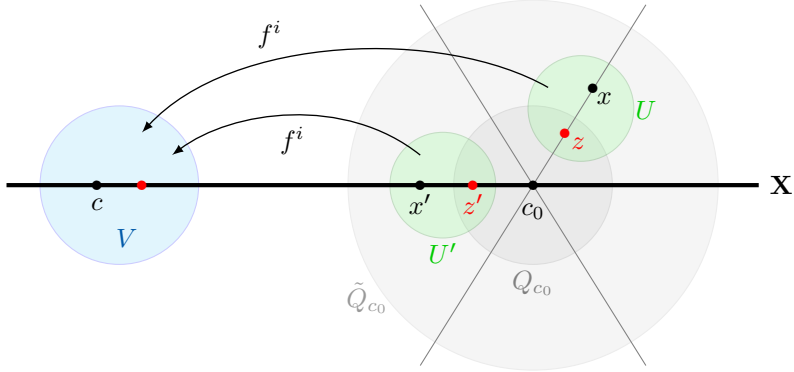
\begin{figure}
\centering
    \begin{tikzpicture}
        \filldraw[blue,fill=cyan!50!white,opacity=0.2] (-3.5,0) circle (30pt);
        \filldraw[gray,fill=gray!50!white,opacity=0.15] (2,0) circle (70pt);
        \filldraw[gray,fill=gray!50!white,opacity=0.2] (2,0) circle (30pt);
        \filldraw[green!50!black,fill=green!50!white,opacity=0.2] (0.8,0) circle (20pt);
        \filldraw[green!50!black,fill=green!50!white,opacity=0.2] (2.636,1.02) circle (20pt);
        \draw[ultra thick] (-5,0) -- (5,0);
        \draw[gray] (0.5,2.4) -- (3.5,-2.4);
        \draw[gray] (0.5,-2.4) -- (3.5,2.4);
        
        \node [black, font=\bfseries] at (5.3,0) {$\PJ$};
        \node [black, font=\bfseries] at (-3.8,-0.02) {$\bullet$};
        \node [black, font=\bfseries] at (-3.8,-0.25) {$c$};
        \node [red, font=\bfseries] at (-3.2,-0.02) {$\bullet$};
        \node [black, font=\bfseries] at (0.5,-0.02) {$\bullet$};
        \node [black, font=\bfseries] at (0.5,-0.25) {$x'$};
        \node [red, font=\bfseries] at (1.2,-0.02) {$\bullet$};
        \node [red, font=\bfseries] at (1.2,-0.25) {$z'$};
        \node [black, font=\bfseries] at (2,-0.02) {$\bullet$};
        \node [black, font=\bfseries] at (2,-0.35) {$c_0$};
        \node [red, font=\bfseries] at (2.424,0.68) {$\bullet$};
        \node [red, font=\bfseries] at (2.6,0.55) {$z$};
        \node [black, font=\bfseries] at (2.795,1.27) {$\bullet$};
        \node [black, font=\bfseries] at (2.95,1.15) {$x$};
        
        \node [black, font=\bfseries] at (-1.2,0.6) {$f^i$};
        \node [black, font=\bfseries] at (-1.5,2) {$f^i$};
        \node [cyan!30!blue, font=\bfseries] at (-3.4,-0.7) {$V$};
        \node [gray, font=\bfseries] at (2,-1.3) {$Q_{c_0}$};
        \node [gray!80!white, font=\bfseries] at (-0.2,-1.5) {$\tilde{Q}_{c_0}$};
        \node [green!80!black, font=\bfseries] at (0.8,-0.9) {$U'$};
        \node [green!80!black, font=\bfseries] at (3.5,1) {$U$};
        
        \draw[line width=0.5pt,-latex] (0.5,0.4) .. controls (-0.5,1.2) and (-2,1) .. (-2.8,0.4);
        \draw[line width=0.5pt,-latex] (2.2,1.3) .. controls (0.8,2.1) and (-2,2) .. (-3.2,0.7);
    \end{tikzpicture}
    \caption{Case 1 in the proof of Lemma \ref{visiting-from-jf}.}
    \label{fig:inside-p2}
\end{figure}
    
    \vspace{0.1in}
    
    \noindent \textbf{Case 2:} $z \in \PJ^{-1} \backslash \mathbf{Q}$. 
    
    The connected component $W$ of $\Omega$ containing $z$ must have at least one boundary component contained in $\PJ$, say $X_j$. If $X_j$ were the boundary component of a rotation domain and $A_j^0$ were contained in $W$, then $W$ would lie entirely within the rotation domain, contradicting the assumption that $z$ lies in $J(f)$. In all the other cases, $W$ contains a connected component of $\PJ^{-1}$ attached to a critical point of $X_j$. Since $W \cap \PJ^{-1} \backslash \mathbf{Q}$ is compactly contained in $W$, there is a point $z'$ in $W \cap \PJ^{-1} \cap \mathbf{Q}$ such that $\dist_\Omega(z,z') = O(1)$. This reduces us to Case 1.

    \vspace{0.1in}
    
    \noindent \textbf{Case 3:} $z \not\in \collar$.

    Similar to Case 2, since $\Jthick \backslash \collar$ is a compact subset of $\Omega$, there is a point $z' \in \PJ^{-1}$ satisfying $d_\Omega(z,z') = O(1)$. This reduces us to Cases 1 and 2.
    
    \vspace{0.1in}
    
    \noindent \textbf{Case 4:} $z \in \collar$.

    By Corollary \ref{approx-rot-to-preim}, there exists an approximate rotation 
    \[
    f^i:(U,w) \to (U',w')
    \]
    such that $w'$ is in $\PJ^{-1}$ and that $w$ satisfies $d_\Omega(w,z) = O(1)$ and $\dist_\Omega(w,\partial U) \succ 1$. By Cases 1 and 2, there also exists a univalent map
    \[
    f^j: \left(B',x'\right) \to \left(V,c\right)
    \]
    such that $B' \subset \Omega$ is a hyperbolic ball of radius $\asymp 1$ centered at $x'$, $c$ is a critical point of $f$ in $\PJ$, and $d_\Omega(w', x') = O(1)$. We can assume that $B'$ is inside of $U'$ by shrinking $B'$ by a little bit. As such, the lift $(f^i|_U)^{-1}(B')$ contains a hyperbolic ball $B$ of radius $\asymp 1$ centered at $x = (f^i|_U)^{-1}(x')$ with distance $d_\Omega(x,z) := O(1)$. Therefore, $f^{i+j}: (B,x) \to (V,c)$ is the desired univalent map.
\end{proof}

To prove Theorem \ref{critical}, we will apply Lemmas \ref{approx-rot-02} and \ref{visiting-from-jf} in a similar manner as the single Siegel case in \cite[\S3]{McM98}. (Compare with \cite[Theorem 8.10]{McM96}.)

\begin{proof}[Proof of Theorem \ref{critical}]
    For any tangent vector $v$ at a point $z$, we denote by $|v|$ the Euclidean length of $v$ and $\| v \|$ the hyperbolic length of $v$ with respect to the hyperbolic metric of $\Omega$ if $z \in \Omega$. If $z$ is outside of $\Omega$, we set $\| v \| = \infty$. By Lemma \ref{1/d-metric}, 
\begin{equation}
\label{lengths}
    \| v \| \asymp \frac{|v|}{\dist(z,\PJ)} \quad \text{ for all } z \in \Othick \text{ and } v \in T_z \RS,
\end{equation}
and the denominator can be replaced by $\dist(z,X_j)$ if $z$ lies in some collar $A_j$.

Let us fix a point $z_0 \in \Jthick$ and a scale $r > 0$. Let $v_0$ be a tangent vector at $z_0$ of length $|v_0|=r$. For every $i \in \N$, let $v_i := df^i_{z_0}(v_0)$ be the pushforward of $v_0$ by $f^i$ at $z_i := f^i(z_0)$. 

Let us fix a small constant $\varepsilon \in (0,1)$ independent of $z$, which will be determined later. By Lemma \ref{lem:expansion}, the proof can be split into the following three distinct cases.

\vspace{0.1in}
    
\noindent \textbf{Case 1:} $1 \leq \| v_0 \| \leq \infty$.

Let $w$ be a point in $\PJ$ closest to $z_0$. By (\ref{lengths}), $|z_0-w| = \dist(z_0,\PJ) = O(r)$. By Lemma \ref{approx-rot-02}, there is an approximate rotation $f^i : (U,y) \to (V,c)$ such that $c$ is a critical point of $f$ on $\PJ$, $|y-w| = O(r)$, and $(U,y)$ has bounded shape with diameter $\asymp r$. This is the univalent map we are looking for.

\vspace{0.1in}
    
\noindent \textbf{Case 2:} There is some $j\geq 1$ such that $\| v_j \|\geq 1$ but $\| v_{j-1} \|\leq \varepsilon$.

By (\ref{lengths}), the distance between $z_{j-1}$ and $\PJ$ satisfies $\dist(z_{j-1},\PJ) \succ \frac{|v_{j-1}|}{\varepsilon}$. Since $\dist(z_{j-1},\mathbf{S}) \succ \dist(z_{j-1}, \PJ)$ as well, then $\Omega$ contains a round disk $D_{j-1}$ centered at $z_{j-1}$ of radius $\asymp \frac{|v_{j-1}|}{\varepsilon}$ on which $f$ is univalent. By Koebe quarter theorem, the image $f(D_{j-1})$ contains another round disk $D_j \subset \Omega$ centered at $z_j$ of radius $\asymp \frac{|v_j|}{\varepsilon}$. Denote by $D_0$ the connected component of $f^{-j}(D_j)$ containing $z_0$. We have a univalent map $f^{j}: (D_0, z_0) \to (D_j, z_j)$.

From Case 1, there is a univalent map $f^i: (U', y') \to (V, c)$ of bounded distortion such that $|y'-z_j| = O(|v_j|)$ and $(U',y')$ has bounded shape with diameter $\asymp |v_j|$. Select $\varepsilon$ to be just small enough such that $U'$ is well contained in $D_j$. Let $(U,y)$ be the lift of $(U',y')$ under the map $f^j|_{D_0}$. Since the inverse branch $(f^j|_{D_0})^{-1}$ has bounded distortion on $U'$, then $|y-z_0|=O(r)$ and $(U,y)$ has bounded shape with diameter $\asymp r$. Therefore, $f^{i+j}: (U, y) \to (V, c)$ is the desired univalent map.
    
\vspace{0.1in}
    
\noindent \textbf{Case 3:} There is some $j \in \N$ such that $\varepsilon < \| v_j \| < 1$.

By Lemma \ref{visiting-from-jf}, there is a univalent map $f^i: (B, x) \to (V,c)$ where $c \in \PJ$ is a critical point of $f$ and $B \subset \Omega$ is a hyperbolic ball of radius $\asymp 1$ centered at a point $x$ which satisfies $d_\Omega(x,z_j) = O(1)$. If $z_j$ is in $\overline{B}$, then set $K' = \overline{B}$; otherwise, set $K' = \overline{B} \cup \gamma$ where $\gamma \subset \Omega$ be the shortest hyperbolic geodesic segment in $\Omega$ connecting $z_j$ and a point on $\partial B$. Let $K$ be the unique lift of $K'$ under $f^j$ containing $z_0$, and let $(U,y)$ be the lift of $(B,x)$ under $f^j|_K$. The map $f^j: (U,y) \to (B, x)$ is univalent.

By the hyperbolic Koebe distortion theorem \cite[Theorem 2.29]{McM94}, since $K'$ has bounded hyperbolic diameter in $\Omega$, the map $f^j$ is approximately an expansion by $\| (f^j)'(z_0) \|$ on $K$. Moreover, since $\| v_j\| \asymp 1$, 
\[
\| (f^j)'(z_0) \| = \frac{\| v_j \|}{\| v_0 \|} \asymp \frac{1}{\| v_0 \|}.
\]
As the hyperbolic inner radius of $B$ about $x$ satisfies $r_{\text{in}, \Omega}(B,x) \asymp 1$, the hyperbolic inner radius of $U$ about $y$ satisfies
\[
r_{\text{in},\Omega}(U,y) \asymp \| v_0 \|.
\]
Also, since $d_\Omega(x,z_j) = O(1)$, we have
\[
d_\Omega(y,z_0) = O(\| v_0 \|).
\]
From (\ref{lengths}) and the two estimates above, we have $\rin(U,y) \asymp r$ and $|y-z_0| = O(r)$. Thus, the map $f^{i+j}: (U,y) \to (V,c)$ is our desired univalent map.
\end{proof}

\subsection{No invariant line fields}
\label{ss:NILF}

Before we show the absence of invariant line fields on the Julia set, we will first restrict to the local Julia sets introduced in Definition \ref{def:little-julia-set}.

\begin{lemma}[No local invariant line fields]
\label{NILF-local}
    For every $j \in \{1,\ldots,m\}$, the local Julia set $J^{\textnormal{loc}}_j$ of $f$ does not carry any invariant line field of $f$.
\end{lemma}

In the proof, we will apply the following proposition by Shen.

\begin{proposition}[{\cite[Proposition 3.2]{She03}}]
\label{shen-criterion}
    Consider a rational function $g$ of degree $\geq 2$ and a forward invariant subset $J$ of $J(g)$. Suppose that for almost every point $x$ in $J$, there is a constant $C>1$, a positive integer $N \geq 2$, and a sequence $h_n :U_n \to V_n$ of holomorphic maps such that:
    \begin{enumerate}[label=\textnormal{(S\textsubscript{\arabic*})}]
        \item\label{S1} $g^i \circ h_n = g^j$ for some $i,j \in \N$;
        \item\label{S2} $U_n$ and $V_n$ are topological disks such that as $n\to \infty$, 
        $$
        \diam(U_n) \to 0 \quad \text{and} \quad \diam(V_n) \to 0;
        $$
        \item\label{S3} $h_n$ is a branched covering map of degree between $2$ and $N$;
        \item\label{S4} there are some critical point $u_n \in U_n$ of $h_n$ and critical value $w_n = h_n(u_n)$ such that both $(U_n,u_n)$ and $(V_n,w_n)$ have $C$-bounded shape;
        \item\label{S5} $U_n$ and $V_n$ are relatively close to $x$, i.e. 
        \[
        \dist(x,U_n) \leq C \diam(U_n) \quad \text{and} \quad \dist(x,V_n) \leq C \diam(V_n).
        \]
    \end{enumerate}
    Then, $g$ admits no invariant line field on $J$.
\end{proposition}

The idea behind this criterion comes from the fact that at almost every point $x$ in $J$, any invariant line field on a small neighborhood around $x$ is almost parallel, but the presence of critical points at small scales would carry non-linearity throughout $J$ and contradict such parallel structure. Shen's criterion was inspired by McMullen's treatment of Feigenbaum maps in \cite{McM94, McM96}.

\begin{proof}[Proof of Lemma \ref{NILF-local}]
    It is sufficient to show that the hypothesis of Proposition \ref{shen-criterion} holds for every point $x$ in $J^{\textnormal{loc}}_j$. (In fact, we will show that the constants $C$ and $N$ can be made uniform in $x$.) There are two cases.
    
    \vspace{0.1in}
    
    \noindent \textbf{Case 1:} $x$ is a critical point of $f$ on $X_j$. 

    In this case, we will take the sequence $\{h_n\}$ to be the first return maps near $x$. 
    
    Let $\{p_k/q_k\}_{k \in \N}$ denote the best rational approximations of the rotation number $\theta_j$ of $f|_{X_j}$, and let $l_k := |p_k - q_k \theta_j|$. Let $\phi$ be a global quasiconformal map conjugating $f|_{X_j}$ with the irrational rotation $R_{\theta_j}|_\T$.
    
    Pick a sufficiently large $n \in \N$ and let $w_n:= f^{q_n}(x)$. By Lemma \ref{counting}, there exists a pair of intervals $I' \subset I''$ in $X_j$ such that
    \begin{enumerate}[label=(\roman*)]
        \item $I'$ contains the level $n$ combinatorial interval centered at $w_n$,
        \item the endpoints of $I'$ split $I''$ into three components each of combinatorial length $\asymp l_n$, and
        \item $I''\backslash I'$ contains no critical values of $f^{q_{n+2}}$.
    \end{enumerate}

    Let $V'$ and $V''$ be the unique pair of disks such that their closures intersect $X_j$ along intervals $I'$ and $I''$ respectively, and that both $\phi(\partial V')$ and $\phi(\partial V'')$ are round circles orthogonal to $\T$. Let us pick $n$ to be large enough such that $V''$ is contained in the collar $A_j$.
    
    Let $V_n$ be the Jordan disk that contains $V'$ and is enclosed by the core curve of the annulus $V'' \backslash \overline{V'}$. Denote by $U'$, $U''$, and $U_n$ the connected component containing $x$ of the preimage under $f^{q_n}$ of $V'$, $V''$ and $V_n$ respectively. By (iii), $f^{q_n} : U'' \backslash \overline{U'} \to V'' \backslash \overline{V'}$ is an unbranched covering map between two annuli.
    
    We claim that $f^{q_n} : U_n \to V_n$ satisfies \ref{S1}-\ref{S5} in Proposition \ref{shen-criterion}. Indeed, \ref{S1} is immediate from the construction. \ref{S3} follows from the fact that, by (ii) and Proposition \ref{bounded-type}, the number of critical points of $f^{q_n}$ on $U''$ is at most some constant $N$ independent of $x$ and $n$.
    
    Take $u_n = x$. By construction, $V'$ is well contained in $V''$ and $U'$ is well contained in $U''$. Since $\partial U_n$ and $\partial V_n$ are core curves of annuli of definite moduli, both $U_n$ and $V_n$ are quasidisks of uniformly bounded dilatation; in particular, they have bounded shape about $x$ and $w_n$ respectively. Thus, \ref{S4} holds.
    
    By construction, \ref{S4} ensures that $\diam(U_n) \asymp \diam(V_n)$. As such, \ref{S2} follows from the fact that $\diam(\phi(V_n)) \asymp l_n \to 0$ as $n \to \infty$, and \ref{S5} follows from $\dist(x,U_n)=0$ and $\dist(x,V_n) \leq \diam(U_n)$. This concludes the proof.
    \vspace{0.1in}
    
    \noindent \textbf{Case 2:} $x$ is not a critical point of $f$ on $X_j$. 
    
    Fix a sequence $\{r_n\}_{n \in \N}$ of small positive real numbers decreasing to $0$. Pick $n \in \N$. By Theorem \ref{critical}, there is a univalent map
    \[
    f^{j_n} : (A_n,u_n) \to \left(\tilde{A}_n,c_n \right)
    \]
    between pointed disks with bounded distortion such that $c_n$ is a critical point on $X_j$, $(A_n,u_n)$ has bounded shape with diameter $\asymp r_n$, and \begin{equation}
    \label{ineq-00}
        |u_n-x| = O(r_n).
    \end{equation}
    
    Let $s_n:= \diam\left(\tilde{A}_n, c_n\right)$; this depends on $r_n$. From Case 1, by appropriately selecting $r_n$, there is some $k_n \in \mathbb{N}$ and some branched covering map
    \[
    f^{k_n}: \left(\tilde{U}_n, c_n\right) \to \left(\tilde{V}_n, v_n\right)
    \]
    of degree at most some constant $N$ independent of $x$ and $n$ such that $\left(\tilde{U}_n,c_n\right)$ and $\left(\tilde{V}_n,v_n\right)$ are pointed disks compactly contained in $\tilde{A}_{n}$ with bounded shape and diameter $\asymp s_n$, and $\dist\left(c_n,\tilde{V}_n\right) = O(s_n)$.
    
    Let $(U_n, u_n)$ and $(V_n, w_n)$ be the pointed disks obtained by pulling back $\left(\tilde{U}_n, c_n\right)$ and $\left(\tilde{V}_n, v_n\right)$ under $f^{j_n}|_{A_n}$ respectively. Then, there is a branched covering map
    \[
    h_n : (U_n, u_n) \to (V_n, w_n)
    \]
    of degree at most $N$ such that $f^{j_n} \circ h_n = f^{k_n+j_n}$ on $U_n$. See Figure \ref{fig:thm-a-case-2}.

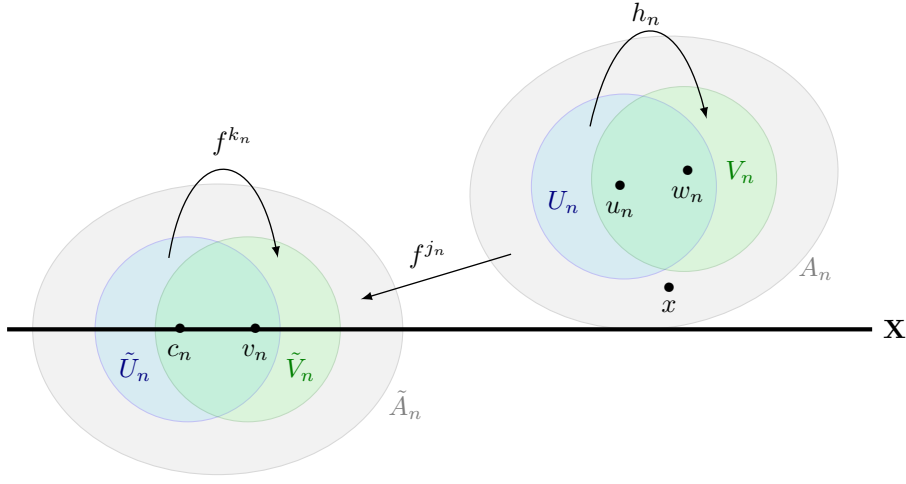
\begin{figure}
 \centering
    \begin{tikzpicture}
        \draw[gray!50!black, fill=gray!50!white, opacity=0.2] (-2.7,0) ellipse (70pt and 55pt);
        \filldraw[blue, fill=cyan!50!white, opacity=0.2] (-3.1,0) circle (35pt);
        \filldraw[green!50!black, fill=green!50!white, opacity=0.2] (-2.3,0) circle (35pt);
        \draw[gray!50!black, fill=gray!50!white, rotate around={10:(3.1,1.95)}, opacity=0.2] (3.1,1.95) ellipse (70pt and 55pt);
        \filldraw[blue, fill=cyan!50!white, opacity=0.2] (2.7,1.9) circle (35pt);
        \filldraw[green!50!black, fill=green!50!white, opacity=0.2] (3.5,2) circle (35pt);
        
        \draw[ultra thick, black] (-5.5,0) -- (6,0);
        \node [black, font=\bfseries] at (6.3,0) {$\PJ$};
        \node [black, font=\bfseries] at (-2.5,2.5) {$f^{k_n}$};
        \node [black, font=\bfseries] at (0.1,1) {$f^{j_n}$};
        \node [black, font=\bfseries] at (3,4.2) {$h_n$};
        \node [black, font=\bfseries] at (-3.2,0) {$\bullet$};
        \node [black, font=\bfseries] at (-3.2,-0.3) {$c_n$};
        \node [black, font=\bfseries] at (-2.2,0) {$\bullet$};
        \node [black, font=\bfseries] at (-2.2,-0.3) {$v_n$};
        \node [black, font=\bfseries] at (2.65,1.9) {$\bullet$};
        \node [black, font=\bfseries] at (2.65,1.6) {$u_n$};
        \node [black, font=\bfseries] at (3.55,2.1) {$\bullet$};
        \node [black, font=\bfseries] at (3.55,1.8) {$w_n$};
        \node [black, font=\bfseries] at (3.3,0.55) {$\bullet$};
        \node [black, font=\bfseries] at (3.3,0.3) {$x$};
        \node [gray, font=\bfseries] at (5.25,0.8) {$A_n$};
        \node [gray, font=\bfseries] at (-0.2,-1.0) {$\tilde{A}_n$};
        \node [blue!50!black, font=\bfseries] at (-3.8,-0.5) {$\tilde{U}_n$};
        \node [green!50!black, font=\bfseries] at (-1.6,-0.5) {$\tilde{V}_n$};
        \node [blue!50!black, font=\bfseries] at (1.9,1.7) {$U_n$};
        \node [green!50!black, font=\bfseries] at (4.25,2.1) {$V_n$};
        \draw[line width=0.5pt,-latex] (-3.35,0.95) .. controls (-3.1,2.5) and (-2.2,2.5) .. (-1.9,0.95);
        \draw[line width=0.5pt,-latex] (1.2,1) -- (-0.8,0.40);
        \draw[line width=0.5pt,-latex] (2.25,2.7) .. controls (2.6,4.3) and (3.4,4.4) .. (3.8,2.8);
\end{tikzpicture}

    \caption{The construction of the branched covering map $h_n$ in Case 2.}
    \label{fig:thm-a-case-2}
\end{figure}
    
    We claim that $h_n$ satisfies \ref{S1}-\ref{S5} in Proposition \ref{shen-criterion}. Indeed, \ref{S1} and \ref{S3} are immediate from the construction. Since $f^{j_n}$ has bounded distortion, $(U_n,u_n)$ and $(V_n,w_n)$ have bounded shape with diameter $\asymp r_n$, and $\dist(u_n,V_n) = O(r_n)$. Therefore, \ref{S2} and \ref{S4} are satisfied. Together with (\ref{ineq-00}), we also have \ref{S5}.
\end{proof}

To prove the first part of Theorem \ref{thm:NILF-main}, we will apply the following theorem.

\begin{theorem}[{\cite[Theorem 3.17]{McM94}}]
\label{thm:toral/attr}
    If a rational function $g$ has an invariant line field on its Julia set $J(g)$, then either $g$ is a \emph{flexible Latt\'es map} or for almost every point $z$ in $J(g)$,
    \[
        \dist\left(g^n(z), P(g)\right) \to 0 \quad \text{ as } n \to\infty.
    \]
\end{theorem}

A comprehensive exposition on Latt\'es maps can be found in \cite{Mil06}. We only need the fact that Latt\'es maps are postcritically finite, so in particular, our map $f$ is not one of them.

\begin{theorem}[No global invariant line fields]
\label{thm:NILF}
    The map $f$ does not admit any invariant line field on its Julia set.
\end{theorem}

\begin{proof}
    Suppose for a contradiction that there is an invariant line field $\mu$ supported on a positive measure subset of $J(f)$. By Theorem \ref{thm:toral/attr}, the set $E$ of points $z$ in $J(f)$ satisfying
    \[
        \dist\left(f^n(z), P(f)\right) \to 0 \quad \text{ as } n \to\infty
    \]
    has full measure. We will first show that $E$ can be decomposed as the union $E_1 \cup E_2 \cup \ldots \cup E_m$ where for each $j$,
    \[
        E_j := \{ z\in J(f) \: : \: \dist(f^n(z), X_j) \to 0 \text{ as } n \to\infty \}.
    \]

    Let us pick a constant $\varepsilon>0$ that is much smaller than the distance between any two connected components of $P(f) \cap J(f)$. Pick any point $z$ in $E$. There is some $N \in \N$ such that $\dist(f^n(z),P(f))<\varepsilon$ for all $n \geq N$. 

    Consider a fixed point $s$ in $\mathbf{S}$. By Lemma \ref{lem:parabolic-or-repelling}, $s$ must be either parabolic or repelling in nature. We claim that the forward orbit of $z$ can only visit the disk $\D(s,\varepsilon)$ at most finitely many times. Indeed, suppose for the contradiction that there is an infinite subsequence $\{n_j\}_{j \in \N}$ such that $f^{n_j}(z) \in \D(s,\varepsilon)$ for all $j \in \N$. As $z$ lies in the Julia set, each $f^{n_j}(z)$ cannot lie in any attracting petal of $s$ if $s$ is parabolic. By the repelling nature of $s$, there is some $\delta>0$ and a further subsequence $\{m_j\}_j$ of $\{n_j\}_j$ such that $\delta \leq |f^{m_j}(z)-s|\leq \varepsilon$ for all $j$. This would imply that $\dist\left(f^n(z), P(f)\right) \geq \delta >0$ for infinitely many $n$'s, which is impossible.

    Suppose $s \in \mathbf{S}$ is preperiodic, so there exists $k \geq 1$ such that $f^k(s)$ is a fixed point. Then, the forward orbit of $z$ can visit a small neighborhood of $s$ at most finitely many times, because, due to the previous argument, it can only get close to $f^k(s)$ at most finitely many times.

    As such, there are some moment $n \geq N$ and some point $w$ on some rotation curve $X_j$ such that $|f^n(z) - w| \leq \varepsilon$. As $X_j$ is $f$-invariant, we have $\dist(f^{n+k}(z),X_j) < \varepsilon$ for all $k \in \N$. Indeed, for small $\varepsilon$, we can make sure that the $\varepsilon$-neighborhood of $X_j$ is disjoint from the preimage of the $\varepsilon$-neighborhood of every connected component of $P(f) \cap J(f) \backslash X_j$. Consequently, $z$ must be in $E_j$.

    So far, we have shown that there is some $j \in \{1,\ldots m \}$ such that $E_j$ is a positive measure subset of the support of $\mu$. Every point in $E_j$ is eventually mapped to the local Julia set $J^{\textnormal{loc}}_j$, that is, $E_j = \bigcup_{n \in \N} f^{-n}\left(E_j \cap J^{\textnormal{loc}}_j\right)$. It follows that the intersection $E_j \cap J^{\textnormal{loc}}_j$ must have positive measure and $\mu$ induces an invariant line field on $J^{\textnormal{loc}}_j$. However, this contradicts Lemma \ref{NILF-local}.
\end{proof}

\subsection{Zero Lebesgue measure}
\label{ss:leb-meas}

The absence of invariant line fields of $f$ only has meaning when its Julia set has positive Lebesgue measure. Under the assumption that none of the $X_j$'s is a Herman curve, we can further obtain the stronger result that $J(f)$ has zero Lebesgue measure.

\begin{definition}
\label{def:porosity}
    Let $J$ be a compact subset of $\RS$. The set $J$ is \emph{porous} if there is some $\alpha >0$ such that for every point $x$ in $J$ and every sufficiently small $r>0$, the disk $\D(x,r)$ contains a smaller disk of radius $\alpha r$ disjoint from $J$.
\end{definition}

Porous sets have upper box dimension less than two, which implies that its Hausdorff dimension is less than two and it has zero Lebesgue measure. The following lemma is an immediate consequence of Theorem \ref{critical}.

\begin{lemma}[Local porosity]
\label{local-porosity}
    For every $j \in \{1,\ldots,m\}$, if $X_j$ is not a Herman curve, then the local Julia set $J^{\textnormal{loc}}_j$ is porous.
\end{lemma}

\begin{proof}
    Pick any point $z$ in $J^{\textnormal{loc}}_j$ and pick a sufficiently small $r >0$. We will show that the round disk $\D(z,r)$ contains a smaller round disk $D$ of radius $\asymp r$ contained in $F(f)$.
    
    By Theorem \ref{critical}, there is a univalent map $f^i: (U,y) \to (V,c)$ of bounded distortion such that $(U,y)$ is a pointed disk inside $\D(z,r)$ with inner radius $\asymp r$ and $c$ is a critical point on a rotation quasicircle $X_j$. Since $X_j$ is a quasicircle and is a boundary component of a rotation domain $W \subset F(f)$, there is a round disk $D' \subset V \cap W$ of diameter $\asymp \diam(V)$. Finally, since $f^i$ has bounded distortion on $U$, the lift of $D'$ under $f^i|_U$ contains a round disk $D$ of radius $\asymp r$ contained in $F(f)$.
\end{proof}

Let us now prove the second part of Theorem \ref{thm:NILF-main}.

\begin{theorem}[Zero area]
    If $f$ admits no Herman curves, then $J(f)$ has zero Lebesgue measure. 
\end{theorem}

\begin{proof}
    Suppose $f$ has no Herman curves and suppose for a contradiction that $J(f)$ has positive measure. The map $f$ must contain a rotation domain, and, in particular, has a non-empty Fatou set. By \cite[Theorem 3.9]{McM94}, the set of points $z$ in $J(f)$ satisfying
    \[
        d(f^n(z), P(f)) \to 0 \quad \text{ as } n \to\infty
    \]
    has full measure. By the same argument as in the proof of Theorem \ref{thm:NILF}, there is some $j \in \{1,\ldots,m\}$ such that $J^{\textnormal{loc}}_j$ has positive measure. However, this contradicts Lemma \ref{local-porosity}.
\end{proof}

With stronger assumptions, we can even obtain global porosity. 

\begin{theorem}[Global porosity]
\label{porosity-theorem}
    Suppose every critical point of $f$ lies in either the basin of an attracting cycle, the boundary of a bounded type rotation domain, or the grand orbit of a bounded type rotation domain. Then, $J(f)$ is porous.
\end{theorem}

\begin{proof}
    The assumption implies that $\mathbf{S}$ is empty and that $\Jthick$ coincides with the whole Julia set. This allows us to apply Theorem \ref{critical} to any point on the Julia set and repeat the proof of Lemma \ref{local-porosity}.
\end{proof}

\section{Combinatorial rigidity of Herman quasicircles}
\label{sec:applications}

In this section, we discuss a number of applications of Theorem \ref{thm:NILF-main} to the space $\HQspace_{d_0,d_\infty,\theta}$ introduced in Definition \ref{main-definition}. Rational maps in $\HQspace_{d_0,d_\infty,\theta}$ contain a single Herman quasicircle of the simplest configuration, and the unicritical ones will serve as model maps for critical quasicircle maps in Section \S\ref{sec:c1plusalpharigidity}.

\subsection{Combinatorial rigidity}
\label{ss:rigidity}

Fix a pair of integers $d_0, d_\infty \geq 2$ and a bounded type irrational number $\theta \in (0,1)$. For every rational map $f$ in $\HQspace_{d_0, d_\infty,\theta}$, we denote by $\comb(f) \in \mathcal{C}_{d_0,d_\infty}$ the combinatorics of $f$ along its Herman quasicircle, as defined in Definition \ref{combi}.

\begin{theorem}
\label{thm:combinatorial-rigidity-1}
    Any two combinatorially equivalent rational maps in $\HQspace_{d_0,d_\infty,\theta}$ are conformally conjugate.
\end{theorem}

In the proof, we will apply the standard pullback argument to promote combinatorial equivalence to quasiconformal conjugacy. The absence of invariant line fields will allow us to further promote the quasiconformal conjugacy to a conformal one.

\begin{proof}
    Suppose $f_1$ and $f_2$ are two combinatorially equivalent rational maps in $\HQspace_{d_0,d_\infty,\theta}$. For each $i \in \{1,2\}$, let $\Hq_i$ be the Herman quasicircle of $f_i$ and $\phi_i: \Hq_i \to \T$ be a quasisymmetric conjugacy between $f_i$ and $R_\theta$. By combinatorial equivalence, the conjugacies can be picked such that $\phi_2^{-1} \circ \phi_1$ preserves the critical points of $f_1$ and $f_2$ along their Herman curves.
    
    For each $i \in \{1,2\}$ and $\bullet\in \{0,\infty\}$, denote by 
    $\mathfrak{B}^\bullet_i$ the immediate basin of attraction of $\bullet$ of the map $f_i$. Since no critical points of $f_i$ are attracted to $\bullet$ other than $\bullet$ itself, the basin $\mathfrak{B}^\bullet_i$ is simply connected and there exists a B\"ottcher coordinate
    \[
    b_i^{\bullet}: (\mathfrak{B}^{\bullet}_i, \bullet) \to (\D,0),
    \]
    that is, a conformal isomorphism such that 
    \[
    b_i^\bullet \circ f_i(z) = b_i^\bullet(z)^{d_{\bullet}} \quad \text{ for all } \quad z \in \mathfrak{B}^{\bullet}_i.
    \]
    Let us also consider the neighborhood 
    \[
    E^{\bullet}_i := \left\{ z \in B_i^\bullet \: : \: |b_i^\bullet(z)| < 1/2 \right\}
    \]
    of $\bullet$ cut out by an equipotential.
    
    Let $h_0 : \RS \to \RS$ be a quasiconformal map such that
    $$
    h_0(z) = \begin{cases}
    (b_2^{\bullet})^{-1} \circ b_1^{\bullet}(z), & \text{if } z \in E_1^{\bullet}, \bullet \in \{0,\infty\}, \\
    \phi_2^{-1} \circ \phi_1(z), & \text{if } z \in \Hq_1, \\
    \text{quasiconformal interpolation}, & \text{if otherwise}. 
    \end{cases}
    $$
    Then, $h_0$ is conformal on $E_1:= E_1^0 \cup E_1^\infty$ and provides a conjugacy between $f_1$ and $f_2$ on $\Hq_1 \cup E_1$.
    
    Our choice of $\phi_1$ and $\phi_2$ ensures that $h_0$ preserves the covering structure of $f_1$ and $f_2$. In particular, we can lift $h_0$ to a quasiconformal map $h_1 : \RS \to \RS$ such that $f_2 \circ h_1 = h_0 \circ f_1$. This new map $h_1$ coincides with $h_0$ on $\Hq_1 \cup E_1$, restricts to a conformal conjugacy between $f_1$ and $f_2$ on $f_1^{-1}(E_1)$, and is homotopic to $h_0$ rel $P(f_1)$. Moreover, $h_1$ has the same quasiconformal dilatation as $h_0$ because both $f_1$ and $f_2$ are holomorphic. Repeat this lifting process to obtain an infinite sequence of uniformly quasiconformal homeomorphisms $\{h_n\}_{n \in \N}$ of $\RS$ such that for all $n\in\N$,
    \begin{enumerate}[label=(\roman*)]
        \item $f_2 \circ h_{n+1} = h_n \circ f_1$;
        \item $h_{n+1} = h_n$ on $f_1^{-n}(\Hq_1 \cup E_1)$;
        \item $h_n$ restricts to a conformal conjugacy between $f_1$ and $f_2$ on $f_1^{-n}(E_1)$.
    \end{enumerate}
    
    By the compactness of the space of normalised quasiconformal maps, $h_n$ converges in subsequence to a quasiconformal map $h_\infty : \RS \to \RS$. The limit $h_\infty$ is a conformal conjugacy between $f_1$ and $f_2$ on the Fatou sets because $\bigcup_{n \geq 0} f^{-n}(E_i)$ coincides with the Fatou set of $f_i$ for $i \in \{1,2\}$. By continuity, since the Julia sets of $f_1$ and $f_2$ are nowhere dense, $h_\infty$ is a global quasiconformal conjugacy between $f_1$ and $f_2$. 
    
    The absence of invariant line fields on the Julia set implies that $\overline{\partial}h_\infty = 0$ almost everywhere on $J(f_1)$. By Weyl's lemma, $h_\infty$ is indeed a conformal conjugacy between $f_1$ and $f_2$ in $\RS$.
\end{proof}

% Anonymous Version
In \cite{Lim23}, \emph{a priori bounds} for Herman rings of rational maps in $\mathcal{H}_{d_0,d_\infty,\theta}$ was established. This leads to the following precompactness result.

% Real Version
% In \cite{Lim23}, we proved \emph{a priori bounds} for Herman rings of rational maps in $\mathcal{H}_{d_0,d_\infty,\theta}$ and as a result obtained the following precompactness result.

\begin{theorem}[{\cite[Theorem 9.1]{Lim23}}]
\label{precompactness}
    Denote by $\He_f$ the Herman ring of a rational map $f$ whenever there is a unique one. For any $\mu>0$ and $N \in \N$, the quotient space 
    \[
    \bigcup_{\theta \in \Theta_N} \left\{ f \in  \mathcal{H}_{d_0,d_\infty,\theta} \: | \:  \modu(\He_f) < \mu \right\}/_\sim
    \]
    is precompact in $\rat_{d_0+d_\infty-1}$.
\end{theorem}

One of the corollaries of \emph{a priori bounds} is that the accumulation space $\mathcal{H}^\partial_{d_0,d_\infty,\theta}$ is always contained in $\HQspace_{d_0,d_\infty,\theta}$. Then, Theorems \ref{APB02} and \ref{thm:combinatorial-rigidity-1} directly imply that
\[
    \HQspace_{d_0,d_\infty,\theta} = \Herspace^\partial_{d_0,d_\infty,\theta}.
\]
This proves Corollary \ref{main-corollary}. Let us now prove the second part of Theorem \ref{thm:combinatorial-rigidity}. For convenience, we will use the notation
\[
\HQspace_{d_0,d_\infty,\Theta_N} := \bigcup_{\theta \in \Theta_N} \HQspace_{d_0,d_\infty,\theta}.
\]

\begin{theorem}
\label{thm:combinatorial-rigidity-2}
    For any integer $N \geq 1$, the map 
    \[
    \Phi: \HQspace_{d_0,d_\infty,\Theta_N}/_\sim \to \mathcal{C}_{d_0,d_\infty} \times \Theta_N, \quad [f] \mapsto \left(\comb(f),\rot(f)\right)
    \]
    is a homeomorphism.
\end{theorem}

\begin{proof}
    Fix $N \geq 1$. By Theorem \ref{thm:combinatorial-rigidity-1}, we know that $\Phi$ is a well-defined bijection. By Theorem \ref{precompactness}, the space $\HQspace_{d_0,d_\infty,\Theta_N}/_\sim$ is Hausdorff and compact, so it remains to show that $\Phi$ is continuous.

    Suppose a sequence $f_n$ of rational maps in $ \HQspace_{d_0,d_\infty,\Theta_N}$ converges to $f \in \HQspace_{d_0,d_\infty,\Theta_N}$. For each $n \in \N$, let $\Hq_n$ denote the Herman quasicircle of $f_n$. There exists a $K(d_0,d_\infty,N)$-quasiconformal map $\phi_n: \RS \to \RS$ that fixes $0$ and $\infty$ and conjugates $f_n|_{\Hq_n}$ and some irrational rotation $R_{\theta_n}$ on the unit circle $\T$.
    
    By passing to a subsequence, the sequence $\phi_{n}$ converges locally uniformly to a $K$-quasiconformal map $\phi: \RS \to \RS$ that fixes $0$ and $\infty$. Since the maps $\phi_n \circ f_n \circ \phi_n^{-1}$ preserve $\T$, so is its limit $\phi \circ f \circ \phi^{-1}$. Thus, the sequence of rotations $R_{\theta_{n}}$ converges to $R_\theta$ where $\theta = \lim_{k\to\infty}\theta_{n}$. It follows that $\Hq := \phi^{-1}(\T) = \lim_{n\to \infty} \Hq_n$ is the Herman quasicircle of $f$ and $\rot(f) = \theta = \lim_{n\to\infty} \rot(f_n)$, and in particular, the limit $\theta$ is independent of the initial choice of subsequence. This proves the continuity of the map $\rot(\cdot)$. To see that $\comb(\cdot)$ is also continuous, observe that given any critical point $c$ of $f$ on $\Hq$, there is a sequence $c_n$ of critical points of $f_n$ such that $c_n \to c$.
\end{proof}

\subsection{Trivial Herman curves}
\label{ss:blaschke}

Consider an integer $d\geq 2$ and a bounded type irrational $\theta \in (0,1)$. Let us denote by $\mathcal{B}_{d,\theta}$ the space of rational maps in $\HQspace_{d,d,\theta}$ which are Blaschke products, i.e. those that commute with the reflection $\tau(z)=1/\bar{z}$ along the unit circle $\T$, or equivalently, those whose Herman quasicircles are $\T$.

For any $T^0 \in \SP^{d_0-1}(\T)$ and $T^\infty \in \SP^{d_\infty-1}(\T)$, we denote the corresponding element in $\mathcal{C}_{d_0,d_\infty}$ by $\mathcal{C}=[(T^0,T^\infty)]$ and say that $\mathcal{C}$ is \emph{symmetric} if $d_0=d_\infty$ and $T^0=T^\infty$. If a Herman curve $\Hq$ has symmetric combinatorics, then every critical point on $\Hq$ is both an outer and an inner critical point, and its outer and inner criticalities coincide.

\begin{proposition}[Blaschke $\leftrightarrow$ combinatorial symmetry]
\label{blaschke-characterization}
    Every $f \in \mathcal{B}_{d,\theta}$ has symmetric combinatorics. Conversely, given a symmetric combinatorial data $\mathcal{C} \in \mathcal{C}_{d,d}$, the map $f \in \HQspace_{d,d,\theta}$ realizing $\mathcal{C}$ as described in Theorem \ref{APB02} is conformally conjugate to a Blaschke product, unique up to conjugacy by a rigid rotation.
\end{proposition}

\begin{proof}
    The first statement follows from the observation that for any general rational map $f \in \HQspace_{d_0,d_\infty,\theta}$, if $f$ has combinatorics $[(T^0,T^\infty)]$, then $\tau \circ f \circ \tau$ lies in $\HQspace_{d_\infty,d_0,\theta}$ with combinatorics $[(T^\infty,T^0)]$.
    
    Suppose $f \in \HQspace_{d,d,\theta}$ has a Herman quasicircle $\Hq$ with symmetric combinatorics $[(T,T)]$. Mark one of the critical points of $f$ and assume it is $z=1$ after conjugation with a linear map. The rational map $g(z) := \tau \circ f \circ \tau$ has a Herman quasicircle $\tau(\Hq)$ with the same rotation number and the same combinatorics $[(T,T)]$ due to combinatorial symmetry. By Theorem \ref{thm:combinatorial-rigidity}, there is a linear map $L(z)=\lambda z$, $\lambda \in \C^*$ such that $g = L \circ f \circ L^{-1}$. Moreover, $L$ can be chosen to preserve the marked critical points of $f$ and $g$, which are $1$ and $\tau(1)=1$. Thus, $\lambda=1$ and $g=f$, which implies that $f$ is a Blaschke product. Uniqueness also follows from rigidity. 
\end{proof}

By Theorem \ref{thm:combinatorial-rigidity-2} and Proposition \ref{blaschke-characterization}, the map $\comb(\cdot)$ induces a homeomorphism between $\mathcal{B}_{d,\theta}/_\sim$ and the space 
\[
\left\{[(T,T)] \in \mathcal{C}_{d,d} \: : \: T \in \SP^{d-1}(\T)\right\}.
\]
Observe that the latter is homeomorphic the quotient space $\mathcal{S}_{d}$ from Definition \ref{combspace}.

\begin{corollary}
\label{topology-blaschke}
    $\comb(\cdot)$ induces a homeomorphism $\comb': \mathcal{B}_{d,\theta}/_\sim \to \mathcal{S}_{d}$. For every $N \geq 1$,
    \[
    \bigcup_{\theta \in \Theta_N} \mathcal{B}_{d,\theta}/_\sim \to \mathcal{S}_d \times \Theta_N, \qquad [f] \mapsto \left( \comb'(f), \rot(f)\right)
    \]
    is a homeomorphism.
\end{corollary}

Let $\mathcal{Z}_{d,\theta}$ denote the space of degree $d$ polynomials $f$ that admit a single Siegel disk $Z$ such that $Z$ is centered at $0$, has rotation number $\rot(f)= \theta$, and contains every free critical point of $f$ on its boundary. We denote by $\comb(f) \in \mathcal{S}_{d}$ the combinatorics of $f|_{\partial Z}$ in accordance to Definition \ref{combi}.

The dynamical relation between $\mathcal{B}_{d,\theta}$ and $\mathcal{Z}_{d,\theta}$ can be formulated via the Douady-Ghys surgery. (Cf. \cite{G84, D87}. See also \cite[\S7.2]{BF14}.) In short, for every map $f$ in $\mathcal{B}_{d,\theta}$, we replace the dynamics of $f \in \mathcal{B}_{d,\theta}$ inside the unit disk with a quasiconformal copy of irrational rotation $R_\theta$ on the unit disk. After straightening via the measurable Riemann mapping theorem, we obtain a linear conjugacy class of a polynomial in $\mathcal{Z}_{d,\theta}$ of the same combinatorics and rotation number.

\begin{corollary}
    For $N \geq 1$, the Douady-Ghys surgery induces a homeomorphism 
\[
    \textnormal{DG}: \bigcup_{\theta \in \Theta_N} \mathcal{B}_{d,\theta}/_\sim \to \bigcup_{\theta \in \Theta_N} \mathcal{Z}_{d,\theta}/_\sim
\]
    satisfying $(\comb, \rot) \circ \textnormal{DG} = (\comb',\rot)$.
\end{corollary}

For $d=3$, a variation of this corollary was previously studied in \cite{Z99}.

\begin{proof}
    Combinatorial rigidity of Siegel polynomials in $\mathcal{Z}_{d,\theta}$ (cf. \cite{Z08}) ensures that $\textnormal{DG}$ is well-defined. The equation $(\comb, \rot) \circ \textnormal{DG} = (\comb',\rot)$ holds because the surgery preserves the combinatorics and the rotation number. Note that the moduli space $\mathcal{Z}_{d,\theta}/_\sim$ is also compact (see \cite[Remark 9.4]{Lim23}). By repeating the proof of Theorem \ref{thm:combinatorial-rigidity-2}, we can show that the map $(\comb,\rot)$ is a homeomorphism from $\bigcup_{\theta \in \Theta_N} \mathcal{Z}_{d,\theta}/_\sim$ onto $\mathcal{S}_d \times \Theta_N$. Together with Corollary \ref{topology-blaschke}, we conclude that the map $\textnormal{DG}$ is a homeomorphism.
\end{proof}

\subsection{Antipode-preserving cubic rational maps}
\label{ss:antipode}
We end this section with an application of rigidity to the following family of cubic rational maps
\[
f_q(z) = z^2\frac{q-z}{1+\bar{q}z}, \qquad q \in \C^*.
\]
This family was first studied in \cite{BBM18} and is characterized by a simple critical fixed point at $0$ and the property that $f_q$ is antipode-preserving, that is, $f_q$ commutes with the antipodal map $z \mapsto -1/\overline{z}$.

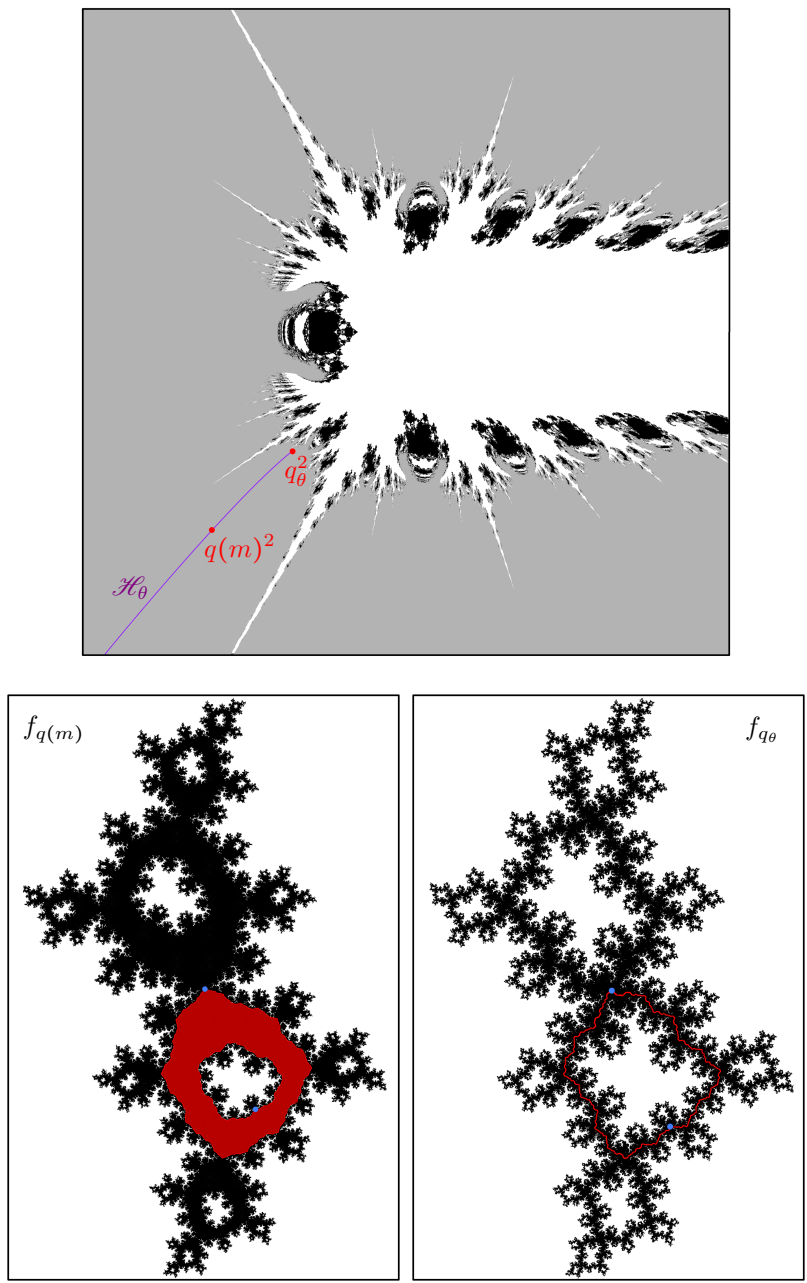
\begin{figure}
    \centering

\begin{tikzpicture}
    \node[anchor=south west,inner sep=0] (image) at (0,0) {\includegraphics[width=0.85\linewidth]{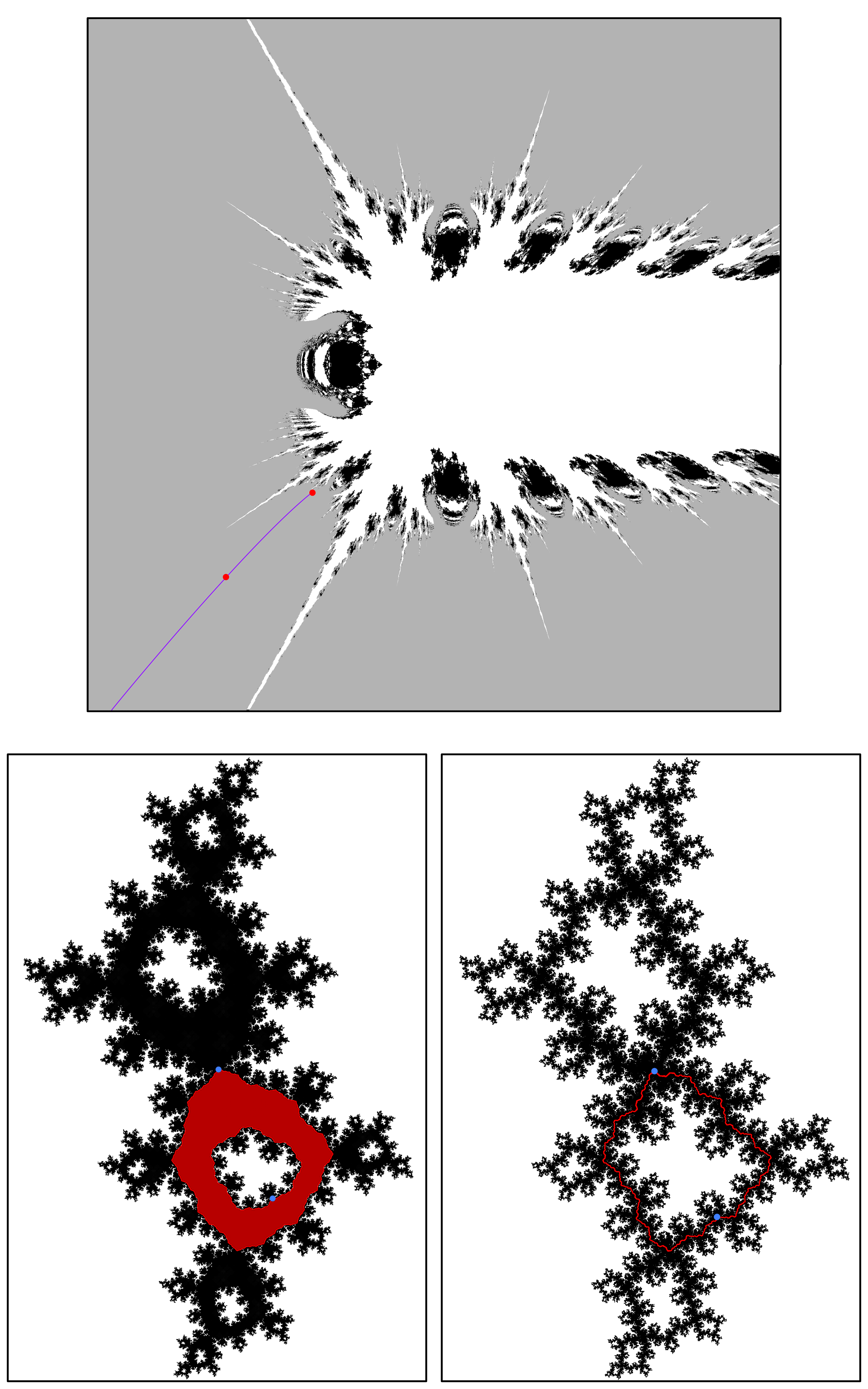}};
    \begin{scope}[
        x={(image.south east)},
        y={(image.north west)}
    ]
        \node [violet, font=\bfseries] at (0.16,0.54) {$\mathscr{H}_\theta$};
        \node [black, font=\bfseries] at (0.94,0.43) {$f_{q_\theta}$};
        \node [black, font=\bfseries] at (0.065,0.43) {$f_{q(m)}$};
        \node [red, font=\bfseries] at (0.295,0.57) {$q(m)^2$};
        \node [red, font=\bfseries] at (0.365,0.63) {$q_\theta^2$};
    \end{scope}
\end{tikzpicture}
    
    \caption{Above: The $q^2$-parameter plane for $\{f_q\}$ containing the golden mean hair $\mathscr{H}_\theta$, colored in purple. Below: The dynamical planes of $f_{q(m)}$ and $f_{q_\theta}$ where $q(m)^2 \approx -12.06-12.30i$ lies on $\mathscr{H}_\theta$ and $q_\theta^2 \approx -7.05 - 7.41i$ is the endpoint of $\mathscr{H}_\theta$. The Herman ring of $f_{q(m)}$ and the Herman quasicircle of $f_{q_\theta}$ are colored in red.}
    \label{fig:antipodal}
\end{figure}

Note that $f_q$ and $f_{q'}$ are linearly conjugate if and only if $q'=-q$, so it is natural to consider the $q^2$-plane as the appropriate parameter space. According to \cite{BBM18}, this parameter space has the remarkable property of admitting Herman rings of arbitrary Brjuno rotation number and modulus. Below, we quote a more precise formulation from a sequel \cite{BBM23} in progress.

\begin{theorem}[Hair Theorem]
    For any Brjuno number $\theta \in (0,1)$, there exists a unique ``hair`` $\mathscr{H}_\theta$ in the $q^2$-plane consisting of all maps $f_q$ with a Herman ring of rotation number $\theta$. They satisfy the following properties.
    \begin{enumerate}[label=\textnormal{(\arabic*)}]
        \item For any $m \in (0,\infty)$, there is a unique parameter $q(m)^2$ in $\mathscr{H}_\theta$ such that $f_{q(m)}$ admits a unique invariant Herman ring of modulus $m$.
        \item The map $(0,\infty) \to \mathscr{H}_\theta$, $m \mapsto q(m)^2$ is an analytic and regular parametrization of $\mathscr{H}_\theta$.
        \item As $m \to \infty$, $|q(m)| \to \infty$.
    \end{enumerate}
    Moreover, the Herman ring locus $\mathscr{H}:=\bigcup_\theta \mathscr{H}_\theta$ in the $q^2$-plane has positive measure. The Hausdorff $1$-measure of the intersection $\mathscr{H} \cap \{|q^2|=r\}$ tends to $1$ as $r\to\infty$.
\end{theorem}

Assuming the theorem above, we can apply our rigidity result and deduce that when $\theta$ is of bounded type, the corresponding hair $\mathscr{H}_{\theta}$ lands at a unique point. See Figure \ref{fig:antipodal}.

\begin{corollary}[Landing of hairs]
    When $\theta$ is of bounded type, the hair $\mathscr{H}_\theta$ has a unique endpoint $\displaystyle{q_\theta^2:=\lim_{m\to 0} q(m)^2}$. The map $f_{q_\theta}$ lies in $\HQspace_{2,2,\theta}$.
\end{corollary}

\begin{proof}
    Herman rings in $\mathscr{H}_\theta$ automatically lie in the space $\Herspace_{2,2,\theta}$. By Theorem \ref{precompactness}, the set of maps in $\mathscr{H}_\theta$ admitting Herman rings of modulus bounded above by some positive constant is precompact. Therefore, the accumulation set $\mathscr{H}^\partial_\theta:= \overline{ \mathscr{H}_\theta }\backslash \mathscr{H}_\theta$ is non-empty and contained in $\HQspace_{2,2,\theta}$. 
    
    Pick any parameter $q^2$ in $\mathscr{H}^\partial_\theta$. Let $\phi: \Hq \to \T$ be a quasiconformal conjugacy between $f_q$ on its Herman quasicircle and the irrational rotation $R_\theta$. Since $f_q$ commutes with $\tau(z)=-1/\overline{z}$ and since $\phi$ is unique up to post-composition with rigid rotation, then $\phi \circ \tau = -\phi$. In particular, $f_q|_\Hq$ must have combinatorics $[\{1\}, \{-1\}]$. By Theorem \ref{thm:combinatorial-rigidity}, maps in $\mathscr{H}^\partial_\theta$ are linearly conjugate to each other, so $\mathscr{H}^\partial_\theta$ must be a singleton.
\end{proof}

\section{Deep points along rotation quasicircles}
\label{sec:deep-points}

Suppose a holomorphic map $f$ admits a bounded type rotation quasicircle $\Hq$ containing a critical point. Let $A$ be a collar of $f|_\Hq$ (cf. Section \S\ref{ss:approx-rot}). We split $A$ along $\Hq$ into two:  the inner collar $A^0$ and the outer collar $A^\infty$. In case \ref{case-B}, we assume that $f|_{A^0}$ is analytically conjugate to irrational rotation.

\begin{definition}
    \label{def:local-filled-julia-set}
    In case \ref{case-H}, we define the \emph{local filled Julia set} $K^{\textnormal{loc}}_A(f)$ of $f$ rel $A$ to be the closure of the set of points whose forward orbit lies entirely in $\overline{A}$ and eventually lands in $\Hq$, i.e.
    \[
        K^{\textnormal{loc}}_A(f) := \overline{ \bigcap_{n\geq 0} f^{-n}\left(\overline{A}\right) \cap \bigcup_{n \geq 0} f^{-n}(\Hq) }.
    \] 
    In case \ref{case-B}, we define the \emph{local filled Julia set} $K^{\textnormal{loc}}_A(f)$ of $f$ rel $A$ to be the closure of the set of points whose forward orbit lies entirely in $\overline{A}$ and eventually lands in $\overline{A^0}$.
\end{definition}

\begin{definition}
    \label{eqn:uniform-depth}
    We say that a subset $S$ of a compact set $J \subset \RS$ is \emph{uniformly deep} if there are positive constants $C,\delta,r>0$ such that for every point $z$ inside the $r$-neighborhood of $S$,
    \[
        \dist(z,J) \leq C \, \dist(z,S)^{1+\delta}. 
    \]
\end{definition} 

The following theorem is a generalization of \cite[Theorem 4.2]{McM98}.

\begin{theorem}[Rotation quasicircles are deep]
\label{deep-point-theorem}
    Consider a rational map $f$ admitting a rotation quasicircle $\Hq$ with bounded type rotation number. If $P(f) \backslash \Hq$ is disjoint from a neighborhood of $\Hq$, then $\Hq$ is uniformly deep in the local filled Julia set $K^{\textnormal{loc}}_A(f)$ rel any collar $A$ of $f|_\Hq$. In particular,
    \begin{enumerate}[label=\textnormal{(\arabic*)}]
        \item if $\Hq$ is a Herman quasicircle, then $\Hq$ is uniformly deep in $J(f)$;
        \item otherwise, $\Hq$ is contained in the closure of a rotation domain $D$, and it is uniformly deep in the closure of the grand orbit of $D$.
    \end{enumerate}
\end{theorem}

In case (1), the local filled Julia set of a Herman curve $\Hq$ is contained in the Julia set $J(f)$ of $f$, so $J(f)$ cannot be porous (see Definition \ref{def:porosity}) and in fact it converges exponentially fast to the whole plane when magnified about any point on $\Hq$. See Figure \ref{fig:julia-set}.

\begin{figure}
    \centering
    \includegraphics[width=\columnwidth]{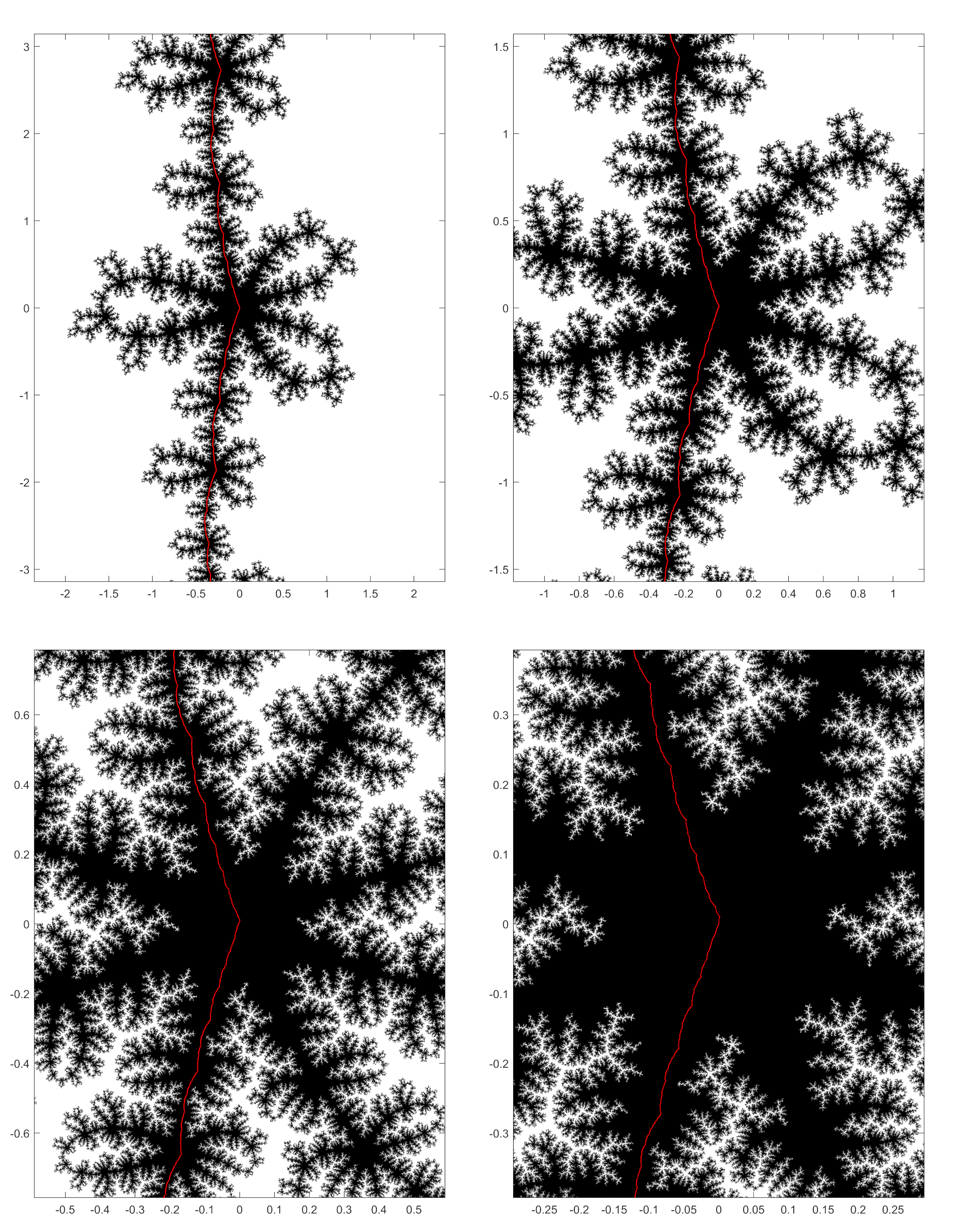}
    \caption{The plot of the Julia set of $f(z)=c z^2 \dfrac{z^2-4z+6}{4z-1}$ in logarithmic coordinates, magnified about its free critical point. The critical value $c = f(1) \approx -0.51094768-0.43067820i$ is chosen such that $f$ admits a Herman quasicircle with golden mean rotation number, shown in red.} 
    \label{fig:julia-set}
\end{figure}

\begin{proof}
    We will first consider case \ref{case-H}. The treatment for case \ref{case-B} is analogous, with minor differences which will be outlined at the end of the proof.
    
    Consider a global quasiconformal map $\phi$ conjugating $f|_\Hq$ and the irrational rotation $R_\theta|_\T$. Recall the function $L(z)$ defined in (\ref{eqn:escape-function}). Pick any sufficiently large $\kappa>1$ such that 
    \[
    A:=\{-\infty \leq L(z)<-\kappa \}
    \]
    is a collar for $\Hq$ disjoint from some small neighborhood of $P(f) \backslash \Hq$. By the H\"older continuity of $\phi$, for every $z \in A$,
    \begin{equation}
        \label{log-distance}
        L(z) \asymp \log(\dist(z,\Hq)).  
    \end{equation}
    Consider the complement $\Omega := \RS \backslash P(f)$. This is the largest open subset of $\RS$ such that $f^n: f^{-n}(\Omega) \to \Omega$ is an unbranched covering map for all $n \in \N$. Similar to Lemma \ref{1/d-metric}, since $A \cap P(f) = \Hq$, the hyperbolic metric $\rho_\Omega(z) |dz|$ of $\Omega$ satisfies
    \begin{equation}
    \label{eqn:hyp-metric-deep}
        \rho_\Omega(z) \asymp \frac{1}{\dist(z,\Hq)} \asymp \rho_{\RS \backslash \Hq}(z) \quad \text{ for all } z \in A.
    \end{equation}
    
    Given an approximate rotation $f^i : U\to V$ (rel $A$), we have $L(f^i(z)) \leq L(z) + O(1)$ for every $z \in U$. In general, we have the following property.
    
    \begin{claim} 
        For every $z \in A$, $L(f(z)) \leq L(z) + O(1)$.
    \end{claim}
    
    \begin{proof}
        Let us denote by $S$ the set of critical points of $f$ on $\Hq$. If $z$ is sufficiently far from $S$, i.e. $\dist(z,S) \succ \dist(z,\Hq)$, then $f$ is an approximate rotation on a neighborhood of $z$. Else, suppose $z$ is close to a critical point $c \in S$. Let $z':=\phi(z)$, $c':=\phi(c)$, and $F:= \phi \circ f \circ \phi^{-1}$. As $F$ is a quasiregular map that restricts to an isometry of the unit circle $\T$, we have $|F(z')-F(c')| \asymp |z'-c'|$. Therefore,
        \[
        \dist(F(z'),\T) \leq |F(z')-F(c')| \asymp |z'-c'| \asymp \dist(z', \T).
        \]
        By taking the logarithm, this inequality implies the claim.
    \end{proof}
    
    Let us equip $\Omega$ with the hyperbolic metric. For $z \in A$, let $r_z := \dist_\Omega\left(z, f^{-1}(\Hq)\right)$. The norm of $f'(z)$ with respect to the hyperbolic metric of $\Omega$ satisfies
    \begin{equation}
        \label{eqn:exp}
    \| f'(z)\| \geq C(r_z) > 1
    \end{equation}
    for some function $C(r)$ where $C(r) \to \infty$ as $r \to 0$.

    Let us pick any point $z$ in $A$ with $L(z) \ll -\kappa$, and denote $z_i := f^i(z)$ for every $i \in \N$. By Corollary \ref{approx-rot-to-preim}, there is an approximate rotation $f^{n_1}: (U,z) \to (V,z_{n_1})$ such that $\dist_\Omega\left(z_{n_1},f^{-1}(\Hq)\right) = O(1)$. By (\ref{eqn:exp}), we have $\| f'(z_{n_1}) \| \geq M > 1$ for some $M$ independent of $z$. On the other hand, $L(z_{n_1 + 1}) \leq L(z) + O(1)$ due to the claim above. We repeat this argument inductively to obtain an increasing sequence of positive integers $\{ n_j \}_j$ such that 
    \begin{align}
    \label{eqn:close-to-preimage}
    \dist_\Omega\left(z_{n_j},f^{-1}(\Hq)\right) &= O(1), \\
    \label{eqn:definite-exp}
    \|f'(z_{n_j})\| &\geq M,\text{ and } \\
    \label{eqn:induct} 
    L(z_{n_j}) &\leq L(z) + O(j).
    \end{align}
    
    By (\ref{log-distance}) and (\ref{eqn:induct}), there exists some $m \in \N$ such that $z_{n_i} \in \overline{A}$ for all $i \in \{1,2,\ldots m\}$ and that
    \begin{equation}
    \label{eqn:const-N}
    m \asymp - L(z) \asymp -\log\left(\dist(z,\Hq)\right).
    \end{equation}
    Then, by (\ref{eqn:definite-exp}) and (\ref{eqn:const-N}),
    \begin{equation}
    \label{eqn:overall-exp}
    \| (f^{n_m})'(z)\| \geq M^{m-1} \succ \dist(z,\Hq)^{-\alpha},
    \end{equation}
    where $\alpha \asymp \log M$.
    
    By (\ref{eqn:close-to-preimage}), we can pick an arc $\gamma_m \subset \Omega$ of hyperbolic length $O(1)$ joining $z_{n_m}$ and some point $y'$ in $A \cap f^{-1}(\Hq)$. By (\ref{eqn:overall-exp}), $\gamma_m$ lifts under $f^{n_m}$ to an arc $\gamma \subset \Omega$ joining $z$ and some point $y \in f^{-n_m-1}(\Hq)$ of hyperbolic length $O(\dist(z,\Hq)^{\alpha})$. Therefore, by (\ref{eqn:hyp-metric-deep}),
    \[
    |z-y| = O\left(\dist(z,\Hq)^{\alpha+1}\right).
    \]
    We can make sure that the forward orbit of $y$ stays within $\overline{A}$ and is thus contained in $K^{\textnormal{loc}}_A(f)$. As such, the estimate above implies that $\Hq$ is uniformly deep in $K^{\textnormal{loc}}_A(f)$.

    In case \ref{case-B}, every point $z \in A$ that eventually lands in the inner collar $A^0$ is contained in $K^{\textnormal{loc}}_A(f)$. Hence, we only need to consider points $z \in A$ whose forward orbit stays entirely within the outer collar $A^\infty$. Since Corollary \ref{approx-rot-to-preim} is applicable for points in $A^\infty$, we can repeat the same analysis above.
\end{proof}

A stronger notion of deep points is the following. We say that a point $x$ in a compact subset $J$ of $\RS$ is a \emph{measurable deep point} of $J$ if there is some $\alpha >0$ such that for any $r>0$, the Lebesgue measure of $\D(z,r) \backslash J$ is $O(r^{2+\delta})$. 

McMullen \cite{McM98} showed that every point along the boundary of a bounded type quadratic Siegel disk is a measurable deep point of the filled Julia set. Fagella and Henriksen \cite{FH17} also showed that a similar result holds for cubic Blaschke products admitting a bounded type Herman ring, which can be thought of as a model map for quadratic Siegel disks. Below, we state a generalization of their results.

\begin{corollary}
\label{thm:meas-deep-point}
    Consider a rational map $f$ that is J-rotational of bounded type. For every rotation domain $D$ of $f$, every point along $\partial D$ is a measurable deep point of the closure of the grand orbit of $D$.
\end{corollary}

\begin{proof}
    Consider one of the rotation quasicircles of $f$, say $X_j$, and suppose it is a boundary component of a rotation domain $D$. Consider the collar $A_j$ introduced in \S\ref{ss:decomposition}. The previous theorem states that every point along $X_j$ is a deep point of the local filled Julia set $K^{\textnormal{loc}}_{A_j}(f)$. To prove the corollary, \cite[Proposition 2.24]{McM96} states that it is sufficient to show that the boundary of $K^{\textnormal{loc}}_{A_j}(f)$ is porous. The proof is analogous to that of Lemma \ref{local-porosity}. 
\end{proof}

\section{Rigidity of critical quasicircle maps}
\label{sec:c1plusalpharigidity}

From this section onwards, we will be studying rotation quasicircles admitting a single critical point, i.e. critical quasicircle maps $f: \Hq \to \Hq$.

The rigidity problem for critical quasicircle maps in case \ref{case-B} was already solved in \cite{McM98}, \cite[\S4]{AL22}, and \cite{Y08}. As such, we will restrict our attention to the rigidity problem in case \ref{case-H}. In the bounded type regime, this is equivalent to the assumption that both inner and outer criticalities are at least two. (Refer to Proposition \ref{trichotomy}.) Our goal is to study renormalizations of unicritical Herman quasicircles and ultimately prove Theorem \ref{thm:c1plusalpharigidity}.

Throughout this section, we fix a triplet of integers $d_0\geq 2$, $d_\infty \geq 2$, and $N \geq 1$. Unless otherwise stated, any quasicircle $\Hq \subset \RS$ considered will be assumed to separate $0$ and $\infty$. Denote by $Y_\Hq^0$ and $Y_\Hq^\infty$ the connected components of $\RS \backslash \Hq$ containing $0$ and $\infty$ respectively. 

\subsection{Critical quasicircle maps}
\label{ss:model-maps}

We say that a critical quasicircle map $f: \Hq \to \Hq$ is a \emph{$(d_0,d_\infty)$-critical quasicircle map} if the critical point $c$ of $f$ has inner criticality $d_0(c) = d_0$ and outer criticality $d_\infty(c) = d_\infty$.

The prototype example of critical quasicircle maps in case \ref{case-H} comes from the family $\HQspace_{d_0,d_\infty,\theta}$.

\begin{proposition}
\label{prop:prototype-example}
    Consider a bounded type irrational number $\theta \in (0,1)$.
    \begin{enumerate}[label=\textnormal{(\arabic*)}]
        \item There exists a unique rational map $F_{d_0,d_\infty,\theta}$ in $\HQspace_{d_0,d_\infty,\theta}$ such that the point $z=1$ is a critical point with full local degree $d_0+d_\infty-1$. It is of the form
    \[
        F_{d_0,d_\infty,\theta}(z) := - c \: \cfrac{ \displaystyle\sum_{j=d_0}^{d_0+d_\infty-1} \binom{d_0+d_\infty-1}{j} \cdot (-z)^j}{ \displaystyle\sum_{j=0}^{d_0-1} \binom{d_0+d_\infty-1}{j} \cdot (-z)^j},
    \]
        for some unique $c \in \C^*$, which is the image of $1$. Restricted to its Herman quasicircle $\Hq$, $F_{d_0,d_\infty,\theta} : \Hq \to \Hq$ is a $(d_0,d_\infty)$-critical quasicircle map with rotation number $\theta$ and a unique critical point at $1 \in \Hq$.
        \item When $d_0=d_\infty=d$, the map $F_{d,d,\theta}$ is a Blaschke product of the form
    \[
    B_{d,\alpha}(z) := e^{2\pi i \alpha} z^{d} \cdot \frac{ \displaystyle\sum_{j=0}^{d-1} \binom{2d-1}{j} (-1)^j z^{d-1-j} }{\displaystyle\sum_{j=0}^{d-1} \binom{2d-1}{j} (-1)^j z^j}
    \]
        for some unique $\alpha \in [0,1)$. In this case, $\Hq$ is the unit circle.
    \end{enumerate}
\end{proposition}

\begin{proof}
    Such a map $F_{d_0,d_\infty,\theta}$ must have unicritical combinatorics $[(T^0,T^\infty)] \in \mathcal{C}_{d_0,d_\infty}$ where $T^0 = (1,\ldots,1) \in \SP^{d_0-1}(\T)$ and $T^\infty = (1,\ldots,1) \in \SP^{d_\infty-1}(\T)$. The formulas can be computed from the fact that $F_{d_0,d_\infty,\theta}$ fixes $0$ with local degree $d_0$, fixes $\infty$ with local degree $d_\infty$, and has $1$ as the only other critical point; see \cite[Example 2.10]{Lim23} and \cite[Proposition 10.1]{Lim23} for details. The uniqueness of $F_{d_0,d_\infty,\theta}$ is due to Theorem \ref{thm:combinatorial-rigidity}, and item (2) is a consequence of Proposition \ref{blaschke-characterization}.
\end{proof}

The Julia sets of some of these rational maps are illustrated in Figure \ref{fig:cqc-comparison}. The following lemma, when applied to $f = F_{d_0,d_\infty,\theta}$, will come in handy later towards the end of this section.

\begin{lemma}
\label{lem:fatou-set-structure}
    Fix a bounded type $\theta$, a map $f$ in $\HQspace_{d_0,d_\infty,\theta}$ and a point $\bullet \in \{0,\infty\}$. Denote by $\Hq$ the Herman quasicircle of $f$.
    \begin{enumerate}[label=\textnormal{(\arabic*)}]
        \item The immediate attracting basin $\mathfrak{B}^\bullet$ of $\bullet$ of $f$ is a topological disk with locally connected boundary.
        \item Let $b^\bullet:(\mathfrak{B}^\bullet,\bullet) \to (\D,0)$ be the corresponding B\"ottcher coordinate. For every point $x$ on $\Hq$, if $x$ is not a pre-critical point ($f^n(x)$ is not a critical point for all $n\geq 0$), then there exists a unique angle $t \in \R \backslash \Z$ such that the external ray $\{ z \in \mathfrak{B}^\bullet \: : \: \arg b^\bullet(z) = 2\pi t \}$ lands at $x$.
    \end{enumerate}
\end{lemma}

\begin{proof}
    Assume $\bullet = \infty$ for convenience.
    Since no finite critical point is contained in $Y^\infty_\Hq$, the basin $\mathfrak{B}^\bullet$ must be simply connected. To prove the rest, we will appeal to Douady-Ghys surgery.
    
    By \cite[Proposition 2.12]{Lim23}, there exists a degree $d_\infty$ polynomial $P$ admitting a Siegel disk $Z$ with rotation number $\theta$, and a quasiconformal map $\psi: \RS \to \RS$ that is conformal on $\mathfrak{B}^\infty$, sends $Y^\infty_\Hq$ onto $\RS \backslash \overline{Z}$, and $\psi \circ f(z) = P \circ \psi(z)$ for all $z \in Y^\infty_\Hq$. Every critical finite critical point of $P$ is contained in the boundary of $Z$. Then, it suffices to show that the basin of infinity $\mathfrak{B}_P$ of $P$ satisfies properties analogous to items (1) and (2).

    Firstly, by \cite{WYZZ} (or \cite{Pe96} for $d_\infty=2$), the boundary of $\mathfrak{B}_P$ is locally connected, so every external ray of $P$ lands. For every $t \in [0,1)$, denote by $\Gamma_t$ the external ray of $P$ with angle $t$. The ray $\Gamma_0$ lands at a repelling fixed point. We will consider the open disk $D = \mathfrak{B}_P \backslash \overline{\Gamma_0}$.

    Now, fix a point $x$ on the boundary of $Z$ that is not a pre-critical point. Then, for every $n \geq 1$, there exists a unique lift $D_{-n}$ of $D$ under $P^n$ such that its boundary contains $x$. By the uniqueness of $D_{-n}$, every external ray landing at $x$ must be contained in $D_{-n}$ for all $n$. Observe that each $D_{-n}$ is bounded by a pair of external rays of the form $\Gamma_{t_n}$, $\Gamma_{t_n + d_\infty^{-n}}$ and we have $[t_{n+1},t_{n+1} + d_\infty^{-n-1}] \subset [t_n,t_n + d_\infty^{-n}]$. As $n \to \infty$, $t_n$ converges to a unique limit $t_* \in \R \backslash \Z$ and $\Gamma_{t_*}$ is the unique external ray landing at $x$.
\end{proof}

Let $f: \Hq \to \Hq$ be a $(d_0,d_\infty)$-critical quasicircle map and let $c \in \Hq$ be its critical point. We call an annular neighborhood $A$ of $\Hq$ \emph{$f$-relevant} if it satisfies the following properties.
\begin{enumerate}[label=(R\textsubscript{\arabic*})]
    \item\label{R01} The map $f$ admits a holomorphic extension to an annular neighborhood $A$ of $\Hq$ on which $c$ is the only critical point.
    \item\label{R02} The annulus $A$ can be decomposed into a disjoint union of an open disk neighborhood $B$ of $c$ and a topological rectangle $R$ intersecting $\Hq$ along $\Hq \backslash B$.
    \item\label{R03} On $B$, $f$ is a degree $d_0+d_\infty-1$ covering map branched only at $c$.
    \item \label{R04} The preimage of $f(B)\cap \Hq$ under $f$ is the union of the interval $B \cap \Hq$, $2d_\infty -2$ pairwise disjoint open quasiarcs in $Y_\Hq^\infty$ connecting $c$ and $\partial B \cap Y_\Hq^\infty$, and $2d_0-2$ pairwise disjoint open quasiarcs in $Y_\Hq^0$ connecting $c$ and $\partial B \cap Y_\Hq^0$.
    \item\label{R05} On the interior of $R$, $f$ is a conformal isomorphism onto the interior of $f(R)$, and the preimage of $f(R)$ under $f|_A$ is precisely $R$.
\end{enumerate}

Let us denote by 
\[
\mathcal{HQ}(d_0,d_\infty,N,K,\mu)
\]
the space of $(d_0,d_\infty)$-critical $K$-quasicircle maps $f: \Hq \to \Hq$ with rotation number in $\Theta_N$ that admits an $f$-relevant $2\mu$-collar neighborhood $A$ of $\Hq$ whose image $f(A)$ contains a $\mu$-collar neighborhood of $\Hq$. For example, for any $\theta \in \Theta_N$, $F_{d_0,d_\infty,\theta}$ lies in $\mathcal{HQ}(d_0,d_\infty,N,K,\mu)$ where $K$ and $\mu$ depend only on $d_0$, $d_\infty$, and $N$.

For any annular neighborhood $A$ of a quasicircle $\Hq$ and for $\bullet \in \{0,\infty\}$, we denote by $A^\bullet$ the annulus $A \cap Y_\Hq^\bullet$. Given any $\mu>0$, we say that an open neighborhood $A$ of $\Hq$ is a \emph{$\mu$-collar neighborhood} of $\Hq$ if $\modu(A^0)\geq \mu$ and $\modu(A^\infty) \geq \mu$.

\subsection{Quasicritical circle maps}
\label{ss:uni-herman-qc}

To study critical quasicircle maps, we will make use of Avila-Lyubich's theory of quasicritical circle maps in \cite[\S3]{AL22}.

\begin{definition}
    For any integer $d\geq 2$, a \emph{$d$-quasicritical circle map} is an orientation-preserving homeomorphism $g: \T \to \T$ of the circle with the following properties.
    \begin{enumerate}[label=(Q\textsubscript{\arabic*})]
        \item\label{QCC1} The map $g$ admits a $\T$-symmetric quasiregular extension of the form $B_{d,\alpha} \circ h$ on some $\T$-symmetric annular neighborhood $A$ of $\T$ where $\alpha \in [0,1)$ and $h$ is some $\T$-symmetric quasiconformal map on $\C$.
        \item\label{QCC2} The annulus $A$ can be decomposed into a disjoint union of a $\T$-symmetric open disk neighborhood $B$ of $1$ and a $\T$-symmetric topological rectangle $R$ intersecting $\T$ along $\T \backslash B$.
        \item\label{QCC3} On $B$, $g$ is a degree $2d-1$ quasiregular covering map branched only at $1$, and it is holomorphic at the set of points $z$ in $B$ such that $\T$ does not separate $z$ and $f(z)$. 
        \item \label{QCC4} The preimage of $g(B)\cap \T$ under $g$ is the union of the interval $B \cap \T$, $2d-2$ pairwise disjoint open quasiarcs in $Y_\T^\infty$ connecting $1$ and $\partial B \cap Y_\T^\infty$, and $2d-2$ pairwise disjoint open quasiarcs in $Y_\T^0$ connecting $1$ and $\partial B \cap Y_\T^0$.
        \item\label{QCC5} On the interior of $R$, $g$ is a conformal isomorphism onto the interior of $g(R)$, and the preimage of $g(R)$ under $f|_A$ is precisely $R$.
    \end{enumerate}
\end{definition}

Denote by $\Cir(d,N,K,\delta)$ the space of all $d$-quasicritical circle maps $g: \T \to \T$ such that the rotation number $\theta$ of $g$ is in $\Theta_N$, the map $h$ in \ref{QCC1} is $K$-quasiconformal, and that there exists a $2\delta$-collar annular neighborhood $A$ of $\T$ satisfying \ref{QCC1}--\ref{QCC5} whose image $g(A)$ is also a $\delta$-collar neighborhood of $\T$. Note that our parameters differ slightly from those used by Avila and Lyubich, but it is not difficult to show that they encode equivalent amount of information.

Even though only $2$-quasicritical circle maps are discussed in \cite[\S3]{AL22}, the results and proofs still hold for general $d$-quasicritical circle maps. Such maps admit the usual cross ratio distortion bounds and, as a result, they are quasisymmetrically rigid.

\begin{theorem}[{\cite[Theorem 3.9]{AL22}}]
\label{cqc-real-bounds}
    Every quasicritical circle map $g: \T \to \T$ in $\Cir(d,N,K,\delta)$ is quasisymmetrically conjugate to the irrational rotation with dilatation depending only on $(d,N,K,\delta)$. 
\end{theorem}

Similar to critical circle maps, quasicritical circle maps also admit complex bounds, since the main ingredients of the proof, namely real bounds and Schwarz lemma, are available. We will apply complex bounds later in the proof of Theorem \ref{thm:complex-bounds}.

The primary motivation behind introducing quasicritical circle maps is that critical quasicircle maps can be identified as a gluing of two quasicritical circle maps.

\begin{proposition}
\label{welding}
    Every map $f: \Hq \to \Hq$ in $\mathcal{HQ}(d_0,d_\infty,N,K,\mu)$ is a welding of two quasicritical circle maps. There is an annular neighborhood $A$ of $\Hq$ on which $f$ is holomorphic, and a pair of quasicritical circle maps $g_0$ and $g_\infty$ such that for each $\bullet \in \{0,\infty\}$,
    \begin{enumerate}[label=\textnormal{(\arabic*)}]
        \item $g_\bullet$ is in $\Cir(d_\bullet,N,L,\delta)$ for some $L=L(K)>1$ and $\delta= \delta(d_0,d_\infty,K,\mu)>0$;
        \item there is an $L$-quasiconformal map $\phi_\bullet: \RS \to \RS$ that maps $\Hq$ to $\T$ and conjugates $f|_{A \cap \overline{Y_\Hq^\bullet}}$ and $g_\bullet|_{\phi_\bullet(A) \cap \overline{Y_\T^\bullet}}$.
    \end{enumerate}
\end{proposition}

\begin{proof}
    Let $A$ be an $f$-relevant neighborhood of $\Hq$. For $\bullet \in \{0,\infty\}$, let $\phi_\bullet : Y_\Hq^\bullet \to Y_\T^\bullet$ denote the Riemann mapping fixing $\bullet$ whose continuous extension to the boundary sends the critical point of $f$ to $1$. Since $\Hq$ is a $K$-quasicircle, the map $\phi_\bullet$ extends to a global $L(K)$-quasiconformal map sending $\Hq$ to $\T$. Let $g_\bullet := \phi_\bullet \circ f \circ \phi_\bullet^{-1}$ on $\phi_\bullet(A^\bullet)$ and apply the Schwarz reflection principle to extend $g_\bullet$ to a $\T$-symmetric quasiregular map that restricts to a self homeomorphism of $\T$. Properties \ref{R01}--\ref{R05} for $f$ immediately transfer to \ref{QCC1}--\ref{QCC5} for $g_\bullet$, so $g_\bullet$ is the desired quasicritical circle map.
\end{proof}

This proposition is the key towards transferring known results on quasicritical circle maps to critical quasicircle maps. For instance, we have a quantitative version of Theorem \ref{petersen}.

\begin{lemma}
\label{regularity}
    Given a map $f: \Hq \to \Hq$ in $\mathcal{HQ}(d_0,d_\infty,N,K,\mu)$, there is a quasiconformal map $h$ on $\RS$ that restricts to a conjugacy between $f|_\Hq$ and the rigid rotation $R_\theta|_\T$, and has dilatation depending only on $(d_0,d_\infty,N,K,\mu)$.
\end{lemma}

\begin{proof}
    This follows directly from Theorems \ref{cqc-real-bounds} and \ref{welding}.
\end{proof}

\begin{remark}
    In general, if $f$ is a multicritical quasicircle map with arbitrary irrational rotation number, then $f$ is still conjugate to irrational rotation. Indeed, similar to Proposition \ref{welding}, $f$ is a conformal welding of two \emph{multi-quasicritical circle maps} $g_0$ and $g_\infty$. In \cite[Theorem 1.5]{Pe04}, Petersen showed that such maps satisfy the usual cross ratio distortion bounds, which in turn implies that they do not admit any wandering interval. (Compare with \cite[\S3]{Pe00}.)
\end{remark}

\subsection{Commuting pairs}
\label{ss:commuting-pairs}

Before we delve into a discussion on renormalization, let us define the abstract notion of commuting pairs relevant in our context.

Let us denote by $\mathbb{H}$ and $-\mathbb{H}$ the standard upper and lower half planes in $\C$ respectively.

\begin{definition}
\label{def:com-pair}
    Let $\I \Subset \C$ be a closed quasiarc containing $0$ on its interior. A \emph{commuting pair} $\zeta$ based on $\mathbf{I}$ is a pair of orientation preserving analytic homeomorphisms 
    \[
    \zeta = \left(f_-: I_- \to f_-(I_-), \: 
    f_+: I_+ \to f_+(I_+) \right)
    \]
    with the following properties.
    \begin{enumerate}[label = (P\textsubscript{\arabic*})]
        \item\label{P-1} $I_-$ and $I_+$ are closed subintervals of $\mathbf{I}$ of the form $[f_+(0),0]$ and $[0,f_-(0)]$ respectively such that $\mathbf{I} = I_- \cup I_+ = f_-(I_-) \cup f_+(I_+)$ and $I_- \cap I_+ = \{0\}$.
        \item\label{P-2} For all $x \in I_\pm \backslash \{0\}$, $f_\pm'(x) \neq 0$.
        \item\label{P-3} Both $f_-$ and $f_+$ admit holomorphic extensions to a neighborhood $B$ of $0$ on which $f_-$ commutes with $f_+$ and $f_- \circ f_+ (\I \cap B) \subset I_-$.
    \end{enumerate}
    Additionally, a commuting pair $\zeta$ is a \emph{critical commuting pair} if
    \begin{enumerate}[label = (P\textsubscript{\arabic*}), start=4]
        \item\label{P-4} $0$ is a critical point of both $f_-$ and $f_+$.
    \end{enumerate}
    The quasiarc $\mathbf{I}$ is called the \emph{base} of $\zeta$. We say that $\zeta$ is \emph{normalized} if $f_+(0) = -1$. A critical commuting pair $\zeta$ is called a \emph{$(d_0,d_\infty)$-critical commuting pair} if for any quasiconformal map $\phi$ mapping $I_-$ and $I_+$ to real intervals $[-1,0]$ and $[0,1]$ respectively and for any sufficiently small round disk $D$ centered at $\phi(f_+(f_-(0)))$, the number of connected components of $\phi (f_+ \circ f_-)^{-1} \phi^{-1}(D \cap -\He)$ in $-\He$ is $d_\infty$, whereas the number of connected components of $\phi (f_+ \circ f_-)^{-1} \phi^{-1}(D \cap \He)$ in $\He$ is $d_0$. Refer to Figure \ref{fig:ccp} for an illustration when $(d_0,d_\infty)=(3,2)$.
\end{definition}

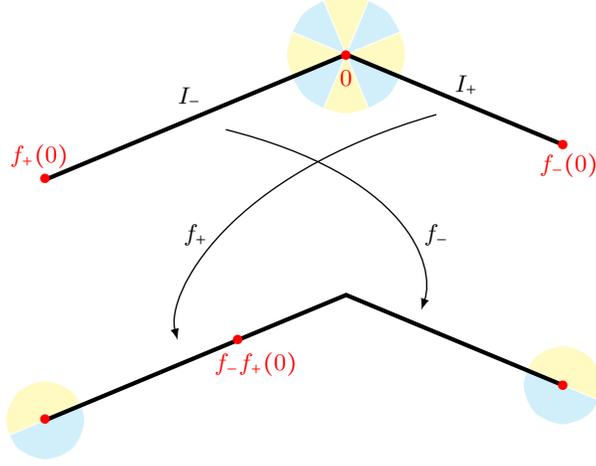
\begin{figure}
    \centering
    \begin{tikzpicture}[scale=0.8]
        \draw[white, fill=cyan!40!white,opacity=0.4] (0,0) -- (0.9,0.375) .. controls (0.72,0.72) .. (0.375,0.9) -- (0,0);
        \draw[white, fill=cyan!40!white,opacity=0.4] (0,0) -- (-0.9,0.375) .. controls (-0.72,0.72) .. (-0.375,0.9) -- (0,0);
        \draw[white, fill=cyan!40!white,opacity=0.4] (0,0) -- (0.9,-0.375) .. controls (0.72,-0.72) .. (0.375,-0.9) -- (0,0);
        \draw[white, fill=cyan!40!white,opacity=0.4] (0,0) -- (-0.9,-0.375) .. controls (-0.72,-0.72) .. (-0.375,-0.9) -- (0,0);
        \draw[white, fill=yellow!50!white,opacity=0.6] (0,0) -- (0.9,0.375) .. controls (1.0182,0) .. (0.9,-0.375) -- (0,0);
        \draw[white, fill=yellow!50!white,opacity=0.6] (0,0) -- (-0.9,0.375) .. controls (-1.0182,0) .. (-0.9,-0.375) -- (0,0);
        \draw[white, fill=yellow!50!white,opacity=0.6] (0,0) -- (0.375,0.9) .. controls (0,1.0182) .. (-0.375,0.9) -- (0,0);
        \draw[white, fill=yellow!50!white,opacity=0.6] (0,0) -- (0.375,0.-0.9) .. controls (0,-1.0182) .. (-0.375,-0.9) -- (0,0);
        \draw[white,fill=yellow!50!white,opacity=0.6] 
        (-4.4,-5.8333) .. controls (-4.52,-5.603) .. (-4.75,-5.483) .. controls (-5,-5.4042) .. (-5.25,-5.483) .. controls (-5.48,-5.603) .. (-5.6,-5.8333) .. controls (-5.6788,-6.083) .. (-5.6,-6.333) -- (-4.4,-5.8333);
        \draw[white,fill=cyan!40!white,opacity=0.4] 
        (-4.4,-5.8333) .. controls (-4.3212,-6.083) .. (-4.4,-6.333) .. controls (-4.52,-6.563) .. (-4.75,-6.683) .. controls (-5,-6.7618) .. (-5.25,-6.683) .. controls (-5.48,-6.563) .. (-5.6,-6.333) -- (-4.4,-5.8333);
        \draw[white,fill=yellow!50!white,opacity=0.6]
        (3,-5.25) .. controls (3.12,-5.02) ..
        (3.35,-4.9) .. controls (3.6, -4.8212) ..
        (3.85,-4.9) .. controls (4.08,-5.02) ..
        (4.2,-5.25) .. controls (4.2788,-5.5) ..
        (4.2,-5.75) -- (3,-5.25);
        \draw[white,fill=cyan!40!white,opacity=0.4]
        (3,-5.25) .. controls (2.9212,-5.5) ..
        (3,-5.75) .. controls (3.12,-5.98) ..
        (3.35,-6.1) .. controls (3.6,-6.1788) ..
        (3.85,-6.1) .. controls (4.08,-5.98) ..
        (4.2,-5.75) -- (3,-5.25);
        
        \draw[ultra thick] (-5,-2.083) -- (0,0) -- (3.6,-1.5);
        \draw[ultra thick] (-5,-6.083) -- (0,-4) -- (3.6,-5.5);
        \draw[line width=0.5pt,-latex] (-2,-1.25) .. controls (1,-2) and (1.5,-3.5) .. (1.25,-4.25);
        \draw[line width=0.5pt,-latex] (1.5,-1) .. controls (-2,-2) and (-3,-4) .. (-2.8,-4.75);

        % labels
        \node [red, font=\bfseries] at (0,-0.03) {$\bullet$};
        \node [red, font=\bfseries] at (0,-0.4) {\small $0$};
        \node [red, font=\bfseries] at (3.6,-1.52) {$\bullet$};
        \node [red, font=\bfseries] at (-5,-2.088) {$\bullet$};
        \node [red, font=\bfseries] at (3.6,-5.52) {$\bullet$};
        \node [red, font=\bfseries] at (-5,-6.088) {$\bullet$};
        \node [red, font=\bfseries] at (-1.8,-4.76) {$\bullet$};
        \node [black, font=\bfseries] at (-2.6,-0.7) {\small $I_-$};
        \node [black, font=\bfseries] at (2,-0.5) {\small $I_+$};
        \node [red, font=\bfseries] at (3.7,-1.85) {\small $f_-(0)$};
        \node [red, font=\bfseries] at (-5.1,-1.7) {\small $f_+(0)$};
        \node [red, font=\bfseries] at (-1.5,-5.15) {\small $f_-f_+(0)$};
        \node [black, font=\bfseries] at (1.5,-3) {\small $f_-$};
        \node [black, font=\bfseries] at (-2.5,-3) {\small $f_+$};
\end{tikzpicture}

    \caption{A cartoon of a $(3,2)$-critical commuting pair.}
    \label{fig:ccp}
\end{figure}

\begin{definition}
\label{def:crit-comm-pair}
    We say that a $(d_0,d_\infty)$-critical commuting pair $\zeta = (f_-,f_+)$ is \emph{renormalizable} if there exists a positive integer $\chi=\chi(\zeta)$ that corresponds to the first time $f_-^{\chi+1}\circ f_+(0)$ lies in the interior of $I_+$. If renormalizable, we call the $(d_\infty,d_0)$-critical commuting pair 
    \[
        p\renorm \zeta := 
        \left(
        f_-^\chi \circ f_+|_{[0,f_-(0)]}, \:
        f_-|_{[f_-^\chi f_+(0),0]}
        \right)
    \]
    the \emph{pre-renormalization} of $\zeta$, and we call the normalized $(d_0,d_\infty)$-critical commuting pair obtained by conjugating $p\renorm \zeta$ with the antilinear map $z \mapsto -f_-(0) \bar{z}$ the \emph{renormalization} $\renorm \zeta$ of $\zeta$. 
    
    If $\renorm \zeta$ is again renormalizable, we call $\zeta$ twice renormalizable, and so on. If $\zeta$ is infinitely renormalizable, we define the \emph{rotation number} of $\zeta$ to be the irrational number
    \[
        \rot(\zeta) := [0; \chi(\zeta), \chi(\renorm \zeta), \chi(\renorm^2 \zeta), \ldots].
    \]
\end{definition}

In what follows, we only consider commuting pairs that are infinitely renormalizable. Our renormalization operator transforms the rotation number according to the Gauss map $G(x) := \left\{ \frac{1}{x} \right\}$.

\begin{lemma}
\label{rotation-number}
    For any critical commuting pair $\zeta$ and $n \geq 1$, $\rot(\renorm^n \zeta) = G^n(\rot(\zeta))$.
\end{lemma}

For any $a \in \C$, let us denote by $T_a(z) := z+a$ the translation by $a$. For any irrational $\theta \in (0,1)$, the (non-critical) commuting pair
\begin{equation}
    \label{translation-pair}
    \mathbf{T}_\theta = \left( T_\theta|_{[-1,0]}, \: T_{-1}|_{[0,\theta]} \right)
\end{equation}
on intervals along the real line is infinitely renormalizable with rotation number $\theta$. The pair of translations (\ref{translation-pair}) gives a combinatorial model for normalized critical commuting pairs of the same rotation number.

Gluing the two ends of the real interval $[\theta-1,\theta]$ by $T_1$ projects the modified pair of translations $(T_\theta|_{[\theta-1,0]}, T_{\theta-1}|_{[0,\theta]})$ into the standard irrational rotation $R_\theta$ on the unit circle $\T$. In general, one can convert a commuting pair to a quasicircle map as follows.

\begin{proposition}
\label{gluing}
    Let $\zeta= (f_-|_{I_-}, f_+|_{I_+})$ be a commuting pair. Let $G_\zeta$ be the gluing map which corresponds to identifying $z$ with $f_+(z)$ for every point $z$ in a neighborhood of $f_-(0)$. Then, $G_\zeta$ projects the pair $(f_-|_{[f_+f_-(0),0]}, f_+f_-|_{[0,f_-(0)]})$ into a quasicircle map $f_\zeta: \Hq_\zeta \to \Hq_\zeta$ having the same rotation number as $\zeta$. If $\zeta$ is $(d_0,d_\infty)$-critical, then $f_\zeta: \Hq_\zeta \to \Hq_\zeta$ is a $(d_0,d_\infty)$-critical quasicircle map.
\end{proposition}

Conversely, we can obtain a commuting pair out of a $(d_0,d_\infty)$-critical quasicircle map $f: \Hq \to \Hq$ as follows. By conjugation with a linear map, let us assume that the critical point of $f$ is normalized at $1$. Replace $\Hq$ with its lift under the universal covering $z \mapsto e^{2\pi i z}$. In these logarithmic coordinates, $\Hq$ is a $\Z$-periodic quasicircle passing through $0$ and $\infty$. Replace $f$ with its corresponding lift $F$ admitting a critical point at $c_0:=0$ and a critical value $c_1:=F(0)$ located in the interval $[0,1] \subset \Hq$. Then, 
\[
\zeta_f := \left( F|_{[-1,0]}, \: T_{-1}|_{[0,c_1]} \right)
\]
is a commuting pair with the same rotation number as $f$. Applying the gluing operation from Proposition \ref{gluing} to $\zeta_f$ results in a $(d_0,d_\infty)$-critical quasicircle map conformally conjugate to $f: \Hq \to \Hq$.

We define renormalizations $\renorm^n f$ of $f$ to be the renormalizations of the commuting pair $\zeta_f$. This can be more explicitly described as follows. Denote by $\{p_n/q_n\}_{n \in \N}$ the best rational approximations of the rotation number $\theta$ of $f$. For every $n$, let $c_{q_n} := T_{-p_n}F^{q_n}(0)$. The \emph{$n$\textsuperscript{th} pre-renormalization} of $f$ is the critical commuting pair 
\[
    p\renorm^n f := p \renorm^n \zeta_f = \left( 
    T_{-p_n} F^{q_n}\big|_{[c_{q_{n-1}},0]}, \: 
    T_{-p_{n-1}} F^{q_{n-1}}\big|_{[0, c_{q_n}]} 
    \right),
\]
and the \emph{$n$\textsuperscript{th} renormalization} $\renorm ^n f$ of $f$ is the normalized $(d_0,d_\infty)$-critical commuting pair obtained by conjugating $p \renorm^n f$ with either the antilinear map $z \mapsto -c_{q_{n-1}}\bar{z}$ if $n$ is odd, or the linear map $z \mapsto -c_{q_{n-1}} z$ if $n$ is even.

Let us denote by $\mathcal{CP}(d_0,d_\infty,N,K,\mu)$ the space of all normalized $(d_0,d_\infty)$-critical commuting pairs $\zeta =(f_-, f_+): \I \to \I$ with rotation number in $\Theta_N$ such that the gluing procedure described in Proposition \ref{gluing} produces a critical quasicircle map in $\mathcal{HQ}(d_0,d_\infty,N,K,\mu)$. The following is a direct consequence of Lemma \ref{regularity}.

\begin{corollary}
\label{qs-conjugacy}
    Every critical commuting pair $\zeta: \I \to \I$ in $\mathcal{CP}(d_0,d_\infty,N,K,\mu)$ admits a unique quasisymmetric map $h_\zeta: [-1,\theta] \to \I$ conjugating the pair $\mathbf{T}_\theta$ of translations in \textnormal{(\ref{translation-pair})} with $\zeta$, where $\theta$ is the rotation number of $\zeta$. The dilatation of $h_\zeta$ depends only on $(d_0,d_\infty,N,K,\mu)$.
\end{corollary}

\subsection{Butterflies}
\label{ss:complex-bounds}

Consider a critical quasicircle map $f: \Hq\to \Hq$. From now on, it will be more convenient to use logarithmic coordinates by identifying $\Hq$ as a $\Z$-periodic quasicircle passing through $0$ and $\infty$ and $f$ as a map on a neighborhood of $\Hq$ in $\C$ with a critical point at $c_0:=0$ which commutes with the translation $T_1$. We will not notationally distinguish the sets $\Hq$, $Y_\Hq^0$, $Y_\Hq^\infty$ from their respective quotients in $\C / \Z$.

\begin{definition}
    A \emph{bowtie} is a quadruplet of Jordan domains $(V,\domO_-, \domO_+,\domO_\times)$ in $\C$ together with a quasicircle $\Hq$ containing $0$ and $\infty$ satisfying the following properties. 
    \begin{enumerate}[label=(B\textsubscript{\arabic*})]
        \item\label{bow:1} $\domO_-$, $\domO_+$, and $\domO_\times$ are compactly contained in $\ran$.
        \item\label{bow:2} $\domO_- \cap \domO_+ = \emptyset$ and $\overline{\domO_-} \cap \overline{\domO_+} = \{0\} \subset \domO_\times$.
        \item\label{bow:3} $\domO_- \backslash \domO_\times$, $\domO_\times \backslash \domO_-$, $\domO_+ \backslash \domO_\times$, and $\domO_\times \backslash \domO_+$ are all non-empty and connected.
        \item\label{bow:4} $J_- := \Hq \cap \overline{\domO_-}$, $J_+ := \Hq \cap \overline{\domO_+}$, and $J_\times := \Hq \cap \overline{\domO_\times}$ are closed intervals in $\Hq$, and their interiors are precisely $\mathring{J}_- := \Hq \cap \domO_-$, $\mathring{J}_+ := \Hq \cap \domO_+$ and $\mathring{J}_0 := \Hq \cap \domO_\times$ respectively.
    \end{enumerate}
    We call $\Hq$ the \emph{axis} of the bowtie.
\end{definition}

\begin{definition}
    A $(d_0,d_\infty)$-\emph{critical butterfly} $\butterfly$ is a pair of holomorphic maps $(f_-,f_+)$ together with a bowtie $(V, \domO_-, \domO_+,\domO_\times)$ with some axis $\Hq$ satisfying the following properties.
    \begin{enumerate}[label=(B\textsubscript{\arabic*}), start=5]
        \item\label{bow:5} $f_\pm$ is a univalent map from $\domO_\pm$ onto $(\ran \backslash \Hq) \cup f_\pm(\mathring{J}_\pm)$. 
        \item\label{bow:6} Both $f_-$ and $f_+$ extend holomorphically to $\domO_\times$ on which they commute. On $\domO_\times$, the map $f_-\circ f_+$ is a degree $d_0+d_\infty-1$ covering map onto $(\ran \backslash \Hq) \cup f_- f_+(\mathring{J}_\times)$ branched only at $0$.
        \item\label{bow:7} $\Hq$ is $f_\pm$-invariant, that is, whenever $f_\pm$ extends holomorphically to a neighborhood $E$ of a point $x \in \Hq$, then $f_\pm$ sends $E \cap \Hq$ to a subset of $\Hq$. 
        \item\label{bow:8} $I_- := [f_+(0),0]$ is a subset of $J_-$, $I_+ := [0,f_-(0)]$ is a subset of $J_+$, and $(f_-|_{I_-}, f_+|_{I_+})$ is a $(d_0,d_\infty)$-critical commuting pair.
        \item\label{bow:9} There is some integer $m\geq 1$ such that $f_-(0) = f_+^m(b_+)$ and $f_+(0)=f_-(b_-)$ where $J_- =[b_-,0]$ and $J_+ = [0,b_+]$.
    \end{enumerate}
    The \emph{axis} of $\butterfly$ is the quasicircle $\Hq$, the \emph{height} of $\butterfly$ is the integer $m$, and the \emph{rotation number} $\rot(\butterfly)$ of $\butterfly$ is the rotation number of the commuting pair $(f_-|_{I_-}, f_+|_{I_+})$. The interval 
    \[
    \textbf{I}:=[f_+(0),f_-(0)]=I_- \cup I_+ \subset \Hq
    \]
    is called the \emph{base} of $\butterfly$, whereas 
    \[
    \J:=[b_-,b_+]=J_- \cup J_+ \subset \Hq
    \]
    is called the \emph{extended base} of $\butterfly$. We say that $\butterfly$ is \emph{normalized} if $f_-(0)=-1$. The \emph{domain} of $\butterfly$ is the Jordan domain 
    \[
    U := \domO_- \cup \domO_\times \cup \domO_+.
    \]
    The \emph{shadow} of a butterfly $\butterfly$ is the piecewise holomorphic map $F: \dom \to \ran$ where
    \[
    F = \begin{cases}
    f_- & \text{ on } \domO_-,\\
    f_+ &\text{ on } \domO_+,\\
    f_- \circ f_+ &\text{ on } \domO_\times \backslash (\domO_- \cup \domO_+).
    \end{cases}
    \]
    The \emph{limit set} of $\butterfly$ is the non-escaping set of $F$, namely
    \[
    \Lambda_{\butterfly} := \bigcap_{n\geq0} F^{-n}(\dom).
    \]
\end{definition}

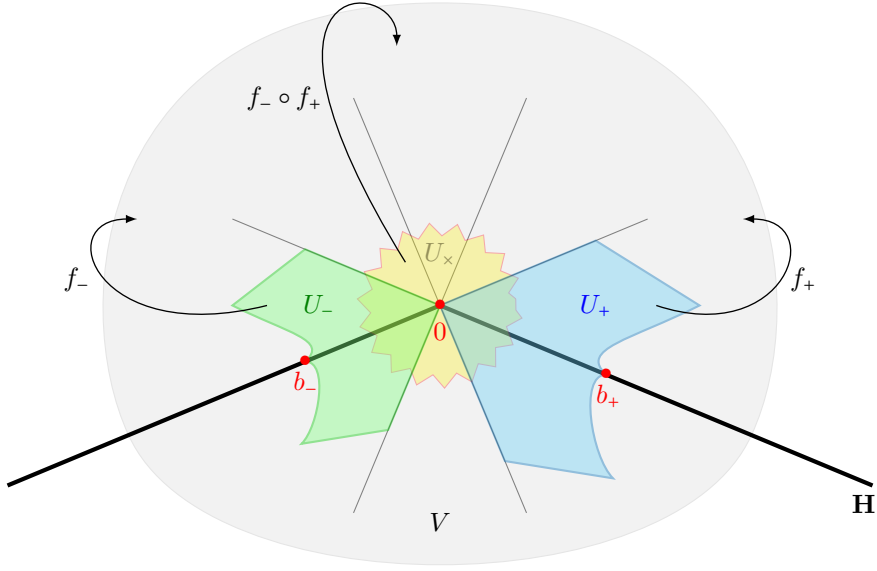
\begin{figure}
    \centering
    \begin{tikzpicture}[scale=1.15]
        \draw[gray,fill=gray!50!white,opacity=0.2] (3.6,-1.5) .. controls (4.3,0) and (4,3.5) .. (0,3.5) .. controls (-4,3.5) and (-4.3,0) .. (-3.6,-1.5) .. controls (-2.7,-3.5) and (2.7,-3.5) .. (3.6,-1.5);
        \draw[ultra thick] (-5,-2.083) -- (0,0) -- (5,-2.083);
        \draw[gray] (0,0) -- (2.4,1);
        \draw[gray] (0,0) -- (-2.4,1);
        \draw[gray] (0,0) -- (1, -2.4);
        \draw[gray] (0,0) -- (-1, -2.4);
        \draw[gray] (0,0) -- (1, 2.4);
        \draw[gray] (0,0) -- (-1, 2.4);
        \filldraw[red, fill=yellow, opacity=0.3, decorate, decoration={zigzag,segment length=9.8pt}, rotate around={5:(0,0)}] (0,0) circle (25pt);
        \draw[thick,cyan!50!blue, fill=cyan!50!white, opacity=0.4] (0,0) -- (1.8,0.75) -- (3,0) .. controls (2.2,-0.4) and (1.6,-0.5) .. (1.92,-0.8) .. controls (1.6,-0.8) and (1.75,-1.75).. (2,-2) -- (0.75,-1.8) -- (0,0);
        \draw[thick,green!80!black, fill=green!50!white, opacity=0.4] (0,0) -- (-1.56,0.65) -- (-2.4,0) .. controls (-1.76,-0.32) and (-1.28,-0.4) .. (-1.56,-0.65) .. controls (-1.28,-0.64) and (-1.4,-1.4).. (-1.6,-1.6) -- (-0.6,-1.44) -- (0,0);
        \draw[line width=0.5pt,-latex] (-2,0) .. controls (-4,-0.5) and (-4.5,1) .. (-3.5,1);
        \draw[line width=0.5pt,-latex] (2.5,0) .. controls (4,-0.5) and (4.5,1) .. (3.5,1);
        \draw[line width=0.5pt,-latex] (-0.4,0.5) .. controls (-2.5,3.9) and (-0.5,3.8) .. (-0.5,3);

        % labels
        \node [black, font=\bfseries] at (4.9,-2.3) {$\Hq$};
        \node [green!50!black, font=\bfseries] at (-1.4,0) {$U_-$};
        \node [blue, font=\bfseries] at (1.8,0) {$U_+$};
        \node [yellow!50!black, font=\bfseries] at (0,0.6) {$U_\times$};
        \node [black, font=\bfseries] at (0,-2.5) {$V$};
        \node [red, font=\bfseries] at (0,0) {$\bullet$};
        \node [red, font=\bfseries] at (0,-0.3) {$0$};
        \node [red, font=\bfseries] at (1.92,-0.8) {$\bullet$};
        \node [red, font=\bfseries] at (-1.56,-0.65) {$\bullet$};
        \node [red, font=\bfseries] at (1.95,-1.05) {$b_+$};
        \node [red, font=\bfseries] at (-1.55,-0.9) {$b_-$};
        \node [black, font=\bfseries] at (-4.2,0.3) {$f_-$};
        \node [black, font=\bfseries] at (4.2,0.3) {$f_+$};
        \node [black, font=\bfseries] at (-1.8,2.4) {$f_-\circ f_+$};
\end{tikzpicture}

    \caption{A $(3,2)$-critical butterfly.}
    \label{fig:butterfly}
\end{figure}

Following de Faria and de Melo \cite{dFdM2}, we will impose geometric assumptions on our butterflies.

\begin{definition}
\label{def:butterfly}
    For $K> 1$, a normalized butterfly $\butterfly$ with axis $\Hq$ is called a \emph{$K$-butterfly} if the following conditions are satisfied.
    \begin{enumerate}[label=(G\textsubscript{\arabic*})]
        \item\label{cond:1} $\modu \left(\ran \backslash \overline{\dom}\right) \geq K^{-1}$.
        \item\label{cond:2} The components of $\C \backslash \Hq$, $\C \backslash (\partial \ran\cup\Hq)$, and $\ran \backslash \left( \overline{\dom} \cup \Hq \right)$ are $K$-quasidisks.
        \item\label{cond:3} Any two points in the set $\{b_-,f_+(0),0,f_-(0)=1, b_+\}$ are at least $K^{-1}$ away from each other.
        \item\label{cond:4} The annulus $\overline{V} \backslash \domO_\times$ is contained in $\{ K^{-1} < |z| < K \}$.
    \end{enumerate}
    In general, a butterfly is a \emph{$K$-butterfly} if it is linearly conjugate to a normalized $K$-butterfly.
\end{definition}

We endow the space of butterflies with the topology where a sequence of butterflies $\butterfly_n = \{(f_{-,n},f_{+,n}), (V_n, U_{-,n}, U_{+,n}, U_{\times,n})\}$ converges to the butterfly $\butterfly=\{(f_{-},f_{+}), (V, U_{-}, U_{+}, U_{\times})\}$ if 
\begin{enumerate}[label=(\roman*)]
    \item in the Carath\'eodory topology (refer to \cite[\S5]{McM94}),
    \begin{align*}
        \left( U_{-,n},f_{+,n}(0) \right) &\to \left( U_{-},f_{+}(0) \right), & (U_{\times,n},0) &\to (U_{\times},0),\\
        \left( U_{+,n},f_{-,n}(0) \right) &\to \left( U_{+},f_{-}(0) \right), & \left( V_n, f_{-,n}f_{+,n}(0) \right) &\to \left( V, f_{-}f_{+}(0) \right);
    \end{align*}
    \item $\butterfly_n$ has an axis converging to an axis of $\butterfly$ in Hausdorff metric;
    \item $f_{-,n}$ converges uniformly to $f_-$ on compact subsets of $U_{-} \cup U_{\times}$, and $f_{+,n}$ converges uniformly to $f_+$ on compact subsets of $U_{+} \cup U_{\times}$.
\end{enumerate}

\begin{proposition}
     \label{compactness}
     The space of normalized $(d_0,d_\infty)$-critical $K$-butterflies with rotation number in $\Theta_N$ is compact.
\end{proposition}

\begin{proof}
    The compactness the geometric objects in the sense of (i) and (ii) follow from the compactness of appropriately normalized $K$-quasiconformal maps as well as properties \ref{cond:1}, \ref{cond:2}, \ref{cond:3}, and \ref{cond:4}. The subsequential convergence to limit maps $(f_-,f_+)$ as described in (iii) follows from \cite[Theorem 5.6]{McM94}.
\end{proof}

Similar to critical circle maps, critical quasicircle maps also admit complex a priori bounds.

\begin{theorem}[Complex bounds]
\label{thm:complex-bounds}
    Given $f: \Hq \to \Hq$ in $\mathcal{HQ}(d_0,d_\infty,N,K,\mu)$, there exist constants $n_0 \in \N$ and $K'>1$ depending only on $(d_0,d_\infty,N,K,\mu)$ such that for all $n \geq n_0$, the $n^{\text{th}}$ pre-renormalization of $f$ extends to a $K'$-butterfly $\butterfly_n: \dom_n \to \ran_n$.
\end{theorem}

Recall from Proposition \ref{welding} that $f$ is a welding of two quasicritical circle maps $g_0$ and $g_\infty$. By \cite[\S3.5]{AL22}, the pre-renormalizations of $g_0$ and $g_\infty$ have holomorphic extensions admitting a butterfly structure with a priori bounds. In the proof below, we will glue the half-butterflies of the two maps in order to obtain a butterfly for the pre-renormalization of $f$.

\begin{proof}
    Let $U$, $\phi_0$, $\phi_\infty$, $g_0 \in \Cir(d_0,N,L,\delta)$, and $g_\infty \in \Cir(d_\infty,N,L,\delta)$ be from Proposition \ref{welding}. We will outline the construction of butterflies for $g_0$ and $g_\infty$ and then glue them to the desired butterfly for $f$. Let us work in logarithmic coordinates, in which $\Hq \subset \C /\Z$ is a quasicircle passing through $0$ and the critical point of $f$ is at $0$. For all $j \in \Z$, we write $c_j := f^j(0)$. For $\bullet \in \{0,\infty\}$ and $j \in \Z$, let $c^\bullet_0=0$ denote the critical point of $g_\bullet$ and let $c^\bullet_j := g_\bullet^j(0)$. 
    
    Let $\{p_n/q_n\}_{n \in \N}$ denote the best rational approximations of the rotation number $\theta$ of $f$. For $\bullet \in \{0,\infty\}$ and $n \geq 2$, observe that the two critical points of $g^{q_n}_\bullet$ that are next to $c^\bullet_{q_{n-1}}$ are $c^\bullet_0$ and $c^\bullet_{q_{n-1}-q_n}$. From now on, we will fix $n \in \N$ larger than some constant $m \in \N$ that is to be determined. Let us recall the construction of butterflies extending the $n$\textsuperscript{th} pre-renormalization of $g_\bullet$. 
    
    For any $k \geq 1$, let $D^\bullet_k$ be the open round disk such that $\partial D^\bullet_k$ intersects $\T$ orthogonally and $D^\bullet_k \cap \T$ is the open interval $\left( c^\bullet_{q_{k+1}}, c^\bullet_{q_k-q_{k+1}}\right) \subset \T$. For $n \gg m$, there exists a $\T$-symmetric univalent lift $A^\bullet_{n,m}$ of $(D^\bullet_{n-m} \backslash \T) \cup \left( c^\bullet_{q_{n-1}},c^\bullet_{q_n} \right)$ under $g_\bullet^{q_{n}}$ intersecting $\T$ on the interval $\left( c^\bullet_{q_{n-1}-q_n} , c^\bullet_0 \right)$. Similarly, there also exists a $\T$-symmetric univalent lift $B^\bullet_{n,m}$ of $(D^\bullet_{n-m} \backslash \T) \cup \left( c^\bullet_{q_{n-1}},c^\bullet_{q_{n-2}} \right)$ under $g_\bullet^{q_{n-1}}$ intersecting $\T$ on the interval $\left( c^\bullet_0,c^\bullet_{q_{n-2}-q_{n-1}} \right)$.

    \begin{claim}
    For any $\varepsilon >0$, there are some constants $n_0,m \in \N$ depending only on $\varepsilon$ and $(d_\bullet, N,L,\delta)$ such that $n_0\geq m$ and for all $n \geq n_0$,
    \begin{equation}
    \label{eqn:key-estimate}
        \max\{\diam(A^\bullet_{n,m}), \diam(B^\bullet_{n,m})\} \leq \varepsilon \cdot  \diam(D^\bullet_{n-m}).
    \end{equation}
    \end{claim}
    
    \begin{proof}
    Let $I_n^\bullet$ denote the closed interval in $\T$ between $c_0^\bullet$ and $c_{q_n}^\bullet$. Based on the key estimates in \cite[(3.9)]{AL22} and \cite[Proposition 3.2]{dFdM2}, there are constants $b_1, b_2>0$ such that for all $z \in A_{n,m}^\bullet$,
    \begin{equation}
    \label{ineq:key-estimate}
        \frac{\dist\left( g_\bullet(z), g_\bullet(I_{n-1}^\bullet) \right)}{|g_\bullet(I_{n-1}^\bullet)|} \leq b_1 \cdot \frac{\dist\left( g_\bullet^{q_{n}}(z), I_{n-1}^\bullet \right)}{|I_{n-1}^\bullet|} + b_2.
    \end{equation}
    \cite[Lemma 3.6]{AL22} guarantees that near the critical point, the inverse branch of $g_\bullet$ is highly contracting in big scales relative to $I_{n-1}^\bullet$, which yields
    \begin{equation}
    \label{ineq:key-contraction}
        C \cdot \left( \frac{\diam(A_{n,m}^\bullet)}{|I_{n-1}^\bullet|} \right)^\sigma \leq b_1 \cdot \frac{\diam(D_{n-m}^\bullet)}{|I_{n-1}^\bullet|} + b_2
    \end{equation}
    for some constants $C>0$ and $\sigma>1$. Applying Theorem \ref{cqc-real-bounds} to $g_\bullet$, there are also constants $K_2>K_1>1$ such that for all sufficiently large $n$,
    \begin{equation}
    \label{ineq:key-real-bounds}
        K_1^m \leq \frac{\diam(D_{n-m}^\bullet)}{|I_{n-1}^\bullet|} \leq K_2^m.
    \end{equation}
    Note that all the intermediate constants above depend only on $(d_\bullet, N,L,\delta)$. Let us pick $\varepsilon>0$. By combining (\ref{ineq:key-contraction}) and (\ref{ineq:key-real-bounds}), for sufficiently large $m$, we have
    \[
        \diam(A_{n,m}^\bullet) \leq \varepsilon \cdot \diam(D_{n-m}^\bullet).
    \]
    We can repeat the same analysis for $B_{n,m}^\bullet$. 
    \end{proof}
    
    Let $D_{n-m}$, $A_{n,m}$ and $B_{n,m}$ be the interior of the closure of $\bigcup_\bullet \phi_\bullet^{-1}(D_{n-m}^\bullet \cap Y^\bullet_\T)$, $\bigcup_\bullet \phi_\bullet^{-1}(A_{n,m}^\bullet \cap Y^\bullet_\T)$, and $\bigcup_\bullet \phi_\bullet^{-1}(B_{n,m}^\bullet \cap Y^\bullet_\T)$ respectively. Since $\phi_0$ and $\phi_\infty$ are $L$-quasiconformal, the claim implies that there are some constants $n_0,m \in \N$ depending only on $(\varepsilon, d_0, d_\infty,N,K,\mu)$ such that $n_0\geq m$ and for $n\geq n_0$, 
    \begin{equation}
    \label{eqn:key-estimate-02}
        \max\{\diam(A_{n,m}), \diam(B_{n,m})\} \leq \varepsilon \cdot  \diam(D_{n-m}).
    \end{equation}
    
    Let $C_{n,m}$ denote the connected component of $f^{-q_{n-1}}(A_{n,m})$ containing the critical point $0$. Since $C_{n,m} \cap \Hq = (c_{-q_{n}}, c_{-q_{n-1}})$, the map $f^{q_{n-1}}: C_{n,m} \to A_{n,m}$ is a degree $d_0+d_\infty-1$ covering map branched exactly at $0$, and $f^{q_{n-1}}$ maps ${A_{n,m} \cap C_{n,m}}$ univalently onto $(A_{n,m} \backslash \Hq) \cup \left(c_{q_{n-1}-q_n}, 0 \right)$. By making $n_0$ higher if necessary, $A_{n,m} \cup B_{n,m}$ is contained in a neighborhood of the critical point $0$ in which $f\equiv \psi(z^{d_0+d_\infty-1})$ for some univalent map $\psi$ with universally bounded distortion, which implies that $\diam(C_{n,m}) \asymp \diam(A_{n,m} \cap C_{n,m})$. Therefore, (\ref{eqn:key-estimate-02}) can be upgraded to
    \begin{equation}
    \label{eqn:key-estimate-03}
        \max\{\diam(A_{n,m}), \diam(B_{n,m}), \diam(C_{n,m})\} \leq \varepsilon \cdot  \diam(D_{n-m}).
    \end{equation}
    
    By construction, the pointed disk $(D_{n-m},0)$ has bounded shape. We can select an appropriate $\varepsilon$ such that (\ref{eqn:key-estimate-03}) implies that the union $\dom_n := A_{n,m} \cup B_{n,m} \cup C_{n,m}$ is compactly contained in $\ran_n := D_{n-m}$ and $\ran_n \backslash \overline{\dom_n}$ is an annulus with modulus greater than some universal constant. Therefore, the pair $(f^{q_n}, f^{q_{n-1}})$ and the bowtie $(V_n, A_{n,m}, B_{n,m}, C_{n,m})$ form a butterfly $\butterfly_n$ with axis $\Hq$ that extends the $n$\textsuperscript{th} pre-renormalization of $f$ and clearly satisfies \ref{cond:1}.
    
    It is also clear from the construction that the components of $\RS\backslash \Hq$ and $\RS \backslash (\ran_n \cup \Hq)$ are $K'$-quasidisks. Every component of $V_n \backslash \left( \overline{U_n} \cup \Hq \right)$ is also a $K'$-quasidisk since its boundary is a union of quasiarcs meeting at definite angles. Hence, \ref{cond:2} holds. Condition \ref{cond:3} follows from Lemmas \ref{lem:qc-control} (2) and \ref{regularity}, and \ref{cond:4} follows from the construction of $\ran_n$ and Koebe distortion theorem.
\end{proof}

In the proof above, the butterfly extending the $n$\textsuperscript{th} pre-renormalization has height equal to $a_{n}+1$, where $a_{n}$ is the $n$\textsuperscript{th} term of the continued fraction expansion of the rotation number. If we apply the construction in the proof to a rational map in Proposition \ref{prop:prototype-example}, the corresponding limit set is contained in the Julia set of the rational map, which is nowhere dense.

\begin{corollary}
\label{cor:rational-case}
    For any $m \geq 2$ and any $\theta \in \Theta_N$, there exists a $(d_0,d_\infty)$-critical $K$-butterfly having rotation number $\theta$, height $m$, and a nowhere dense limit set, where $K$ depends only on $(d_0,d_\infty,m,N)$. 
\end{corollary}

This corollary will be applied in Lemma \ref{promote-qs-conjugacy}. By Proposition \ref{gluing}, commuting pairs also admit complex bounds.

\begin{theorem}
\label{thm:complex-bounds-butterflies}
    Given $\zeta$ in $\mathcal{CP}(d_0,d_\infty,N,K,\mu)$, there exist some $n_0 \in \N$ and $K'>1$ depending only on $(d_0,d_\infty,N,K,\mu)$ such that for all $n \geq n_0$, the $n^{\text{th}}$ pre-renormalization of $\zeta$ extends to a $K'$-butterfly $\butterfly_n: \dom_n \to \ran_n$.
\end{theorem}

\subsection{Quasiconformal rigidity}
\label{ss:qc-conj}

Let us fix a bounded type irrational number $\theta \in \Theta_N$ for some $N \geq 1$. Recall from Corollary \ref{qs-conjugacy} that two butterflies of the same criticality and bounded type rotation number must be quasisymmetrically conjugate on their bases.

\begin{lemma}
\label{promote-qs-conjugacy}
    Suppose two $(d_0,d_\infty)$-critical $K$-butterflies $\butterfly_1: \dom_1 \to \ran_1$ and $\butterfly_2:\dom_2 \to \ran_2$ have the same height $m$ and rotation number $\theta$. The unique quasisymmetric conjugacy between $\butterfly_1$ and $\butterfly_2$ on their bases extend to a quasiconformal conjugacy $h: \ran_1 \to \ran_2$ between $\butterfly_1$ and $\butterfly_2$ with dilatation depending only on $(d_0,d_\infty,N,K)$.
\end{lemma}

The proof is an application of the pullback argument, and it is similar to \cite[Theorem 3.1]{dF99}.

\begin{proof}
    By Corollary \ref{qs-conjugacy}, there exists a unique quasisymmetric conjugacy $h: \textbf{I}_1 \to \textbf{I}_2$ between $\butterfly_1$ and $\butterfly_2$ on their bases. Since $\butterfly_1$ and $\butterfly_2$ have the same height, we can extend $h$ to a quasisymmetric conjugacy on the extended bases $\J_1$ and $\J_2$ by setting $h(z) := f_{-,2}^{-1} \circ h \circ f_{-,1}(z)$ for $z \in J_{-,1} \backslash \mathbf{I}_1$ and $h(z) := f_{+,2}^{-k} \circ h \circ f_{+,1}^k(z)$ for $z \in J_{+,1} \backslash \mathbf{I}_1$ where $k \in \N$ is the first time $f_{+,1}^k(z)$ lies on $\mathbf{I}_1$.
    
    By \ref{cond:1} and \ref{cond:2}, we can perform quasiconformal interpolation and extend $h$ to a global $L$-quasiconformal map $h_0$ that is equivariant on the boundaries of the butterflies of $\butterfly_1$ and $\butterfly_2$. Note that the dilatation $L$ depends only on $(d_0, d_\infty,N,K)$.
    
    Next, we apply the pullback argument to obtain a sequence of $L$-quasiconformal maps $h_n$ as follows. Outside of $\dom_1$, we take $h_n \equiv h_{n-1}$; within $\dom_1$, we set $h_n$ to be the lift of $h_{n-1}: \ran_1 \to \ran_2$ via $\butterfly_1$ and $\butterfly_2$. By equivariance, $h_n$ is well-defined, and it gives a conjugacy between $\butterfly_1|_{\ran_1 \backslash F_1^{-n}(\dom_1)}$ and $\butterfly_2|_{\ran_2 \backslash F_2^{-n}(\dom_2)}$.
    
    By the compactness of the space of normalized $L$-quasiconformal maps, $h_n$ converges to a subsequential limit $h_\infty:\RS \to \RS$. Note that $h_n$ stabilizes pointwise outside of $\Lambda_{\butterfly_1}$. We will claim that $\Lambda_{\butterfly_1}$ is nowhere dense, which ultimately implies that the limit $h_\infty$ is unique in the sense that $h_n \to h_\infty$ as $n\to \infty$, and $h_\infty$ conjugates $\butterfly_1$ and $\butterfly_2$.

    Corollary \ref{cor:rational-case} guarantees the existence of a $(d_0,d_\infty)$-critical butterfly $\butterfly: \dom \to \ran$ that has the same rotation number $\theta$ and height $m$ and that its limit set $\Lambda_{\butterfly}$ is nowhere dense. By applying the same pullback argument above, we obtain a quasiconformal map $g: \ran \to \ran_1$ that restricts to a conjugacy between $\butterfly$ and $\butterfly_1$ on $\ran\backslash \Lambda_{\butterfly}$. Since $\Lambda_{\butterfly}$ is nowhere dense, then $g$ extends to a full conjugacy between $\butterfly$ and $\butterfly_1$, which then implies that $\Lambda_{\butterfly_1}$ is indeed nowhere dense.
\end{proof}

Next, we can spread around the quasiconformal conjugacy between butterflies of sufficiently deep renormalizations throughout the entire Herman curves. Compare with \cite[Theorem 3.19]{AL22}.

\begin{theorem}[Quasiconformal rigidity]
\label{qc-conjugacy}
    Given any two maps $f_1$ and $f_2$ in $\mathcal{HQ}(d_0,d_\infty,N,K,\mu)$ of the same rotation number, there is an $L$-quasiconformal map $h$ on $\RS$ that restricts to a conjugacy between $f_1$ and $f_2$ in some $\delta$-collar neighborhoods of their Herman quasicircles. The constants $L$ and $\delta$ depend only on $(d_0,d_\infty,N,K,\mu)$.
\end{theorem}

\begin{proof}
     Let $f_1: \Hq_1 \to \Hq_1$ and $f_2: \Hq_2 \to \Hq_2$ be two $(d_0,d_\infty)$-critical quasicircle maps of rotation number $\theta$. Without loss of generality, assume that $f_1$ is induced from the rational map $F_{d_0,d_\infty,\theta}$ from Proposition \ref{prop:prototype-example}.
    
    By Lemma \ref{regularity}, there is a global quasiconformal map $h$ that sends $\Hq_1$ to $\Hq_2$ and restricts to a conjugacy between $f_1|_{\Hq_1}$ and $f_2|_{\Hq_2}$. For $i\in\{1,2\}$ and sufficiently large $m \in \N$, Theorem \ref{thm:complex-bounds} states that the $m^{\text{th}}$ pre-renormalization of $f_i$ extends to a $K'$-butterfly $\butterfly_i = (f_{-,i},f_{+,i}): \dom_i \to \ran_i$. Clearly, $h$ induces a quasisymmetric conjugacy between $\butterfly_1$ and $\butterfly_2$ on their bases. By Lemma \ref{promote-qs-conjugacy}, $h$ can be modified to a quasiconformal map $h_\infty: V_1 \to V_2$ that conjugates $\butterfly_1: \dom_1 \to \ran_1$ and $\butterfly_2: \dom_1 \to \ran_1$.

    It remains to spread the conjugacy around the Herman curve. 
    To do this, we will apply Lemma \ref{lem:fatou-set-structure}, which asserts that given any point $x \in \Hq_1$ that is not an iterated preimage of the critical point $c$ of $f_1$, there is a unique pair of external rays from $0$ and $\infty$ landing at $x$; we denote by $\gamma_x$ the union of these rays.
    The rays $\gamma_{f_{+,1}(c)}$ and $\gamma_{f_{-,1}(c)}$ are disjoint away from $0$ and $\infty$, and they are mapped to $\gamma_{f_{-,1}f_{+,1}(c)}$ under $f_{-,1}$ and $f_{+,1}$ respectively.
    Denote by $\Pi_\pm$ the connected component of $\ran_1 \backslash \left( \gamma_{f_{-,1}f_{+,1}(c)} \cup \gamma_{f_{\pm,1}(c)} \right)$ containing the quasiarc $[f_{-,1}f_{+,1}(c),f_{\pm,1}(c)] \subset \Hq_1$. 
    Observe that $\Pi_\pm$ is simply connected, and the intersection of its boundary and a small neighborhood of $\Hq_1$ is the union of a segment of $\gamma_{f_{-,1}f_{+,1}(c)}$ and a segment of $\gamma_{f_{\pm,1}(c)}$. 
    
    Let us lift both $\Pi_+$ and $\Pi_-$ by iterates of $f_1$ until the first return to $\Pi_+ \cup \Pi_-$ and obtain a tiling of a neighborhood of $\Hq_1$.
    Let us also use $h_\infty\left(\Pi_\pm\right)$ to construct a similar dynamical tiling for $f_2$ around $\Hq_2$. 
    Lastly, lift the quasiconformal conjugacy $h_\infty: \Pi_+ \cup \Pi_- \to h_\infty\left( \Pi_+ \cup \Pi_-\right)$ to these dynamical tilings. The equivariant properties of the external rays ensures that this lift is a continuous extension of $h$ itself and moreover, it is a quasiconformal conjugacy between $f_1$ and $f_2$ on the neighborhood of their Herman quasicircles.
\end{proof}

\subsection{$C^{1+\alpha}$ rigidity}
\label{ss:coneplusalpha}
    
Consider a pair of $(d_0,d_\infty)$-critical quasicircle maps $f_1: \Hq_1 \to \Hq_1$ and $f_2: \Hq_2 \to \Hq_2$ of rotation number $\theta$. By Theorem \ref{qc-conjugacy}, there exist collars $A_1$ and $A_2$ for $f_1|_{\Hq_1}$ and $f_2|_{\Hq_2}$ respectively, and a global quasiconformal map $\phi: \RS \to \RS$ conjugating $f_1|_{A_1}$ and $f_2|_{A_2}$. Our goal is to improve the regularity of the quasiconformal conjugacy $\phi$ and prove Theorem \ref{thm:c1plusalpharigidity}. We will do so via McMullen's Dynamic Inflexibility Theorem.

Let $\textit{Hol}$ denote the set of all holomorphic maps $g: U \to \RS$ where $U$ is any open subset of $\RS$. Endow $\textit{Hol}$ with the topology where $g_n : W_n \to \RS$ converges to $g: W \to \RS$ if for every compact subset $X \subset W$, $W_n$ contains $X$ for all sufficiently large $n$ and $g_n \to g$ uniformly in $X$.

For each $i \in \{1,2\}$, define
\[
    \mathcal{F}(f_i) := \left\{ [g : U \to \RS] \in \textit{Hol} \: : \: f_i^i = g \circ f_i^j \text{ for some } j \in \N \text{ on some open } U \subset A_i \right\}.
\]
Let us pick a skinnier collar $A_1' \Subset A_1$ for $\Hq_1$ and let $A_2' = \phi(A_1')$. Consider the local filled Julia set $K^{\textnormal{loc}}_{A'_i}(f_i)$ of $f_i$ rel $A'_i$, which is a forward invariant compact subset of $A_i$. (See Definition \ref{def:local-filled-julia-set}.) The pair $\left(\mathcal{F}(f_i), K^{\textnormal{loc}}_{A'_i}(f_i)\right)$ forms a holomorphic dynamical system in the sense of McMullen \cite[\S9]{McM96}.

\begin{theorem}[{\cite[Theorem 9.15]{McM96}}]
\label{McM-criterion-01}
    Suppose there is a $K$-quasiconformal conjugacy $\phi$ between two holomorphic dynamical systems $(\mathcal{F}_1,K_1)$ and $(\mathcal{F}_2,K_2)$. If $(\mathcal{F}_1,K_1)$ is uniformly twisting and $K_1$ has a uniformly deep subset $S$, then $\phi$ is uniformly $C^{1+\alpha}$-conformal on $S$.
\end{theorem}

Roughly speaking, a holomorphic dynamical system $\left(\mathcal{F}, K\right)$ is \emph{uniformly twisting} if $\mathcal{F}$ has robust nonlinearity at every point in $K$ at every scale. A more precise definition of nonlinearity and uniform twisting can be found in \cite[\S9.3]{McM96}. In our discussion, we will only require the following criterion for uniform twisting.

\begin{proposition}[{\cite[Proposition 4.7]{McM98}}]
\label{McM-criterion-02}
    Consider a subset $\mathcal{F}$ of $\textit{Hol}$ and a compact subset $K$ of $\RS$. Suppose for any sequence of affine maps $\textnormal{af}_n(z) = \alpha_n (z- \beta_n)$ with $\alpha_n \to \infty$ and $\beta_n \in K$, there is a sequence of maps $g_n$ in $\mathcal{F}$ such that the rescaling $\textnormal{af}_n \circ g_n \circ \textnormal{af}_n^{-1}$ converges in subsequence to a non-constant holomorphic map in $\textit{Hol}$ with a critical point. Then, the dynamical system $(\mathcal{F},K)$ is uniformly twisting.
\end{proposition}

In our context, we will consider the pair of dynamical systems 
\[
\left( \mathcal{F}(f_1), K^{\textnormal{loc}}_{A_1'}(f_1) \right) \quad \text{and} \quad \left( \mathcal{F}(f_2), K^{\textnormal{loc}}_{A_2'}(f_2) \right)
\]
which are quasiconformally equivalent via $\phi$.
Let us assume without loss of generality that $f_1$ is a rational map of the form $F_{d_0,d_\infty,\theta}$ from Proposition \ref{prop:prototype-example}. In the proof of Lemma \ref{NILF-local}, we have shown that the map $f_1$ and the set $K^{\textnormal{loc}}_{A_1'}(f_1)$ satisfy properties \ref{S1}--\ref{S5} in the hypothesis of Proposition \ref{shen-criterion}. These properties immediately imply the hypothesis of Proposition \ref{McM-criterion-02}, and so $\left(\mathcal{F}(f_1), K^{\textnormal{loc}}_{A_1'}(f_1) \right)$ is indeed uniformly twisting. Recall from Theorem \ref{deep-point-theorem} that $\Hq_1$ is a a uniformly deep subset of the local filled Julia set $K^{\textnormal{loc}}_{A_1'}(f_1)$ with constants depending only on $(d_0,d_\infty,N)$. By Theorem \ref{McM-criterion-01}, the quasiconformal conjugacy $\phi$ between $f_1$ and $f_2$ is uniformly $C^{1+\alpha}$-conformal on $\Hq_1$. This completes the proof of Theorem \ref{thm:c1plusalpharigidity}.

\section{Consequences of \texorpdfstring{$C^{1+\alpha}$}{C} rigidity}
\label{sec:exp-convergence}

We end this paper with a discussion on a number of applications of $C^{1+\alpha}$ rigidity of critical quasicircle maps. We will again fix a triplet of integers $d_0\geq 2$, $d_\infty \geq 2$, $N \geq 1$, and a bounded type rotation number $\theta \in \Theta_N$. 

\subsection{Smoothness}
\label{ss:smoothness}

Consider a $(d_0,d_\infty)$-critical quasicircle map $f: \Hq \to \Hq$ of rotation number $\theta$.

\begin{corollary}
\label{cor:c1-smooth}
     The quasicircle $\Hq$ is $C^1$ smooth if and only if $d_0 = d_\infty$.
\end{corollary}

See Figure \ref{fig:smoothness} for an example of a $C^1$ smooth Herman curve that is not a Euclidean circle.

\begin{figure}
    \centering
    \includegraphics[width=\linewidth]{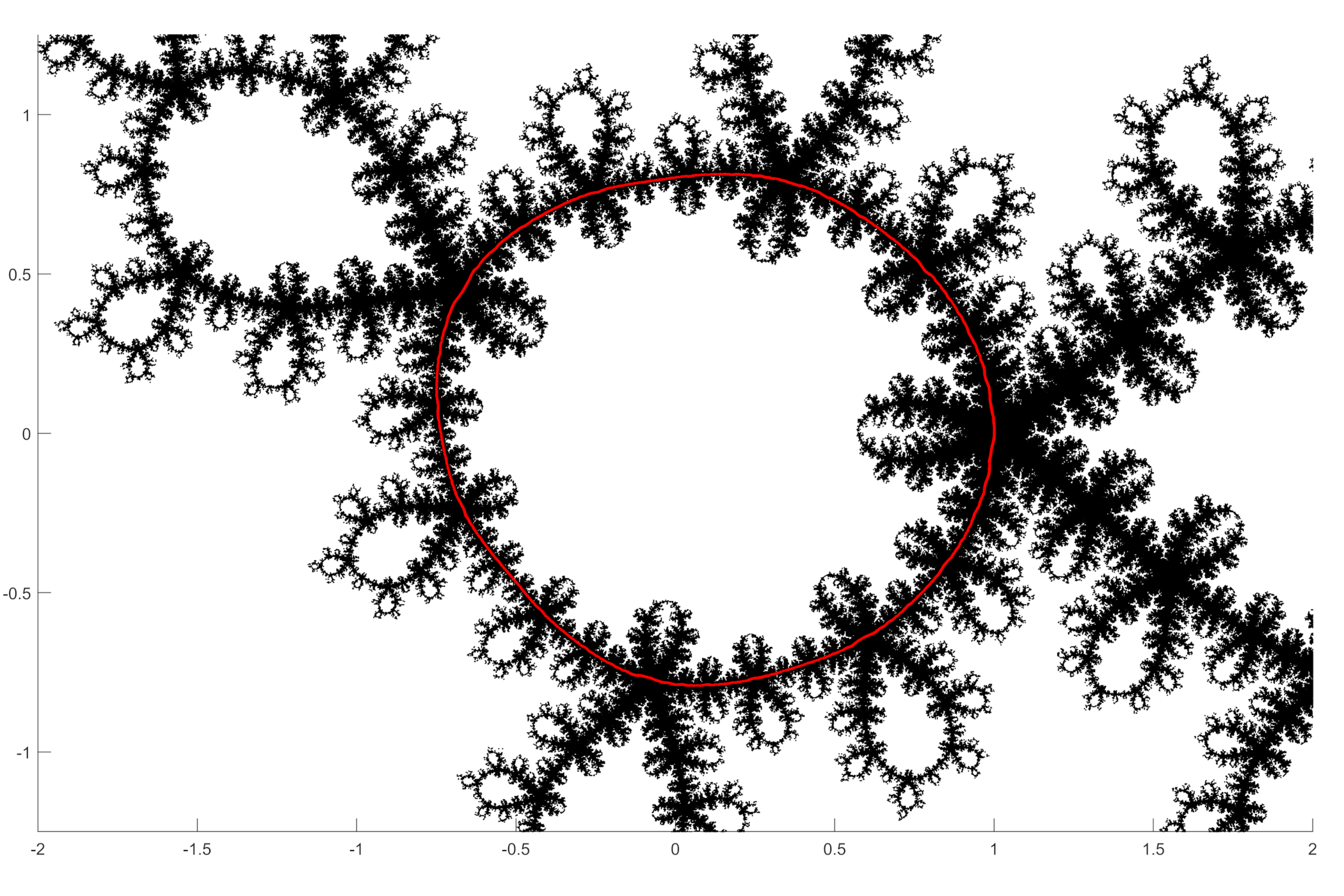}
    \caption{A $C^1$ smooth $(2,2)$-critical Herman curve $\Hq$ within the Julia set of the cubic rational map $f(z)=cz\frac{z^2-3z+\frac{\lambda}{c}}{1+\left(\frac{\lambda}{c}-3\right)z}$. The map $f$ is a perturbation of a Blaschke product of the form $cz^2\frac{z-3}{1-3z}$. It is characterized by fixed points at $0$ and $\infty$ with multipliers $\lambda=0.9 e^{2\pi i \theta}$ and $0$ respectively, as well as a critical point at $1$ with local degree $3$. The critical value $c=f(1)\approx -0.507844-0.457336i$ is picked such that $\Hq$ exists with golden mean rotation number $\theta$.}
    \label{fig:smoothness}
\end{figure}

\begin{proof}
    Suppose $d_0\neq d_\infty$. If $\Hq$ were $C^1$ smooth near the critical value, then the angle of $\Hq$ at the critical point is equal to $\frac{\pi(2d_0-1)}{d_0+d_\infty-1}$, which is not equal to $\pi$, and so it cannot have a tangent.
    
    Now, suppose $d_0=d_\infty$. Proposition \ref{prop:prototype-example} gives us a $(d_0,d_\infty)$-critical circle map $g: \T \to \T$ with rotation number $\theta$. By $C^{1+\alpha}$ rigidity, there exists a uniformly $C^{1+\alpha}$-conformal conjugacy $\phi: \T \to \Hq$ between $g$ and $f$. Lemma \ref{lem:holder} below implies that $\phi$ is indeed a continuously differentiable parametrization of $\Hq$.
\end{proof}

\begin{lemma}
\label{lem:holder}
    The complex derivative of a uniformly $C^{1+\alpha}$-conformal map is H\"older continuous with exponent $\alpha$.
\end{lemma}

\begin{proof}
    Suppose $\phi: U \to V$ is uniformly $C^{1+\alpha}$-conformal, that is, for all $z \in U$, the complex derivative $\phi'(z)$ exists and the function
    \[
        \eta_z(t) := \frac{\phi(z+t)-\phi(z)}{t} - \phi'(z), 
    \]
    satisfies $|\eta_z(t)| \leq C \,|t|^\alpha$ for some uniform constant $C>0$ when $|t|$ is sufficiently small. Whenever two points $z,w$ on $U$ are sufficiently close,
    \[
        |\phi'(z)-\phi'(w)| = |\eta_w(z-w) - \eta_z(w-z)| \leq 2C \, |z-w|^\alpha.
    \]
    This proves the lemma.
\end{proof}

Let us complete the proof of Corollary \ref{cor:smoothness} with the following proposition.

\begin{proposition}
\label{prop:HD>1}
    If $d_0 \neq d_\infty$, the Hausdorff dimension of $\Hq$ is greater than some constant $D>1$ which depends only on $d_0$, $d_\infty$, $N$, and the dilatation of $\Hq$.
\end{proposition}

\begin{proof}
    Suppose $d_0 \neq d_\infty$. Given a point $x \in \Hq$ and a small scale $r>0$, we define
    \[
        \beta_\Hq(x,r) := \frac{1}{r} \inf_{L \in \mathcal{L}(x,r)} \sup_{z \in \D(x,r) \cap \Hq} \dist(z,L),
    \]
    where $\mathcal{L}(x,r)$ denotes the set of lines in $\C$ intersecting the disk $\D(x,r)$, as well as
    \[
        \beta_\Hq(x) := \liminf_{r \to 0} \beta_\Hq(x,r).
    \]
    The quantity $\beta_\Hq(x,r)$ measures how far $\Hq$ is from being a line segment near $x$ at scale $r$. Due to a result by Bishop and Jones \cite{BJ97}, it is sufficient for us to show that $\beta_\Hq(x)$ is uniformly bounded below by some positive constant depending on $d_0$, $d_\infty$, $N$, and the dilatation $K$ of $\Hq$. 
    
    In the proof below, we will first show that the beta number at the critical point $c$ is positive. We then transfer this property around $c$ to every non-critical point via Koebe distortion. To do this, we will use the \emph{bounded turning} characterization of quasicircles, that is, for any two distinct points $a$ and $b$ on $\Hq$, the ratio of $|a-b|$ to the diameter of the interval $[a,b] \subset \Hq$ connecting $a$ and $b$ of the smallest diameter is bounded below by some positive constant depending on $K$.

    \begin{claim}
        There is some $\beta_0=\beta_0(d_0,d_\infty,K)>0$ such that $\beta_\Hq(c)\geq \beta_0$.
    \end{claim}

    \begin{proof}
        Suppose instead that $\beta_\Hq(c) < \beta_0$ where $\beta_0$ is a small constant that is to be determined. Then, there exist sequences of positive real numbers $\{r_n\}$ and $\{s_n\}$ such that $r_n \to 0$ and $s_n \to \beta_0$ as $n \to \infty$, and the intersection of $\Hq$ and the disk $D_n :=\D(c,r_n)$ is contained in the $r_n s_n$-neighborhood $S_n$ of a straight line. See Figure \ref{fig:thin-strip}.

\begin{figure}
    \centering
    \begin{tikzpicture}[scale=0.78]
        \draw[white,fill=yellow!50!white,opacity=0.4] (-6,-0.7) -- (6,-0.7) -- (6,0.4) -- (-6,0.4) -- (-6,-0.7);

        \draw[yellow!60!white] (-6,0.4) -- (6,0.4);        
        \draw[yellow!60!white] (-6,-0.7) -- (6,-0.7);
        \filldraw[green,fill=green!50!white,opacity=0.2] (0,0) circle (3); 
        
        \draw[ultra thick] (-6,1.167) -- (-1.6,-0.667) -- (0,0) -- (1.2,-0.5) -- (3.6,0.5) -- (5,-0.25) -- (6,-1.25);

        % labels
        \node [black, font=\bfseries] at (-5.7,1.35) {$\Hq$};
        \node [red, font=\bfseries] at (0,0) {$\bullet$};
        \node [red, font=\bfseries] at (0,0.32) {$c$};
        \node [red, font=\bfseries] at (2.99,0.23) {$\bullet$};
        \node [red, font=\bfseries] at (-3,-0.1) {$\bullet$};
        \node [red, font=\bfseries] at (3.1,-0.05) {$z_2$};
        \node [red, font=\bfseries] at (-3.05,-0.4) {$z_1$};
        \node [green!80!black, font=\bfseries] at (0,2) {$D_n$};
        \node [yellow!60!black, font=\bfseries] at (-5,-0.2) {$S_n$};
        
        \draw[gray, line width=0.5pt,latex'-latex'] (0,-0.05) -- (0,-3);
        \draw[gray,line width=0.5pt,latex'-latex'] (-6.6,0.4) -- (-6.6,-0.7);
        \node [gray,font=\bfseries] at (-0.3,-1.5) {$r_n$};
        \node [gray,font=\bfseries] at (-7.3,-0.1) {$2r_n s_n$};
        \node [white,font=\bfseries] at (7.7,-0.4) {\footnotesize$.$};
\end{tikzpicture}

    \caption{$\Hq \cap D_n$ is contained in the strip $S_n$.}
    \label{fig:thin-strip}
\end{figure}
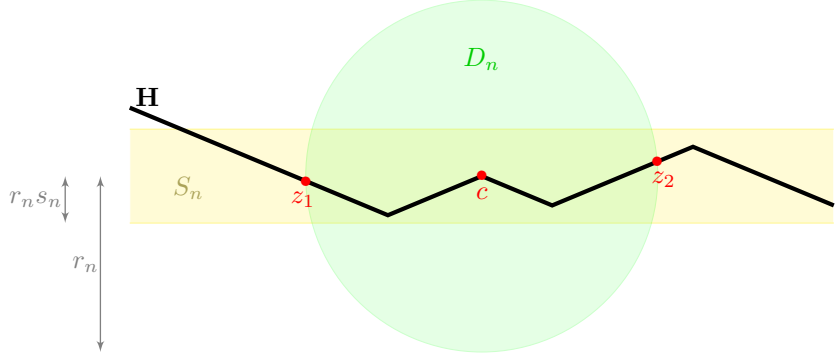

        Let us label the two connected components of $\partial D_n \cap S_n$ by $X_1$ and $X_2$. For $i \in \{1,2\}$, the diameter of $X_i$ is between $2 r_n s_n$ and $2\sqrt{2} \, r_n s_n$ by elementary trigonometry. Let $\Gamma=[z_1,z_2]$ denote the closed interval that is the connected component of $\overline{D_n} \cap \Hq$ containing $c$. We claim that $z_1 \in X_1$ and $z_2 \in X_2$. Indeed, if otherwise,
        \[
        \frac{|z_1-z_2|}{\diam(\Gamma)} \leq \frac{\max_i\diam(X_i)}{\diam([c,z_1])} = O\left(s_n\right),
        \]
        and as $n \to \infty$, the right hand side becomes small depending on $\beta_0$. When $\beta_0$ is sufficiently small depending on $K$, this estimate would contradict the bounded turning property.
        
        Let $d:= d_0 + d_\infty -1$, the local degree of $f$ at $c$. Suppose for a contradiction that $d$ is even. Near $c$, the map $f$ is close to the power map 
        \[
        g(z) := f(c) + \frac{f^{(d)}(c)}{d!}(z-c)^d,
        \]
        that is, $|f(z)-g(z)| = O(|z-c|^{d+1})$. Since $d$ is even, $g$ will send both $z_1$ and $z_2$ to points that are very close to one another and
        \begin{align*}
            |f(z_1)-f(z_2)| &\leq |g(z_1)-g(z_2)| + O(|z_1-c|^{d+1}) + O(|z_2-c|^{d+1}) \\
            &= O\left(r_n^d (s_n^d + r_n)\right).
        \end{align*}
        However, since the interval $f(\Gamma)=[f(z_1),f(z_2)] \subset \Hq$ has diameter at least $\asymp r_n^d$,
        \[
            \frac{|f(z_1)-f(z_2)|}{\diam(f(\Gamma))} = O(s_n^d + r_n).
        \]
        Again, as $n \to \infty,$ this estimate would contradict the bounded turning property provided that $\beta_0$ is sufficiently small depending on $d_0$, $d_\infty$, and $K$. Therefore, $d$ must be odd.

        Since $d$ is odd, the image $S'_n := f(S_n \cap D_n)$ is close to being a straight strip of width $\asymp r_n^d s_n$ inside of $f(D_n)$, which is close to a round disk of radius $\asymp r_n^d$. 
        Again, denote by $Y^0$ and $Y^\infty$ the inner and outer components of $\RS \backslash \Hq$. 
        We claim that the two connected components of $f(D_n) \backslash S'_n$ belong to different components of $\RS \backslash \Hq$, which we will denote by $B^0 \subset Y^0$ and $B^\infty \subset Y^\infty$. 
        Indeed, suppose instead that both are contained in $Y^\infty$ without loss of generality. There is some $i \in \{1,2\}$ such that $f(X_i) \cup Y^\infty$ contains a Jordan curve enclosing $f(X_j)$ where $j \in \{1,2\}\backslash\{i\}$. However, this would imply the existence of a closed interval in $\Hq$ having endpoints in $f(X_i)$ and diameter $\succ r_n^d$, which would again contradict the bounded turning property.

        For $\bullet \in \{0, \infty\}$, the number of components of $f^{-1}(B^\bullet)$ contained in $Y^\bullet$ is $d_\bullet$. Since $f$ is close to the $d$\textsuperscript{th} power map $g$ and $S_n'$ is close to being a thin straight strip, then $d_0 = d_\infty$. This yields a contradiction.
        \end{proof}

        Let us pick any point $x$ on $\Hq$ and any sufficiently small scale $r>0$. By Lemma \ref{approx-rot-02}, there is an approximate rotation $f^i : (U,y) \to (V,c)$ such that $y$ lies on $\Hq$, $c$ is a critical point of $f$, and $(U,y)$ is a pointed disk that well contains the interval $[x,y] \subset \Hq$ and has bounded shape and inner radius $\rin(U,y) \asymp r$. (Note that, from this moment on, implicit constants may depend on $N$.) 

        Let us denote by $\delta$ the inner radius of $(V,c)$. Consider a small disk $B_\varepsilon = \D(y, \varepsilon r)$ where $0<\varepsilon<1$. Since $f^i$ has bounded distortion on $U$, the image $f^i(B_\varepsilon)$ will have bounded shape and diameter $\asymp \varepsilon \delta$. 
        
        The claim implies that there is some constant $C=C(d_0,d_\infty,K)>0$ such that for sufficiently small $r$ (and thus $\delta$), we can find an interval $[a',b'] \subset \Hq$ contained in $f^i(B_\varepsilon)$ and a point $w'$ on $\Hq$ such that the distance between $w'$ and the unique straight line $L_{a',b'}$ passing through $a'$ and $b'$ is at least $C \varepsilon \delta$. Denote by $a$, $b$, and $w$ the lift of $a'$, $b'$, and $w'$ under $f^i|_U$.
        
        Since $f^i$ has bounded distortion on $U$, the Euclidean triangle with vertices $a', b', w'$ should be almost similar to that with vertices $a,b,w$. More precisely, there is some constant $M_\varepsilon > 0$ which shrinks to $0$ as $\varepsilon \to 0$ such that
        \[
            \left| \frac{a-w}{b-w} \Bigg/ \frac{a'-w'}{b'-w'} -1 \right| \leq M_\varepsilon.
        \]
        Therefore, we can pick $\varepsilon$ depending on $C$ such that the distance between $w$ and the unique straight line $L_{a,b}$ passing through $a$ and $b$ satisfies $ \dist(w, L_{a,b}) \succ C \varepsilon r$. Together with the bounded turning property, this implies that $\beta_\Hq(x,r) \geq \beta$ for some $\beta=\beta(d_0,d_\infty,K,N)>0$.
\end{proof}

\subsection{Universality}
\label{ss:universality}

Let $f_1: \Hq_1 \to \Hq_1$ and $f_2: \Hq_2 \to \Hq_2$ be two $(d_0,d_\infty)$-critical quasicircle maps of the same rotation number $\theta$. By $C^{1+\alpha}$ rigidity, there exists a uniformly $C^{1+\alpha}$-conformal conjugacy $\phi: \Hq_1 \to \Hq_2$ between $f_1$ and $f_2$. Corollaries \ref{cor:dimension} and \ref{cor:universality} will follow from below.

\begin{corollary}
    Quasicircles $\Hq_1$ and $\Hq_2$ have the same Hausdorff dimension, lower box dimension, and upper box dimension.
\end{corollary}

\begin{proof}
    By Lemma \ref{lem:holder}, the complex derivative $\phi'$ is continuous. Since quasicircles are compact, the map $\phi$ must be bi-Lipschitz. Since dimension is bi-Lipschitz invariant, the claim follows.
\end{proof}

We say that a sequence $\{a_n\}_{n\in\N}$ of complex numbers \emph{converges exponentially fast} to $a$ if there are constants $C>0$, $n_0 \in \N$, and $\lambda \in (0,1)$ such that for all $n\geq n_0$, $|a_n-a| \leq C \lambda^n$. Given some data $\heartsuit$, we also say that $\{a_n\}$ converges $\heartsuit$-exponentially fast to $a$ if the constants $C,n_0,\lambda$ depend only on $\heartsuit$.

Recall the notion of scaling ratios in (\ref{eqn:scaling-ratio}).

\begin{corollary}
\label{cor:universality-scaling-ratio}
    Asymptotically, $f_1$ and $f_2$ have the same scaling ratios:
    \[
        \frac{s_n(f_2)}{s_n(f_1)} \xrightarrow{} 1 \quad \text{exponentially fast as } n \to \infty.
    \]
\end{corollary}
    
\begin{proof}
    Assume without loss of generality that $0$ is the critical point of both $f_1$ and $f_2$, and let $c_n:= f^n_1(0)$. By Lemma \ref{lem:qc-control} (2), the bounded type assumption implies that $|c_{q_n}| \asymp \delta^n$ for some $\delta \in (0,1)$.
    
    Near $0$, we can write $\phi(z) = z \left(\lambda + \eta(z)\right)$ where $\lambda \in \C^*$ and $\eta(z) = O(|z|^{\alpha})$. Then,
    \begin{align*}
        \frac{s_n(f_2)}{s_n(f_1)} -1 = \frac{\phi(c_{q_{n+1}})}{c_{q_{n+1}}} \cdot \frac{c_{q_{n}}}{\phi(c_{q_{n}})} -1 = \cfrac{ \eta(c_{q_{n+1}}) - \eta(c_{q_{n}}) }{ \lambda+\eta(c_{q_{n}}) } = O(\delta^{\alpha n}).
    \end{align*}
    Therefore, the ratio $s_n(f_2)/s_n(f_1)$ tends to $1$ exponentially fast.
\end{proof}

\subsection{Exponential convergence of renormalizations}
\label{ss:exp-convergence}

Let us fix $K >1$ and $\mu >0$. For brevity, we will denote by $\clubsuit$ the data $(d_0,d_\infty,N,K,\mu)$.
    
\begin{theorem}[Exponential convergence]
    \label{thm:exponential-convergence}
    Given any two commuting pairs $\zeta$ and $\hat{\zeta}$ in $\mathcal{CP}(\clubsuit)$ with the same rotation number, their renormalizations converge together exponentially fast in the following sense. Let us denote the $n$\textsuperscript{th} renormalization 
    of $\zeta$ and $\hat{\zeta}$ by $\renorm^n \zeta = (f_{n,-}|_{I_{n,-}}, f_{n,+}|_{I_{n,+}})$ and $\renorm^n \hat{\zeta} = \left(\hat{f}_{n,-}|_{\hat{I}_{n,-}}, \hat{f}_{n,+}|_{\hat{I}_{n,+}}\right)$ respectively. Then,
    \begin{enumerate}[label=\textnormal{(\arabic*)}]
        \item the Hausdorff distance between $I_{n,\pm}$ and $\hat{I}_{n,\pm}$ tends to $0$ $\clubsuit$-exponentially fast;
        \item for sufficiently large $n$ depending on $\clubsuit$, both $f_{n,\pm}$ and $\hat{f}_{n,\pm}$ extend holomorphically to the $\varepsilon(\clubsuit)$-neighborhood of $I_{n,\pm} \cup \hat{I}_{n,\pm}$ on which the sup norm of $f_{n,\pm}-\hat{f}_{n,\pm}$ converges $\clubsuit$-exponentially fast to $0$.
    \end{enumerate}
\end{theorem}

\begin{proof}
    Let $\psi$ be a quasiconformal conjugacy between $\zeta$ and $\hat{\zeta}$. Recall that the renormalization $\renorm^n \zeta$ is obtained by conjugating $p \renorm^n \zeta$ with the map
        \[
            \tau_{n}(z) := \begin{cases}
                -c_n z, & \text{ if } n \text{ is odd,}\\
                -c_n \bar{z}, & \text{ if } n \text{ is even,}
            \end{cases}
        \]
    sending $-1$ to an endpoint $c_n$ of the base of $p\renorm^n\zeta$. Similarly, denote by $\hat{\tau}_n(z)$ the corresponding rescaling map for $p\renorm^n \hat{\zeta}$ with scaling factor $-\hat{c}_n$ where $\hat{c}_n := \psi(c_n)$.
    
    By Theorem \ref{thm:complex-bounds-butterflies}, there are constants $n_1 = n_1(\clubsuit) \in \N$ and $L=L(\clubsuit) >1$ such that for all $n \geq n_1$, the pre-renormalization $p\renorm^n \zeta$ extends to an $L$-butterfly with range $V_n$. Denote the range of the corresponding butterfly for $\renorm^n \zeta$ by $\ranc_n = \tau_n^{-1}(V_n)$. By \ref{cond:4}, 
    \begin{equation}
    \label{eqn:diameter}
        \diam(\ranc_n) \leq L \quad \text{and} \quad \diam(V_n) \leq |c_n| L.
    \end{equation}

    Since $\psi$ is $C^{1+\alpha}$-conformal at the critical point $0$, there exist positive constants $\alpha, C, \lambda, r$ depending only on $\clubsuit$ such that for $|z|<r$,
        \begin{equation}
            \label{eqn:c1plusalpha}
            \psi(z)=\lambda z+\eta(z), \quad \text{ where } \quad |\eta(z)|\leq C |z|^{1+\alpha}.
        \end{equation}        
    The sequence $\{c_n\}_{n \in \N}$ converges $\clubsuit$-exponentially fast to $0$ due to Lemma \ref{lem:qc-control} (1). By (\ref{eqn:diameter}), there is some $n_2 = n_2(\clubsuit) \in \N$ such that for $n \geq n_2$, the disk $V_n$ has diameter at most $r$.
    By (\ref{eqn:c1plusalpha}), for $n \geq n_2$, 
    \begin{equation}
    \label{eqn:rescaled-error-term}
        \sup_{z \in \ranc_n} \left| \frac{\eta(-c_n z)}{\hat{c}_n}\right| \leq C  \frac{|c_n|^{1+\alpha}}{|\hat{c}_n|} \sup_{z \in \ranc_n}|z|^{1+\alpha} \leq C L^{1+\alpha}  \left|\frac{c_n}{\hat{c}_n}\right| |c_n|^\alpha.
    \end{equation}
    Since $\psi(c_n) = \hat{c}_n$, then by (\ref{eqn:c1plusalpha}) again, the sequence of ratios $\{\hat{c}_n/c_n\}_{n\in\N}$ converges $\clubsuit$-exponentially fast to the derivative $\lambda$ of $\psi$ at $0$. Therefore, the sequence 
        \[
        \sup_{z \in \ranc_n} \left|\frac{\eta(-c_n z)}{\hat{c}_n} \right|
        \]
    also converges $\clubsuit$-exponentially fast to $0$.
        
    The map $\psi_n:=\hat{\tau}_n^{-1}\circ \psi \circ \tau_n$ conjugates $\renorm^n \zeta_1$ and $\renorm^n \zeta_2$. For all even $n \geq n_2$ and $z \in \ranc_n$,
    \begin{align*}
            |\psi_n(z) - z| &\leq \left| \lambda \frac{c_n}{\hat{c}_n} - 1 \right| |z| + \left| \frac{\eta(-c_n z)}{\hat{c}_n}\right|,
    \end{align*}
    and a similar estimate holds for odd $n$. This implies that the sup norm of $\psi_n - \text{Id}$ on $\ranc_n$ converges $\clubsuit$-exponentially fast to $0$, and items (1) and (2) follow immediately.
\end{proof}

\subsection{A horseshoe attractor}
\label{ss:horseshoe}

Let $\mathcal{CP}(d_0,d_\infty,N)$ be the space of all normalized $(d_0,d_\infty)$-critical commuting pairs $\zeta$ with rotation number in $\Theta_N$. Denote by $\sigma$ the shift map acting on the bi-infinite shift space $\Sigma_N = \{1,\ldots,N\}^\Z$ of $N$ symbols, equipped with the infinite product topology. Consider the continuous surjection 
\[
    \xi: \Sigma_N \to \Theta_N, \quad (\ldots,a_{-2},a_{-1},a_0;a_1,a_2,\ldots) \mapsto [0;a_1,a_2,\ldots].
\]
We now prove a more precise formulation of Theorem \ref{horseshoe}.

\begin{theorem}[Renormalization horseshoe]
\label{thm:horseshoe}
    There is a unique renormalization-invariant compact subset $\mathcal{A}=\mathcal{A}(d_0,d_\infty,N)$ of $\mathcal{CP}(d_0,d_\infty,N)$ satisfying the following properties.
    \begin{enumerate}[label=\textnormal{(\arabic*)}]
        \item There is a topological conjugacy $\Phi: \Sigma_N \to \mathcal{A}$ between the renormalization operator $\renorm: \mathcal{A} \to \mathcal{A}$ and the shift map $\sigma: \Sigma_N \to \Sigma_N$ such that $\rot \circ \Phi = \xi$.
        \item For any $\zeta \in \mathcal{A}$ and $\zeta' \in \mathcal{CP}(d_0,d_\infty,N)$, renormalizations $\renorm^n \zeta$ and $\renorm^n \zeta'$ converge together exponentially fast if and only if $\renorm^m \zeta$ has the same rotation number as $\renorm^m \zeta'$ for some $m \in \N$.
    \end{enumerate}
\end{theorem}

In the proof below, we obtain the horseshoe by constructing limits of renormalization towers, and we deduce the rigidity of towers by applying the exponential convergence of renormalizations.

\begin{proof}
    Consider a bi-infinite sequence $\mathbf{a}= (\ldots, a_{-2},a_{-1},a_0;a_1,a_2,\ldots)$ in $\Sigma_N$. For any $k \in \Z$, set $\theta_k := [0;a_{k+1},a_{k+2},a_{k+3},\ldots]$ and let $f_{k}$ be the rational map $F_{d_0,d_\infty,\theta_{k}}$ from Proposition \ref{prop:prototype-example}. By Lemma \ref{rotation-number}, whenever $k+ l \geq 1$, the $(k+l)$\textsuperscript{th} renormalization $\zeta_{k,l}:=\renorm^{k+l} f_{-k}$ has rotation number $\theta_{l}$.

    According to Theorem \ref{thm:complex-bounds-butterflies}, if $k+l$ is sufficiently high, then $\zeta_{k,l}$ always extends to a $K$-butterfly for some $K=K(d_0,d_\infty,N) >1$. By Proposition \ref{compactness}, for any $l \in \Z$, there is a subsequence $\{\zeta_{k(i),l}\}_{i\in \N}$ such that as $i \to \infty$, then $k(i)\to \infty$ and $\zeta_{k(i),l}$ converges to some $(d_0,d_\infty)$-critical commuting pair $\zeta_l$ of rotation number $\theta_{l}$. By a diagonal procedure, we can ensure that $\renorm \zeta_l = \zeta_{l+1}$, giving us a bi-infinite renormalization tower
    \[
        \mathcal{T}_{\mathbf{a}} := (\ldots, \zeta_{-2}, \zeta_{-1}, \zeta_0; \: \zeta_1, \zeta_2, \ldots)
    \]
    of commuting pairs in $\mathcal{CP}(d_0,d_\infty,N)$ where each entry $\zeta_n$ extends to a $K$-butterfly.
    
    Suppose this procedure yields another renormalization tower
    \[
        \mathcal{T}'_{\mathbf{a}} := (\ldots, \zeta'_{-2}, \zeta'_{-1}, \zeta'_0; \: \zeta'_1, \zeta'_2, \ldots).
    \]
    By Theorem \ref{thm:exponential-convergence}, as $n \to \infty$, renormalizations $\renorm^k \zeta_n$ and $\renorm^k \zeta'_n$ converge together exponentially fast at a uniform rate independent of $k$. Since $\renorm^k\zeta_n = \zeta_{n+k}$ and $\renorm^k\zeta'_n = \zeta'_{n+k}$, this implies that $\zeta'_n = \zeta_n$ for all $n \in \Z$. Therefore, the tower $\mathcal{T}_{\mathbf{a}}$ is uniquely defined. 
            
    Set $\Phi(\mathbf{a})$ to be the zeroth entry $\zeta_0$ of $\mathcal{T}_{\mathbf{a}}$ and set 
    \[
        \mathcal{A}:= \{ \Phi(\mathbf{a}) \in \mathcal{CP}(d_0,d_\infty,N) \: : \: \mathbf{a} \in \Sigma_N \}.
        \]
    We have a surjective mapping $\Phi: \Sigma_N \to \mathcal{A}$ satisfying $\rot \circ \Phi = \xi$. Since $\Sigma_N$ is compact and Hausdorff, in order to prove (1), it remains to show that $\Phi$ is injective and continuous.

    Suppose $\Phi(\mathbf{a}_1) = \Phi(\mathbf{a}_2)$ for some $\mathbf{a}_1, \mathbf{a}_2 \in \Sigma_N$. By the identity theorem, the associated towers are equal, namely $\mathcal{T}_{\mathbf{a}_1} = \mathcal{T}_{\mathbf{a}_2}$. For each $i \in \{1,2\}$, we can recover back the bi-infinite sequence $\mathbf{a}_i$ by evaluating the rotation number of each entry of $\mathcal{T}_{\mathbf{a}_i}$, so then $\mathbf{a}_1=\mathbf{a}_2$. Therefore, $\Phi$ is injective.

    Let us show that $\Phi$ is continuous. Suppose a sequence of elements $\mathbf{a}^{(n)}$ in $\Sigma_N$ converges to $\mathbf{a}$. For each $n \in \N$, denote the associated tower by 
    \[
        \mathcal{T}_{\mathbf{a}^{(n)}} = (\ldots, \zeta_{-2}^{(n)}, \zeta_{-1}^{(n)}, \zeta_{0}^{(n)}; \:\zeta_{1}^{(n)}, \zeta_{2}^{(n)}, \ldots).
    \]
    By passing to a subsequence, each $\zeta_k^{(n)}$ converges to some commuting pair $\zeta_k$ as $n \to \infty$, forming a limiting renormalization tower
    \[
        \mathcal{T} = (\ldots, \zeta_{-2}, \zeta_{-1}, \zeta_0; \:\zeta_1, \zeta_2, \ldots).
    \]
    For every $k \in \Z$ and $n \in \N$, Corollary \ref{qs-conjugacy} states that there is a unique quasisymmetric conjugacy $h_k^{(n)}$ between the pair of translations $\mathbf{T}_{\theta_k^{(n)}}$ and $\zeta_k^{(n)}$. Clearly, as $\mathbf{a}^{(n)} \to \mathbf{a}$, the pair $\mathbf{T}_{\theta_k^{(n)}}$ converges to $\mathbf{T}_{\theta_k}$. Since the dilatation of $h_k^{(n)}$ is uniform, the sequence $h_k^{(n)}$ subsequently converges to a quasisymmetric map $h_k$, which conjugates $\mathbf{T}_{\theta_k}$ with $\zeta_{k}$. In particular, $\rot(\zeta_{k}) = \theta_k$. By the uniqueness of renormalization towers, $\mathcal{T}$ coincides with $\mathcal{T}_{\mathbf{a}}$. Therefore, $\Phi\left(\mathbf{a}^{(n)}\right) \to \Phi(\mathbf{a})$. 

    Let us now prove property (2). Pick $\zeta \in \mathcal{A}$ and $\zeta' \in \mathcal{CP}(d_0,d_\infty,N)$. If $\renorm^m \zeta$ and $\renorm^m \zeta'$ have the same rotation number for some $m \in \N$, then again by Theorem \ref{thm:exponential-convergence}, $\renorm^n \zeta$ and $\renorm^n \zeta'$ converge together exponentially fast. Otherwise, there is an infinite sequence $\{ n_k \}_{k \in \N}$ such that the continued fraction expansions of the rotation numbers of $\renorm^{n_k} \zeta$ and $\renorm^{n_k} \zeta'$ have different first term, which clearly implies that $\renorm^{n_k} \zeta$ and $\renorm^{n_k} \zeta'$ cannot converge together.
    \end{proof}

\subsection{Self-similarity}
\label{ss:self-similarity}

Consider a bi-infinite sequence 
\[
    \mathbf{a} = (\ldots, a_{-2}, a_{-1}, a_0; a_1, a_2, \ldots \ldots) \in \Sigma_N
\]
such that $\mathbf{a}$ is $s$-periodic for some $s\geq 1$, i.e. $\mathbf{a} = \sigma^s \mathbf{a}$. To lighten our notation, we assume that $s$ is even. (Else, replace $s$ by $2s$.) Let 
\[
\theta_0 := [0; a_1,a_2,a_3,\ldots].
\]
Then, $\theta_0$ is also $s$-periodic under the Gauss map $G(x) = \left\{\frac{1}{x}\right\}$.

We say that two subsets $P$ and $Q$ of $\C$ are \emph{linearly equivalent} if there is a linear map $g$ such that $g(P)=Q$. Below is a more precise version of Theorem \ref{thm:self-similarity}.

\begin{theorem}[Self-similarity]
    \label{thm:A-invariant-quasicircle}
    There exists a complex number $\mu_{\mathbf{a}} \in \D^*$ such that the following holds. Let $A_{\mathbf{a}}(z) := \mu_{\mathbf{a}} z$.
    \begin{enumerate}[label=\textnormal{(\arabic*)}]
        \item Consider the conjugacy $\Phi : \Sigma_N \to \mathcal{A}$ described in Theorem \ref{thm:horseshoe}. The base of $\Phi(\mathbf{a})$ extends to a unique $A_{\mathbf{a}}$-invariant quasicircle $\Hq_\mathbf{a}$.
        \item $\mu_{\sigma^2\mathbf{a}} = \mu_{\mathbf{a}}$, and $\Hq_{\sigma^2\mathbf{a}}$ is linearly equivalent to $\Hq_{\mathbf{a}}$.
        \item Suppose $f:\Hq \to \Hq$ is a $(d_0,d_\infty)$-critical quasicircle map of rotation number $\theta$ where $G^k(\theta)=\theta_0$ for some even integer $k \geq 0$. Assume $0$ is the critical point of $f$.
        \begin{enumerate}[label=\textnormal{(\alph*)}]
            \item $A_{\mathbf{a}}^{-n}(\Hq)$ converges in the Hausdorff metric to an $A_{\mathbf{a}}$-invariant quasicircle linearly equivalent to $\Hq_{\mathbf{a}}$.
            \item Let $c_{l} := f^l(0)$ for all $l \in \N$. Then, 
                \[
                \frac{c_{q_{n+s}}}{c_{q_n}} = \prod_{i=1}^s s_{n+i}(f) \xrightarrow{} \mu_{\mathbf{a}} \quad \text{ exponentially fast as } n \to \infty.
                \]
        \end{enumerate}
    \end{enumerate}
\end{theorem}

\begin{proof}
    Since $\zeta := \Phi(\mathbf{a})$ satisfies the equation $\renorm^s \zeta = \zeta$, there is a linear map $A_{\mathbf{a}}(z) := \mu_{\mathbf{a}} z$ such that $0<|\mu_{\mathbf{a}}|<1$ and $p\renorm^s \zeta = A_{\mathbf{a}} \circ \zeta \circ A_{\mathbf{a}}^{-1}$. This immediately implies (1). Also, we have $p\renorm^{s+2} \zeta = A_{\mathbf{a}} \circ p\renorm^2 \zeta \circ A_{\mathbf{a}}^{-1}$, which after normalization yields the equation $p\renorm^s (\renorm^2 \zeta) = A_{\mathbf{a}} \circ \renorm^2 \zeta \circ A_{\mathbf{a}}^{-1}$. This implies (2).

    Suppose $f: \Hq \to \Hq$ satisfies the hypothesis in (3). By Theorem \ref{thm:c1plusalpharigidity}, the quasisymmetric conjugacy $\psi$ between $\renorm^k f$ and $\zeta$ extends to a quasiconformal conjugacy that is $C^{1+\alpha}$ near the critical point. Let $\phi(z) := \psi(z)/\psi'(0)$, which conjugates $\renorm^k f$ with a rescaling of $\zeta$, and satisfies
    \[
        \phi(z) = z\left(1 + \eta(z)\right) \quad \text{where} \quad \eta(z) = O( |z|^{\alpha}).
    \]
    Denote by $\hat{\Hq}_{\mathbf{a}}$ the rescaling of $\Hq_{\mathbf{a}}$ by $\psi'(0)^{-1}$.
    
    For any sufficiently large $n \in \N$, there is some $l \geq k$ such that the base $\I \subset \Hq$ of $p\renorm^l f$ has diameter $\asymp |\mu_{\mathbf{a}}|^n$. Denote by $d_H(\cdot, \cdot)$ the Hausdorff distance. Then,
    \begin{align*}
        d_H\left( A_{\mathbf{a}}^{-n}(\I), A_{\mathbf{a}}^{-n} \phi(\I) \right) \leq |\mu_{\mathbf{a}}|^{-n} d_H\left(\I, \phi(\I)\right) \prec |\mu_{\mathbf{a}}|^{-n} \diam(\I)^{1+\alpha} \prec |\mu_{\mathbf{a}}|^{\alpha n}.
    \end{align*}
    As $n \to \infty$, the Hausdorff distance between $A_{\mathbf{a}}^{-n}(\I)$ and $A_{\mathbf{a}}^{-n} \phi(\I)$, which is a subinterval of $\hat{\Hq}_{\mathbf{a}}$, shrinks exponentially fast.

    The proof of (b) is similar to Corollary \ref{cor:universality-scaling-ratio}. Let us write $\phi(z) = z\left(1 + \eta(z)\right)$ where $\eta(z) = O( |z|^{\alpha})$. By Lemma \ref{lem:qc-control} (2), there is some $\delta \in (0,1)$ such that $|c_{q_n}| \asymp \delta^n$. For $n \geq k$, we have the equation $\phi(c_{q_{n+s}}) = A_{\mathbf{a}} \phi(c_{q_n})$, which implies that as $n \to \infty$,
    \[
        \mu_{\mathbf{a}} - \frac{c_{q_{n+s}}}{c_{q_n}} = \frac{\phi(c_{q_{n+s}})}{\phi(c_{q_n})} - \frac{c_{q_{n+s}}}{c_{q_n}} = \frac{c_{q_{n+s}}}{c_{q_n}} \cdot \cfrac{ \eta(c_{q_{n+s}}) - \eta(c_{q_{n}}) }{ 1+\eta(c_{q_n}) }= O(\delta^{\alpha n}).
    \]
    Therefore, the ratio $c_{q_{n+s}}/c_{q_n}$ tends to $\mu_{\mathbf{a}}$ exponentially fast.
\end{proof}

\bibliographystyle{alpha}
 
\bibliography{bibliography}

\end{document}